\documentclass[11pt]{article}
\usepackage{amssymb}
\usepackage{amsmath}
\usepackage{mathrsfs}
\usepackage{enumerate}
\usepackage{graphics}
\usepackage{graphicx}
\usepackage{dsfont}
\usepackage{subfigure}
\usepackage{latexsym,amssymb,amsmath,amsfonts,amsthm}\usepackage{txfonts}
\newcommand{\R}{{\mathbb R}}

\newcommand{\Sp}{{\mathbb S}}
\newcommand{\ds}{\displaystyle}

\newcommand{\be}{\begin{eqnarray}}
\newcommand{\ben}{\begin{eqnarray*}}
\newcommand{\en}{\end{eqnarray}}
\newcommand{\enn}{\end{eqnarray*}}
\newcommand{\ba}{\backslash}
\newcommand{\pa}{\partial}

\newcommand{\ov}{\overline}

\newcommand{\G}{\Gamma}

\newcommand{\Om}{\Omega}

\newcommand{\la}{\lambda}

\newcommand{\wid}{\widetilde}

\newcommand{\ol}{\overline}

\newcommand{\half}{\frac{1}{2}}

\newtheorem{theorem}{Theorem}[section]
\newtheorem{lemma}[theorem]{Lemma}

\newtheorem{remark}[theorem]{Remark}

\newtheorem{algorithm}{Algorithm}[section]

\begin{document}
\renewcommand{\theequation}{\arabic{section}.\arabic{equation}}
\begin{titlepage}
%
\title{\bf Fast imaging of scattering obstacles from phaseless far-field measurements at a fixed frequency}
\author{Bo Zhang\thanks{LSEC, NCMIS and Academy of Mathematics and Systems Science, Chinese Academy of Sciences,
Beijing, 100190, China and School of Mathematical Sciences, University of Chinese Academy of Sciences,
Beijing 100049, China ({\tt b.zhang@amt.ac.cn})}
\and Haiwen Zhang\thanks{NCMIS and Academy of Mathematics and Systems Science, Chinese Academy of Sciences,
Beijing 100190, China ({\tt zhanghaiwen@amss.ac.cn})}
}
\date{}
\end{titlepage}
\maketitle

\vspace{.2in}

\begin{abstract}
This paper is concerned with the inverse obstacle scattering problem with phaseless far-field data at a fixed
frequency. The main difficulty of this problem is the so-called translation invariance property of the modulus
of the far-field pattern or the phaseless far-field pattern generated by one plane wave as the incident field,
which means that the location of the obstacle can not be recovered from such phaseless far-field data at a
fixed frequency. It was recently proved in our previous work \cite{XZZ18} that the obstacle can be uniquely
determined by the phaseless far-field patterns generated by infinitely many sets of superpositions of two
plane waves with different directions at a fixed frequency if the obstacle is a priori known to be a
sound-soft or an impedance obstacle with
real-valued impedance function. The purpose of this paper is to develop a direct imaging algorithm
to reconstruct the location and shape of the obstacle from the phaseless far-field data corresponding to
infinitely many sets of superpositions of two plane waves with a fixed frequency as the incident fields.
Our imaging algorithm only involves the calculation of the products of the measurement data with
two exponential functions at each sampling point and is thus fast and easy to implement.
Further, the proposed imaging algorithm does not need to know the type of boundary conditions on the obstacle
in advance and is capable to reconstruct multiple obstacles with different boundary conditions.
Numerical experiments are also carried out to illustrate that our imaging method is stable, accurate
and robust to noise.

\vspace{.2in}
\end{abstract}

\section{Introduction}\label{sec1}

Problems of scattering of time-harmonic acoustic waves arise in many applications, such as radar and sonar,
remote sensing, geophysics, medical imaging and nondestructive testing. The direct scattering problem is
to determine the scattering solution, given the obstacle and its physical property,
while the inverse scattering problem is to determine the obstacle and/or its physical property
from the measurement information of the scattering solution.
Due to wide applications of direct and inverse scattering problems, such problems have been extensively
studied; see \cite{CK13} for the mathematical and numerical aspects of inverse scattering problems.

In this paper, we are concerned with the inverse problem of recovering scattering obstacles from phaseless
far-field data. For simplicity, we restrict our attention to the two-dimensional case.
Assume that $D\subset\R^2$ is a bounded domain with $C^2-$smooth boundary $\G:=\pa{D}$.
Denote by $\nu$ the unit outward normal on $\G$ to the domain $D$ and by $\mathbb{S}^1$ the unit circle.
Suppose a time-harmonic ($e^{-i\omega t}$ time dependence) plane wave
$$
u^i=u^i(x,d):=\exp(ikd\cdot x)
$$
is incident on the bounded obstacle $D$ from the unbounded part $\R^2\ba\ov{{D}}$,
where, $k=\omega/c>0$ is the wave number and $\omega$ and $c$ are the wave frequency and speed, respectively,
in $\R^2\ba\ov{{D}}$. Then the total field $u=u^i+u^s$, which is the sum of the incident field $u^i$
and the scattered field $u^s$, satisfies the Helmholtz equation in $\R^2\ba\ov{{D}}$:
\be\label{eq1}
\Delta u+k^2 u=0\quad\mbox{in}\;\;\R^2\ba\ov{D}.
\en

Moreover, a boundary condition is required, which depends on the physical property of the obstacle ${D}$.
When $D$ is an impenetrable, sound-soft, $u$ satisfies the Dirichlet boundary condition on $\G$:
\be\label{eq2}
u=0\quad\mbox{on}\;\;\G.
\en
When $D$ is an impenetrable, impedance obstacle, $u$ satisfies the impedance boundary condition on $\G$:
\be\label{eq3}
\frac{\pa u}{\pa\nu}+i k \rho u=0\quad \mbox{on}\;\;\G,
\en
where $\rho\in L^\infty(\G)$ is the impedance function on the boundary $\G$.
If $\rho=0$, the impedance boundary condition is reduced to the Neumann boundary condition
which means that $D$ is a sound-hard obstacle.

When $D$ is a penetrable obstacle, $u$ satisfies the reduced wave equation in $D$:
\be\label{eq4}
\Delta u+k^2nu=0\quad\mbox{in}\;\;{D}
\en
and the transmission boundary condition on $\G$:
\be\label{eq5}
u_+ -u_-=0,\quad\frac{\pa u_+}{\pa\nu}-\lambda\frac{\pa u_-}{\pa\nu}=0\quad\textrm{on}\;\;\G,
\en
where $n\in L^\infty({D})$ is the refractive index in $D$ which is a non-negative function and
characterizes the inhomogeneous material in $D$, $\la$ is a positive transmission constant depending
on the property of the medium in $\R^2\ba\ov{{D}}$ and ${D}$, and ``+/-" denotes the limits from
the exterior and interior of the boundary, respectively.

Further, the scattered field $u^s$ is required to satisfy the Sommerfeld radiation condition
\be\label{eq6}
\lim_{r\to\infty}r^\half\left(\frac{\pa u^s}{\pa r}-ik u^s\right)=0,\quad r=|x|
\en
which guarantees that the scattered field is outgoing.

By the variational method \cite{CC06} or the integral equation method \cite{CK83,CK13} it can be
shown that the Dirichlet scattering problem (\ref{eq1})-(\ref{eq2}) and (\ref{eq6}),
the impedance scattering problem (\ref{eq1}), (\ref{eq3}) and (\ref{eq6}), and
the transmission scattering problem (\ref{eq1}) and (\ref{eq4})-(\ref{eq6}) have a unique solution.
Further, it is known that the scattered field $u^s$ has the asymptotic behavior
\be\label{eq7}
u^s(x)=\frac{e^{ik|x|}}{\sqrt{|x|}}\left[u^\infty(\hat{x};d)
+O\left(\frac{1}{\sqrt{|x|}}\right)\right],\quad |x|\to\infty,
\en
uniformly for all observation directions $\hat{x}=x/|x|$ on $\mathbb{S}^1$,
where $u^\infty$ is called the far-field pattern of the scattered field $u^s$.

The inverse scattering problem is to determine the shape and location of the obstacle $D$ and its
physical property from the near-field (the scattered field $u^s$ or the total field $u$)
or the far field pattern $u^\infty$ and has been extensively studied mathematically and numerically
(see, e.g., the monographs \cite{CC06,CK13,KG08} and the references therein).
In many practical applications, however, the phase of the near-field or the far-field pattern can not
be measured accurately compared with its modulus or intensity and sometimes is even impossible to be
measured, and therefore it is often desirable to reconstruct the scattering obstacle from the modulus
or intensity of the near-field or the far-field pattern (or the phaseless near-field data or
the phaseless far-field data).

Inverse scattering problems with phaseless near-field data are also called the (near-field) phase
retrieval problem in optics and other physical and engineering sciences and have been widely studied
numerically over the past decades
(see, e.g. \cite{BaoLiLv2012,BaoZhang16,Candes15,TGRS03,CFH17,CH17a,CH17b,Li09,MDS92,MD93,Pan11,MOTL97}
and the references quoted there).
Recently, uniqueness and stability results have also been established in \cite{K14,K17,KR16,MH17,N16,N15}
for inverse medium scattering problems with phaseless near-field data.
However, not many results are available for inverse scattering problems with phaseless far-field data
both mathematically and numerically.
This is mainly because the modulus of the far-field pattern corresponding to one incident plane wave
is invariant under translations of the obstacle \cite{KR97,LS04}.
This means that the location of the obstacle can not be determined from the phaseless far-field data
if only one plane wave is used as the incident field.
Therefore, only the shape reconstruction of the obstacle was considered in the literature
for the case of phaseless far-field measurements corresponding to one plane wave as the incident field.
For example, reconstruction methods have been proposed to reconstruct the shape of the obstacle
or the real-valued surface impedance of the obstacle (assuming that the obstacle is known)
from the phaseless far-field data with one plane wave as the incident field
(see \cite{ACZ16,Ivanyshyn07,IvanyshynKress2010,IvanyshynKress2011,KR97,LiLiu15,S16}).
For plane wave incidence no uniqueness results are available for the inverse problem of
recovering scattering obstacles from phaseless far-field data generated by one incident plane wave.
By assuming a priori the obstacle to be a sound-soft ball centered at the origin, uniqueness was established
in determining the radius of the ball from a single phaseless far-field datum in \cite{LZ10}.
In \cite{majda76} it was proved, by studying the high frequency asymptotics of the far-field pattern,
that the shape of a general smooth convex sound-soft obstacle can be recovered from
the modulus of the far-field pattern associated with one plane wave as the incident field.

Recently, it was proved in \cite{ZZ17a} that the translational invariance property of the phaseless
far-field pattern can be broken by using superpositions of two plane waves rather than one plane wave
as the incident fields with an interval of frequencies. A recursive Newton-type iteration
algorithm in frequencies was also given in \cite{ZZ17a} to recover both the location and the shape of
the obstacle simultaneously from multi-frequency phaseless far-field data.
This idea was further extended to inverse scattering by locally rough surfaces with phaseless far-field
data in \cite{ZZ17b}. In \cite{XZZ18} it was rigorously proved for the first time that
the obstacle and the index of refraction of an inhomogeneous medium can be uniquely determined by
the phaseless far-field patterns generated by infinitely many sets
of superpositions of two plane waves with different directions at a fixed frequency
if the obstacle is a priori known to be a sound-soft obstacle or an impedance obstacle with a
real-valued impedance function and the refractive index $n$ of the inhomogeneous medium is real-valued
and satisfies the condition that either $n-1\ge c_1$ or $n-1\le-c_1$ in the support of $n-1$
for some positive constant $c_1$.
This paper develops a fast imaging algorithm to numerically recover the scattering obstacles
by phaseless (or intensity-only) far-field data at a fixed frequency.
A main feature of our imaging algorithm is its capability of depicting the surface of the obstacle
only through computing the products of the measured data and two exponential functions at each
sampling point, leading to both very fast computation speed and very low computational cost.
Moreover, our imaging algorithm does not require a prior knowledge of the physical property of
the obstacle, that is, the type of boundary conditions on the boundary of the obstacle does not
need to know in advance, so it works for both penetrable and impenetrable obstacles.
Numerical experiments show that our imaging algorithm can give a good and reliable
reconstruction of the obstacle, even for the case with a fairly high level of noise in the data,
which is comparable to that obtained by the direct imaging algorithm with full data.
It should be remarked that direct imaging (or sampling) methods have recently attracted more and more
attention due to their low computational cost and fast computation speed,
such as \cite{CCH13,CH15a,CH15b,IJZ12,LZZ18} for near-field data, \cite{G11,LLSS13,LLZ14,LZ13,L17,P10}
for far-field data and \cite{CFH17,CH17a,CH17b} for phaseless near-field data.

This paper is organized as follows.
Section \ref{sec2} is devoted to the direct imaging method for the inverse scattering problem
with phaseless far-field data. A performance analysis of the imaging function is also given
in Section \ref{sec2}. Numerical experiments are carried out in Section \ref{sec3} to illustrate
the effectiveness of the proposed direct imaging method. Some concluding remarks are presented
in Section \ref{sec4}.

\section{The direct imaging method for the inverse problem}\label{sec2}
\setcounter{equation}{0}

In this section, let the wave number $k$ be arbitrarily fixed. Following \cite{XZZ18,ZZ17a,ZZ17b},
we make use of the following superposition of two plane waves as the incident field:
\ben
u^i=u^i_{z_0}(x;d_1,d_2)=u^i_{z_0}(x,d_1)+u^i_{z_0}(x,d_2):=e^{ik(x-z_0)\cdot d_1}+e^{ik(x-z_0)\cdot d_2},
\enn
where $d_1,d_2\in\Sp^2$ are the incident directions and $z_0\in\R^2$.
Then the scattered field $u^s$ has the asymptotic behavior
\ben
u^s(x;d_1,d_2):=u^s_{z_0}(x;d_1,d_2)=\frac{e^{ik|x|}}{\sqrt{|x|}}\left\{u^\infty_{z_0}(\hat{x};d_1,d_2)
+O\left(\frac{1}{|x|}\right)\right\},\;\;|x|\to\infty,
\enn
uniformly for all observation directions $\hat{x}\in\Sp^2$. From the linear superposition principle
it follows that \[u^s_{z_0}(x;d_1,d_2)=u^s_{z_0}(x,d_1)+u^s_{z_0}(x,d_2)\] and
\be\label{lsp}
u^\infty_{z_0}(\hat{x};d_1,d_2)=u^\infty_{z_0}(\hat{x},d_1)+u^\infty_{z_0}(\hat{x},d_2),
\en
where $u^s_{z_0}(x,d_j)$ and $u^\infty_{z_0}(\hat{x},d_j)$ are the scattered field and its far-field
pattern corresponding to the incident plane wave $u^i_{z_0}(x,d_j)$, $j=1,2.$
Our inverse problem is to reconstruct the location and shape of the obstacle $D$ from the phaseless
far-field pattern $|u^\infty_{z_0}(\hat{x};d_1,d_2)|$ for all $\hat{x},d_1,d_2\in\Sp^1$ and
a fixed point $z_0\in\R^2$.

We will propose a direct imaging method to reconstruct the scattering obstacle $D$ from the phaseless
far-field data at a fixed frequency, that is, to solve our inverse problem numerically.
To this end, for the fixed wave number $k$ and for $z_0\in\R^2$ introduce the following imaging
function for continuous data:
\be\label{eq14}
I_{z_0}(z):=\left|\int_{\Sp^1}\int_{\Sp^1}\int_{\Sp^1}\left|u^\infty_{z_0}(\hat{x};d_1,d_2)\right|^2
e^{ik(z-z_0)\cdot d_1}e^{-ik(z-z_0)\cdot d_2}ds(d_1)ds(d_2)ds(\hat{x})\right|,\;\;z\in\R^2.
\en

We will study the behavior of $I_{z_0}(z)$ when $z$ approaches the obstacle and moves away from the obstacle.
To do this, we need the Funk-Hecke formula (see, e.g. \cite[equation (24)]{P10} or \cite[pp. 33]{CK13}).

\begin{lemma}\label{lem4}
For any $k>0$ we have
\be\label{eq15}
\int_{\mathbb{S}^1} e^{ik x\cdot d} ds(d)=2\pi J_0(k|x|),\quad x\in\mathbb{R}^2,
\en
where $J_0$ is the Bessel function of order $0$.
\end{lemma}

Using Lemma \ref{lem4} we can establish the following theorem concerning the imaging function $I_{z_0}$.

\begin{theorem}\label{thm2}
For $z_0\in\R^2$ the indicator function $I_{z_0}(z)$ has the following form:
\be\label{eq9}
I_{z_0}(z)=\left|I^{(1)}_{z_0}(z)+I^{(2)}_{z_0}(z)+I^{(3)}_{z_0}(z)\right|,\quad z\in\mathbb{R}^2
\en
with
\ben\label{eq25}
I^{(1)}_{z_0}(z)&=&\int_{\Sp^1}\left|v^\infty_{z_0}(\hat{x};z)\right|^2ds(\hat{x}),\\ \label{eq26}
I^{(2)}_{z_0}(z)&=&\int_{\S^1}\left|w^\infty_{z_0}(\hat{x};z)\right|^2ds(\hat{x}),\\ \label{eq12}
I^{(3)}_{z_0}(z)&=&2\pi J_0(k|z-z_0|)\int_{\Sp^1}\int_{\Sp^1}\left|u^\infty_{z_0}(\hat{x};d)\right|^2
\left(e^{ik(z-z_0)\cdot d}+e^{-ik(z-z_0)\cdot d}\right)ds(d)ds(\hat{x}),
\enn
where
\ben
v^\infty_{z_0}(\hat{x};z)&=&\int_{\Sp^1}u^\infty_{z_0}(\hat{x};d)e^{-ik(z-z_0)\cdot d}ds(d),\\
w^\infty_{z_0}(\hat{x};z)&=&\int_{\Sp^1}u^\infty_{z_0}(\hat{x};d)e^{ik(z-z_0)\cdot d}ds(d),
\enn
are the far-field patterns of
the solutions to the Dirichlet scattering problem (\ref{eq1})-(\ref{eq2}) and (\ref{eq6}) for the case
when $D$ is a sound-soft obstacle, to the impedance scattering problem (\ref{eq1}), (\ref{eq3}) and
(\ref{eq6}) for the case when $D$ is an impedance obstacle or to the transmission scattering problem
(\ref{eq1}) and (\ref{eq4})-(\ref{eq6}) for the case when $D$ is a penetrable obstacle,
corresponding to the incident fields $u^i(x)=2\pi J_0(k|x-z|)$
and $u^i(x)=2\pi J_0(k|x+z-2z_0|)$, respectively.
\end{theorem}

\begin{proof}
For simplicity, we only consider the impedance obstacle case. For other cases, the proofs are similar.

For $d_1,d_2\in\Sp^1$ we have $u^i_{z_0}(x;d_1,d_2)=u^i_{z_0}(x;d_1)+u^i_{z_0}(x;d_2)$ for $x\in\R^2$,
so $u^\infty_{z_0}(\hat{x};d_1,d_2)=u^\infty_{z_0}(\hat{x};d_1)+u^\infty_{z_0}(\hat{x};d_2)$
for $\hat{x}\in\Sp^1$. Inserting this into (\ref{eq14}), we obtain that
\ben
I_{z_0}(z)&=&\Bigg|\int_{\Sp^1}\int_{\Sp^1}\int_{\Sp^1}\left(|u^\infty_{z_0}(\hat{x};d_1)|^2
+u^\infty_{z_0}(\hat{x};d_1)\ol{u^\infty_{z_0}(\hat{x};d_2)}
+\ol{u^\infty_{z_0}(\hat{x};d_1)}u^\infty_{z_0}(\hat{x};d_2)\right.\\
&&\left.+|u^\infty_{z_0}(\hat{x};d_2)|^2\right)
e^{ik(z-z_0)\cdot d_1}e^{-ik(z-z_0)\cdot d_2}ds(d_1)ds(d_2)ds(\hat{x})\Bigg|
\enn
Interchanging the order of integration and using Lemma \ref{lem4}, we have
\ben
I_{z_0}(z)&=&\Bigg|\int_{\Sp^1}\bigg|\int_{\Sp^1}u^\infty_{z_0}(\hat{x};d)e^{-ik(z-z_0)
  \cdot d}ds(d)\bigg|^2ds(\hat{x})+\int_{\Sp^1}\bigg|\int_{\Sp^1}
  u^\infty_{z_0}(\hat{x};d)e^{ik(z-z_0)\cdot d}ds(d)\bigg|^2ds(\hat{x})\\
&&+\int_{\Sp^1}e^{-ik(z-z_0)\cdot d}ds(d)\cdot\int_{\Sp^1}\int_{\Sp^1}
  \left|u^\infty_{z_0}(\hat{x};d)\right|^2e^{ik (z-z_0)\cdot d}ds(d)ds(\hat{x})\\
&&+\int_{\Sp^1}e^{ik(z-z_0)\cdot d}ds(d)\cdot\int_{\Sp^1}\int_{\Sp^1}
  \left|u^\infty_{z_0}(\hat{x};d)\right|^2e^{-ik (z-z_0)\cdot d}ds(d)ds(\hat{x})\Bigg|\\
&=&\Bigg|\int_{\Sp^1}\left|v^\infty_{z_0}(\hat{x};z)\right|^2ds(\hat{x})
  +\int_{\Sp^1}\left|w^\infty_{z_0}(\hat{x};z)\right|^2ds(\hat{x})\\
&&+2\pi J_0(k|z-z_0|)\int_{\Sp^1}\int_{\Sp^1}\left|u^\infty_{z_0}(\hat{x};d)\right|^2
  \left(e^{ik (z-z_0)\cdot d}+e^{-ik(z-z_0)\cdot d}\right)ds(d)ds(\hat{x})\Bigg|
\enn
Note that $u^\infty_{z_0}(\hat{x};d)$ is the far-field pattern of the scattering solution to the impedance
scattering problem (\ref{eq1}), (\ref{eq3}) and (\ref{eq6}) with the incident field
$u^i(x)=u^i_{z_0}(x;d)=e^{ikd\cdot (x-z_0)}$. Therefore, by Lemma \ref{lem4} again it follows that
$v^\infty_{z_0}(\hat{x};z)$ and $w^\infty_{z_0}(\hat{x};z)$ are the far-field patterns of the scattering
solutions to the impedance scattering problem (\ref{eq1}), (\ref{eq3}) and (\ref{eq6})
with the incident field $u^i(x)=\int_{\Sp^1}u^i_{z_0}(x;d)e^{-ik(z-z_0)\cdot d}ds(d)=2\pi J_0(k|x-z|)$
and $u^i(x)=\int_{\Sp^1}u^i_{z_0}(x;d)e^{ik(z-z_0)\cdot d}ds(d)=2\pi J_0(k|x+z-2z_0|)$, respectively.
The proof is thus complete.
\end{proof}

We now study the behavior of the imaging function $I_{z_0}(z)$. To this end, we introduce some notations.
Let $J_n$ be the Bessel function of order $n$ for any non-negative integer $n$. Then we have \cite{CK13}
\be\label{eq13}
J_n(t)&=&\sum_{p=0}^\infty\frac{(-1)^p}{p!(n+p)!}\left(\frac{t}{2}\right)^{n+2p},\quad t\in\R,\\ \label{eq21}
J_n(t)&=&\sqrt{\frac{2}{\pi t}}\cos\left(t-\frac{n\pi}{2}-\frac{\pi}{4}\right)
\left\{1+O\left(\frac{1}{t}\right)\right\},\quad t\rightarrow\infty.
\en
In particular, $J_0(t)$ takes its maximum at $t_0=0$ with $J_0(t_0)=1$ and $J_1(t)$ takes its maximum
at $t_1\approx 1.84$ with $J_1(t_1)\thickapprox0.581$. Further, we have
\be\label{eq22}
J'_0(t)=-J_1(t)
\en
The behavior of the Bessel functions $J_0$ and $J_1$ is presented in Figure \ref{fig10}.
\begin{figure}[htbp]
  \centering
  \subfigure{\includegraphics[width=3in]{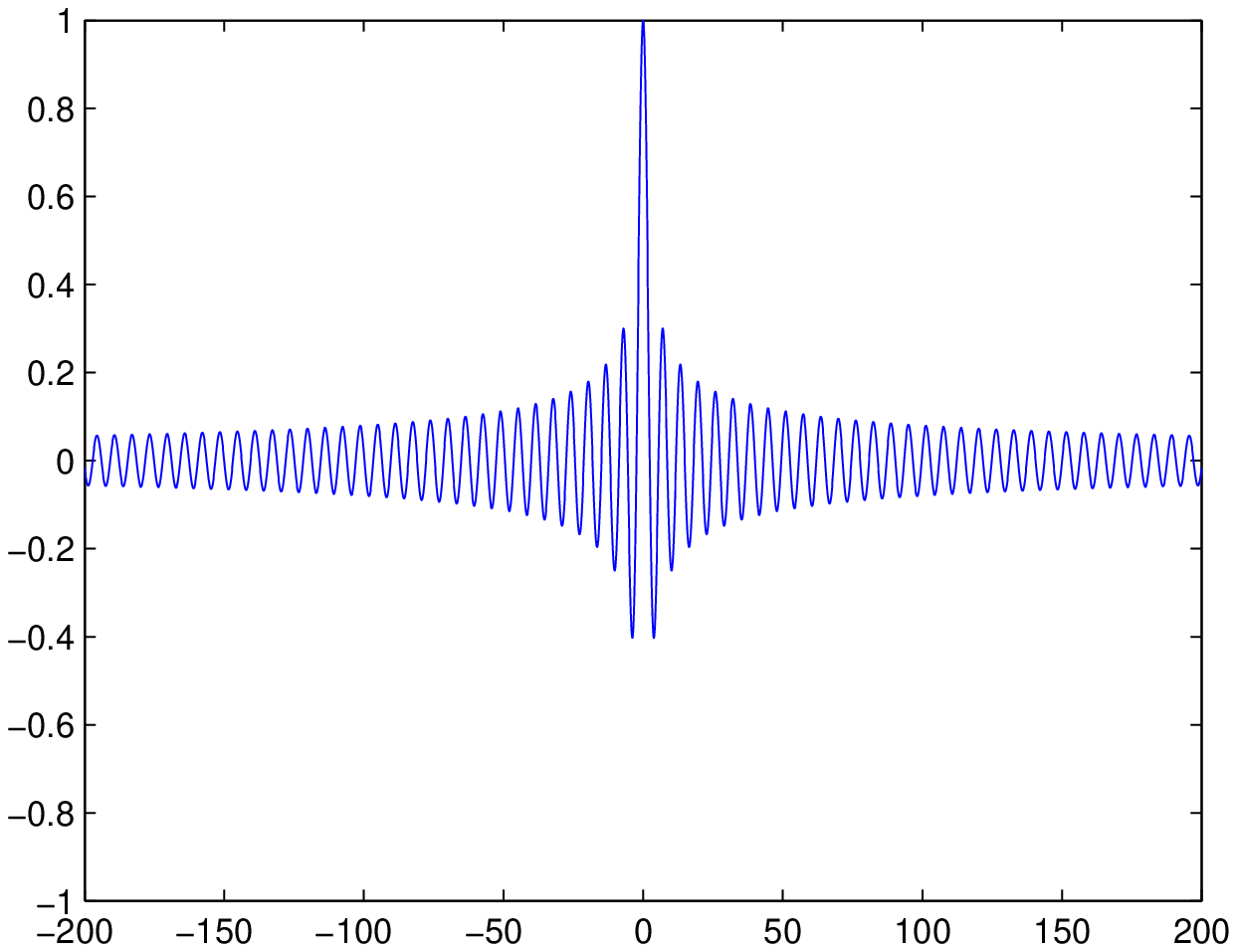}}
  \subfigure{\includegraphics[width=3in]{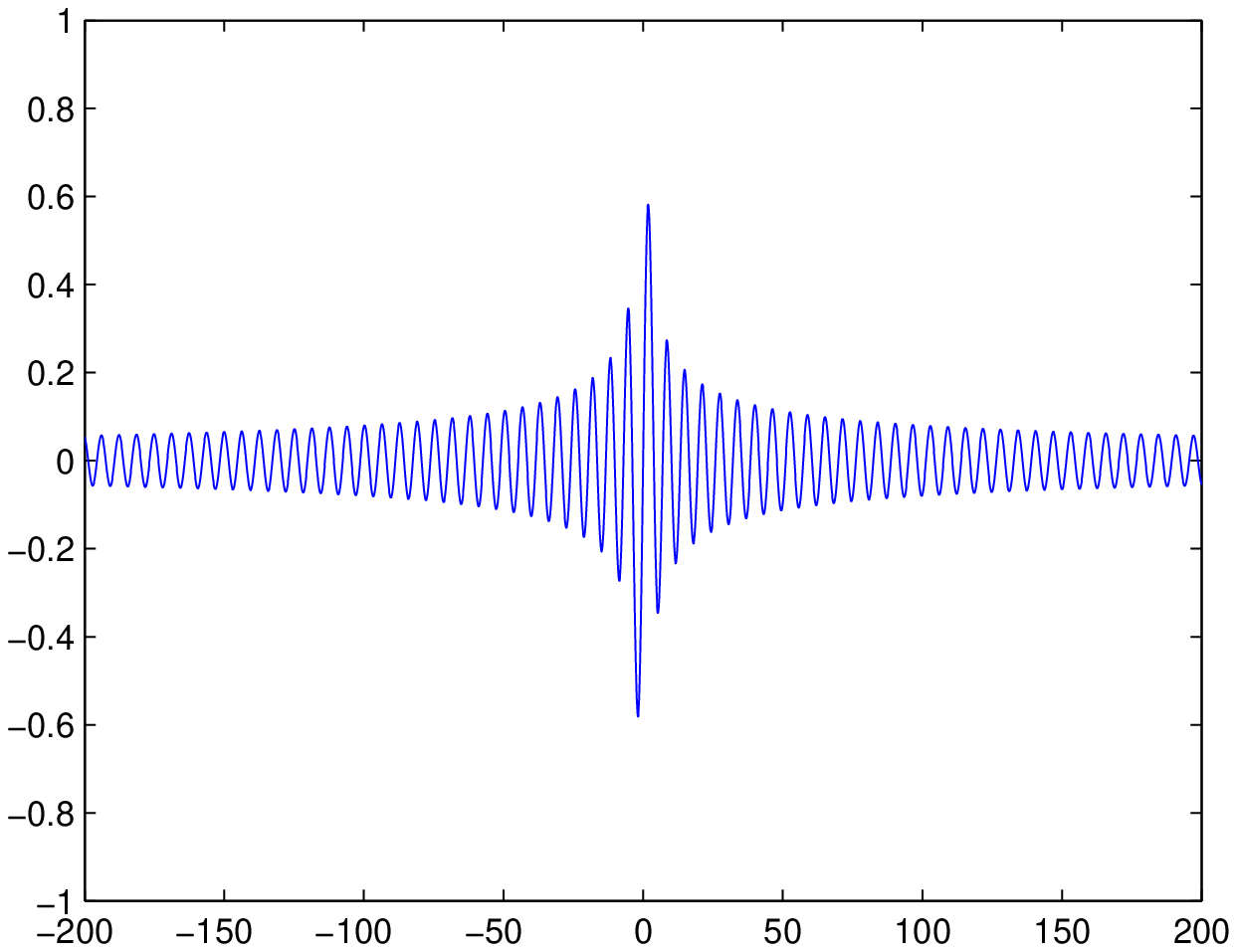}}
\caption{The behavior of the Bessel functions $J_0$ (left) and $J_1$ (right).
}\label{fig10}
\end{figure}

For $k>0$ denote by $\Phi_k(x,y)$ the fundamental solution of the Helmholtz equation $\Delta u+k^2u=0$
in $\R^2$ which is given by
\ben
\Phi_k(x,y):=\frac{i}{4}H^{(1)}_0(k|x-y|),\;\; x,y\in\R^2,x\neq y,
\enn
where $H^{(1)}_0$ is the Hankel function of the first kind of order $0$.
Define the single- and double-layer potentials
\ben
&&(\mathcal{S}_k\varphi)(x):=\int_\Gamma\Phi_k(x,y)\varphi(y)ds(y),\quad x\in\R^2\ba\Gamma,\\
&&(\mathcal{D}_k\varphi)(x):=\int_\Gamma\frac{\pa\Phi_k(x,y)}{\pa\nu(y)}\varphi(y)ds(y),
\quad x\in\R^2\ba\Gamma
\enn
and the boundary integral operators
\ben
&&(S_k\varphi)(x):=\int_\Gamma\Phi_k(x,y)\varphi(y)ds(y),\quad x\in\Gamma,\\
&&(K_k\varphi)(x):=\int_\Gamma\frac{\pa\Phi_k(x,y)}{\pa\nu(y)}\varphi(y)ds(y),\quad x\in\Gamma,\\
&&(K^T_k\varphi)(x):=\int_\Gamma\frac{\pa\Phi_k(x,y)}{\pa\nu(x)}\varphi(y)ds(y),\quad x\in\Gamma,\\
&&(T_k\varphi)(x):=\frac{\pa}{\pa\nu(x)}\int_\Gamma\frac{\pa\Phi_k(x,y)}{\pa\nu(y)}\varphi(y)ds(y),
  \quad x\in\Gamma.
\enn
Let
\ben
&&(S^\infty_k\varphi)(\hat{x}):=\frac{e^{i\pi/4}}{\sqrt{8\pi k}}\int_\Gamma e^{-ik\hat{x}\cdot y}
  \varphi(y)ds(y),\quad \hat{x}\in\Sp^1,\\
&&(K^\infty_k\varphi)(\hat{x}):=\frac{e^{i\pi/4}}{\sqrt{8\pi k}}
  \int_\Gamma\frac{\pa e^{-ik\hat{x}\cdot y}}{\pa\nu(y)}\varphi(y)ds(y),\quad\hat{x}\in\Sp^1.
\enn
Then it follows that $(S^\infty_k\varphi)(\hat{x})$ and $(K^\infty_k\varphi)(\hat{x})$
are the far-field patterns of $(\mathcal{S}_k\varphi)(x)$ and $(\mathcal{D}_k\varphi)(x)$, respectively.
Now define the volume potential
\ben
(\mathcal{V}_k\varphi)(x):=k^2\int_D \Phi_k(x,y)[n(y)-1]\varphi(y)ds(y),\quad x\in\R^2
\enn
and its restriction and its normal derivative at $\G$, respectively:
\ben
&&(V_k\varphi)(x):=(\mathcal{V}_k\varphi)(x),\quad x\in\Gamma,\\
&&(V_{k,\nu}\varphi)(x):=\frac{\pa\mathcal{V}_k\varphi}{\pa\nu}(x),\quad x\in\Gamma,
\enn
For mapping properties of the above operators, we refer to \cite{CK83,CK13}.

We first consider the impedance obstacle case. By Theorem \ref{thm2}, the imaging function $I_{z_0}(z)$
has the form (\ref{eq9}). For $z\in\R^2$ define
\ben
v^s_{z_0}(x;z):=\int_{\Sp^1}u^s_{z_0}(x;d)e^{-ik(z-z_0)\cdot d}ds(d),\quad x\in\R^2\ba\ov{D}.
\enn
Then $v^s_{z_0}(x;z)$ is the solution to the problem
\be\label{eq10}
\Delta u^s+k^2 u^s&=&0 \quad{\rm in}\;\;\R^2\ba\ov{{D}},\\ \label{eq10+}
\frac{\pa u^s}{\pa\nu}+ik\rho u^s&=&f_z \quad{\rm on}\;\;\G,\\ \label{eq11}
\lim_{r\to\infty}r^\half\left(\frac{\pa u^s}{\pa r}-ik u^s\right)&=&0,\quad r=|x|
\en
with the boundary data
\ben
f_z(x)&=&-2\pi\left(\frac{\pa}{\pa\nu(x)}+ik\rho(x)\right)J_0(k|x-z|)\\
&=&-2\pi k\left(-J_1(k|x-z|)\frac{(x-z)\cdot\nu(x)}{|x-z|}+i\rho(x)J_0(k|x-z|)\right),
\enn
where we have used the formula (\ref{eq22}). By \cite[Theorem 2.2]{KG08} it is known that
$v^s_{z_0}(x;z)$ can be expressed as
\ben
v^s_{z_0}(x;z)=(\mathcal{S}_k\varphi_z)(x)+i(\mathcal{D}_k\ov{K_k}K_k^T\varphi_z)(x),\quad x\in\R^2\ba\ov{D}.
\enn
Here, $\ov{K_k}$ denotes the conjugate of $K_k$ (i.e., $\ov{K_k}\varphi:=\ov{K_k\ol{\varphi}}$) and
$\varphi_z\in H^{-1/2}(\Gamma)$ is the unique solution to the boundary integral equation
$A_I\varphi_z=f_z$, where $A_I$ is bijective and thus invertible in $H^{-1/2}(\Gamma)$ and defined by
\ben
A_I\varphi:=\left(K^T_k-\frac{1}{2}I\right)\varphi+i T_k\ov{K_k}K^T_k\varphi
+ik\rho\left[S_k\varphi+i\left(K_k+\frac{1}{2}I\right)\ov{K_k}K^T_k\varphi\right]
\enn
By a similar argument as in the proof of \cite[Theorem 2.2]{KG08}, it can be shown that
$A_I$ is bijective and thus boundedly invertible in $C(\Gamma)$. Thus,
\ben
C_1\|f_z\|_{C(\Gamma)}\leq\|{\varphi_{z}}\|_{C(\Gamma)}\leq C_2\|f_z\|_{C(\Gamma)}
\enn
for some positive constants $C_1$ and $C_2$.

On the other hand, if $\inf_{x\in\G}\rho(x)>0$, then we have that for $x\in\Gamma$,
\ben
f_z(x)=\left\{\begin{array}{ll}
  \ds   -2\pi ki\rho(x) &\textrm{if}\;z=x,\\
  \ds   O\left(|x-z|^{-1/2}\right)&\textrm{if}\;|z-x|>>1,
    \end{array}\right.
\enn
which leads to the results that
\ben
\left\{\begin{array}{ll}
  \|{\varphi_{z}}\|_{C(\Gamma)}
  \geq 2\pi k C_1\inf_{x\in\G}\rho(x)&\textrm{if}\;z\in\Gamma,\\
  \|{\varphi_{z}}\|_{C(\Gamma)}=O\left({d(z,\Gamma)}^{-1/2}\right)&\textrm{if}\;d(z,\Gamma)>>1,
\end{array} \right.
\enn
where $d(z,\G)$ denotes the distance between $z$ and $\G$. If $\rho=0$ (i.e., the Neumann boundary condition),
then we have that for $x\in\Gamma$,
\ben
f_z(x)=\left\{\begin{array}{ll}
 \ds \mp2\pi k J_1(t_1) &\textrm{if}\;z= x\pm t_1\nu(x)/k\\
 \ds O\left({|x-z|^{-1/2}}\right)&\textrm{if}\;|z-x|>>1
  \end{array} \right.
\enn
with $t_1>0$ being defined as above, which leads to the results that
\ben
\left\{\begin{array}{ll}
\ds \|{\varphi_{z}}\|_{C(\Gamma)}\geq 2\pi k J_1(t_1)C_1&\textrm{if}\;d(z,\Gamma)=t_1/k,\\
\ds \|{\varphi_{z}}\|_{C(\Gamma)}=O\left({d(z,\Gamma)^{-1/2}}\right)&\textrm{if}\;d(z,\Gamma)>>1.
\end{array}\right.
\enn
Note that $v^\infty_{z_0}(\hat{x};z)$ is the far-field pattern of $v^s_{z_0}(x;z)$ and has the form
\ben
v^\infty_{z_0}(\hat{x};z)=(S^\infty_k\varphi_z)(\hat{x})
+(K^\infty_k\ov{K_k}K_k^T\varphi_z)(\hat{x}),\quad \hat{x}\in\Sp^1.
\enn
Then, based on the above observation and the mapping properties of the layer potentials,
it is expected that the function
$$
I^{(1)}_{z_0}(z)=\int_{\Sp^1}\left|v^\infty_{z_0}(\hat{x};z)\right|^2ds(\hat{x})
$$
takes a large value when $z\in\pa{D}$ and decays as $z$ moves away from $D$.

Similarly, by letting $D_{z_0}=\{2z_0-x:x\in{D}\}$ be the central symmetric obstacle of ${D}$
with respect to the point $z_0$, it is expected that the function
$$
I^{(2)}_{z_0}(z)=\int_{\Sp^1}\left|w^\infty_{z_0}(\hat{x};z)\right|^2ds(\hat{x})
$$
will take a large value when $z\in{D}_{z_0}$ and decay as $z$ moves away from ${D}_{z_0}$.

For $I^{(3)}_{z_0}(z)$, using (\ref{eq13}) and (\ref{eq21}) we have that for $z\in\R^2$,
\ben
|I^{(3)}_{z_0}(z)|&\leq& 4\pi\left|J_0(k|z-z_0|)\right|
   \int_{\Sp^1}\int_{\Sp^1}|u^\infty_{z_0}(\hat{x};d)|^2ds(d)ds(\hat{x})\\
&\leq&4\pi\int_{\Sp^1}\int_{\Sp^1}|u^\infty_{z_0}(\hat{x};d)|^2ds(d)ds(\hat{x})
\enn
with
\ben
I^{(3)}_{z_0}(z_0)=4\pi\int_{\Sp^1}\int_{\Sp^1}|u^\infty_{z_0}(\hat{x};d)|^2ds(d)ds(\hat{x})
\enn
and $I^{(3)}_{z_0}(z)=O\left({d(z,z_0)^{-1/2}}\right)$ when $d(z,z_0)>>1.$
Therefore, $I^{(3)}_{z_0}(z)$ will take its maximum at $z=z_0$ and decay as $z$ moves away from $z_0$.
From the above discussion, it is expected that the imaging function $I_{z_0}(z)$
will take a large value when the sampling point $z$ approaches $\pa{D}\cup\pa{D}_{z_0}\cup\{z_0\}$
and decay as $z$ moves away from ${{D}}\cup{{D}}_{z_0}\cup\{z_0\}$.
This is indeed confirmed in the numerical examples.

\begin{remark}\label{rem1} {\rm
In the numerical examples, it is observed that for the sound-hard obstacle $D$,
if $d(z_0,D)$ is large enough then the imaging function $I_{z_0}(z)$ takes a small value
on the boundary of the obstacle $D$ and its central symmetric one ${D}_{z_0}$.
This might be due to the fact that $J_1(0)=0$ since, in this case, $v^\infty_{z_0}(\hat{x};z)$ and
$w^\infty_{z_0}(\hat{x};z)$ are the far-field patterns of the solutions to the scattering problem
(\ref{eq10})-(\ref{eq11}) with boundary data $f_z(x)=2\pi k J_1(k|x-z|){(x-z)\cdot\nu(x)}/{|x-z|}$
and $f_z(x)=2\pi k J_1(k|x+z-2z_0|){(x+z-2z_0)\cdot\nu(x)}/{|x+z-2z_0|}$, respectively.
Note that a similar feature was found in \cite{P10} for the orthogonality sampling method.
}
\end{remark}

We now consider the case of a sound-soft obstacle. It is easy to see that
$v^\infty_{z_0}(\hat{x};z)$ and $w^\infty_{z_0}(\hat{x};z)$ are the far-field patterns of
the solutions to the scattering problem
\ben
\Delta u^s+k^2 u^s&=&0\quad{\rm in}\;\;\R^2\ba\ol{{D}},\\
u^s&=&f_z\quad {\rm on}\;\;\G\\
\lim_{r\to\infty}r^\half\left(\frac{\pa u^s}{\pa r}-ik u^s\right)&=&0,\quad r=|x|
\enn
with boundary data $f_z(x)=-2\pi J_0(k|x-z|)$ and $f_z(x)=-2\pi J_0(k|x+z-2z_0|)$, respectively.
By the integral equation method (see, e.g. \cite[Section 3.2]{CK13}) and the properties
of the Bessel functions, it is also expected that the imaging function $I_{z_0}(z)$
takes a large value when the sampling point $z$ approaches $\pa{D}\cup\pa{D}_{z_0}\cup\{z_0\}$
and decays as $z$ moves away from ${{D}}\cup{{D}}_{z_0}\cup\{z_0\}$.
The discussion is similar to the impedance obstacle case, so it is omitted.

Finally, we consider the case of a penetrable obstacle. Define
\ben
&&v^s_{z_0}(x;z):=\int_{\Sp^1}u^s_{z_0}(x;d)e^{-ik(z-z_0)\cdot d}ds(d),\quad x\in\R^2\ba\ol{D},\\
&&v_{z_0}(x;z):=\int_{\Sp^1}u_{z_0}(x;d)e^{-ik(z-z_0)\cdot d}ds(d),\quad x\in D.
\enn
Then the pair of functions ($v^s_{z_0}(\cdot;z),v_{z_0}(\cdot;z))$ is the solution to the scattering problem:
\ben
\Delta u^s+k^2 u^s&=&0\quad\textrm{in}\;\;\R^2\ba\ol{{D}},\\
\Delta u+k^2 nu&=&0\quad\textrm{in}\;\;{D},\\
u^s_+ -u_-=f_{1,z},\quad\frac{\pa u^s_+}{\pa\nu}-\lambda\frac{\pa u_-}{\pa\nu}&=&f_{2,z}\quad\textrm{on}\;\;\G,\\
\lim_{r\to\infty}r^\half\left(\frac{\pa u^s}{\pa r}-ik u^s\right)&=&0,\quad r=|x|
\enn
with boundary data $f_{1,z}(x)=-2\pi J_0(k|x-z|)$ and $f_{2,z}(x)=-2\pi{\pa}J_0(k|x-z|)/{\pa\nu(x)}$,
and $v^\infty_{z_0}(\hat{x};z)$ is the far-field pattern of $v^s_{z_0}(x;z)$.
By the integral equation method (see \cite{CP98}), it follows that
$v^s_{z_0}(x;z)$ and $v_{z_0}(x;z)$ can be expressed as
\ben
v^s_{z_0}(x;z)&=&\lambda(\mathcal{D}_k\varphi_{1,z})(x)+(\mathcal{S}_k\varphi_{2,z})(x),\quad x\in\R^2\ba\ol{D},\\
v_{z_0}(x;z)&=&(\mathcal{D}_k\varphi_{1,z})(x)+(\mathcal{S}_k\varphi_{2,z})(x)+(\mathcal{V}\varphi_{3,z})(x),
\quad x\in D,
\enn
where $\phi_z:=(\varphi_{1,z},\varphi_{2,z},\varphi_{3,z})^T$ is the unique solution to the integral equation
$A_T\phi_z=F_z$ with
\ben
A_T=\left(\begin{array}{ccc}
\frac{\la+1}{2}I+(\la-1)K_k&0&-V_k\\
0&-\frac{\la+1}{2}I-(\la-1)K^T_k&-\la V_{k,\nu}\\
-\wid{K}_k&-\wid{S}_k&I-\wid{V}_k
\end{array}\right),\quad
F_z=\left(\begin{array}{c}
f_{1,z}\\
f_{2,z}\\
0
\end{array}\right).
\enn
Here, $\wid{S}_k$, $\wid{K}_k$ and $\wid{V}_k$ are the restrictions to $D$ of $\mathcal{S}_k\varphi$,
$\mathcal{D}_k\varphi$ and $\mathcal{V}_k\varphi$, respectively, that is,
$\wid{S}_k\varphi:=(\mathcal{S}_k\varphi)|_D$, $\wid{K}_k\varphi:=(\mathcal{D}_k\varphi)|_D$,
$\wid{V}_k\varphi:=(\mathcal{V}_k\varphi)|_D$.
Further, $A_I$ is bijective, so it is invertible in $C(\Gamma)\times C(\Gamma)\times C(\ol{D})$ (see \cite{CP98}).
Thus we have
\ben
C_1\left(\|f_{1,z}\|_{C(\Gamma)}+\|f_{2,z}\|_{C(\Gamma)}\right)
&\le& \|{\varphi_{1,z}}\|_{C(\Gamma)}+\|{\varphi_{2,z}}\|_{C(\Gamma)}+\|{\varphi_{3,z}}\|_{C(\ol{D})}\\
&\le& C_2\left(\|f_{1,z}\|_{C(\Gamma)}+\|f_{2,z}\|_{C(\Gamma)}\right).
\enn
On the other hand, from the properties of the Bessel functions, we have that for $x\in\Gamma$,
\ben
f_{1,z}(x)=\left\{\begin{array}{ll}
-2\pi &\textrm{if}\;\;z=x,\\
O\left({|x-z|^{-1/2}}\right)&\textrm{if}\;\;|z-x|>>1
\end{array}\right.
\;\textrm{and}\;\;\;
f_{2,z}(x)=\left\{\begin{array}{ll}
0 &\textrm{if}\;\;z=x,\\
O\left({|x-z|^{-1/2}}\right)&\textrm{if}\;\;|z-x|>>1.
\end{array} \right.
\enn
Therefore, we have
\ben
\left\{\begin{array}{ll}
\|{\varphi_{1,z}}\|_{C(\Gamma)}+\|{\varphi_{2,z}}\|_{C(\Gamma)}+\|{\varphi_{3,z}}\|_{C(\ol{D})}
    \ge 2\pi C_1&\textrm{if}\;\;z\in\Gamma,\\
\|{\varphi_{1,z}}\|_{C(\Gamma)}+\|{\varphi_{2,z}}\|_{C(\Gamma)}+\|{\varphi_{3,z}}\|_{C(\ol{D})}
   =O\left({d(z,\Gamma)^{-1/2}}\right)&\textrm{if}\;\;d(z,\Gamma)>>1.
\end{array} \right.
\enn
Noting that $v^\infty_{z_0}(\hat{x};z)=\la(K^\infty_k\varphi_{1,z})(\hat{x})+(K^\infty_k\varphi_{2,z})(\hat{x})$,
$\hat{x}\in\Sp^1$, we expect that the function $I^{(1)}_{z_0}(z)$ takes a large value when $z\in\pa{D}$
and decays as $z$ moves away from $D$. Similarly as before for other cases,
the functions $I^{(2)}_{z_0}(z)$ and $I^{(3)}_{z_0}(z)$ have a similar behavior as $I^{(1)}_{z_0}(z)$.
Thus it can be expected that the imaging function $I_{z_0}(z)$
will take a large value when the sampling point $z\in\pa{D}\cup\pa{D}_{z_0}\cup\{z_0\}$
and decay as $z$ moves away from ${D}\cup{D}_{z_0}\cup\{z_0\}$.

\begin{remark}\label{rem2} {\rm
Assume that $u^\infty_{z_0}(\hat{x};d_1,d_2,D)$ and $u^\infty_{z_0}(\hat{x};d_1,d_2,D_{z_0})$
are the far-field patterns of the solutions of either the scattering problem (\ref{eq1})-(\ref{eq2})
and (\ref{eq6}) with the Dirichlet obstacles $D$ and $D_{z_0}$, respectively, or the
impedance scattering problem (\ref{eq1}), (\ref{eq3}) and (\ref{eq6}) with the impedance obstacle and
the impedance function being $(D,\rho)$ and $(D_{z_0},\rho_{z_0})$, respectively, where
$\rho_{z_0}(x):=\rho(2z_0-x)$, $x\in\pa D_{z_0}$,
or the transmission scattering problem (\ref{eq1}), (\ref{eq4})-(\ref{eq6}) with the penetrable obstacle,
the transmission constant and the refractive index being $(D,\la,n)$ and $(D_{z_0},\la,n_{z_0})$, respectively,
where $n_{z_0}(x):=n(2z_0-x)$, $x\in\pa D_{z_0}$. Then it is easy to see that
\be\label{eq30}
u^\infty_{z_0}(\hat{x};d,D_{z_0})=e^{-2ikz_0\cdot\hat{x}}u^\infty_{z_0}(-\hat{x};-d,D),
\quad\forall\;\hat{x},d\in\Sp^1.
\en
Denote by $I_{z_0}(z,D)$ and $I_{z_0}(z,D_{z_0})$ the imaging function (\ref{eq14}) with respect to
$|u^\infty_{z_0}(\hat{x};d_1,d_2,D)|$ and $|u^\infty_{z_0}(\hat{x};d_1,d_2,D_{z_0})|$, respectively.
Then, by (\ref{eq30}) and change of variables we have
\ben
I_{z_0}(z,D_{z_0})&=&\left|\int_{\Sp^1}\int_{\Sp^1}\int_{\Sp^1}\left|u^\infty_{z_0}(\hat{x};d_1,D_{z_0})
  +u^\infty_{z_0}(\hat{x};d_2,D_{z_0})\right|^2e^{ik(z-z_0)\cdot d_1}
  e^{-ik(z-z_0)\cdot d_2}ds(d_1)ds(d_2)ds(\hat{x})\right|\\
&=&\left|\int_{\Sp^1}\int_{\Sp^1}\int_{\Sp^1}\left|u^\infty_{z_0}(-\hat{x};-d_1,D)
   +u^\infty_{z_0}(-\hat{x};-d_2,D)\right|^2
   e^{ik(z-z_0)\cdot d_1}e^{-ik(z-z_0)\cdot d_2} ds(d_1)ds(d_2)ds(\hat{x})\right|\\
&=&\left|\int_{\Sp^1}\int_{\Sp^1}\int_{\Sp^1}\left|u^\infty_{z_0}(\hat{x};d_1,D)
   +u^\infty_{z_0}(\hat{x};d_2,D)\right|^2
   e^{-ik(z-z_0)\cdot d_1}e^{ik(z-z_0)\cdot d_2} ds(d_1)ds(d_2)ds(\hat{x})\right|\\
&=&I_{z_0}(z,D)
\enn
for all $z\in\R^2$. This means that the actual obstacle ${D}$ and the artifact image $D_{z_0}$ can not be
distinguished by the imaging function $I_{z_0}(z)$ with a fixed $z_0$.
However, since the location of the artifact image ${D}_{z_0}$ depends on the point $z_0$, we can use
the imaging function $I_{z_0}(z)$ with two different points $z_0$ to recover the actual obstacle $D$.
}
\end{remark}

We now introduce the direct imaging method for the inverse scattering problem with phaseless far-field data.
We assume that there are $M$ measurement points $\hat{x_i}$, $i=1,\ldots,M$, uniformly distributed on $\Sp^1$
and $N$ sets of two incident directions $d_{1j},d_{2l}$, $j,l=1,\ldots,N$, uniformly distributed on $\Sp^1$
with $d_{1j}\neq d_{2l}$ for $j\not=l$.
For $z_0\in\R^2$ let $|u^\infty_{z_0}(\hat{x_i};d_{1j},d_{2l})|^2$, $i=1,\ldots,M$, $j,l=1,\ldots,N$, be
the measured phaseless far-field data, where $u^\infty_{z_0}(\hat{x};d_{1},d_{2})$ is the far-field pattern
of the solution of either the scattering problem (\ref{eq1})-(\ref{eq2}) and (\ref{eq6}), or the
impedance scattering problem (\ref{eq1}), (\ref{eq3}) and (\ref{eq6}),
or the transmission scattering problem (\ref{eq1}), (\ref{eq4})-(\ref{eq6}), with the incident field
$u^i=u^i(x,d_1,d_2)=\exp[ikd_1\cdot(x-z_0)]+\exp[ikd_2\cdot(x-z_0)]$, $d_1,d_2\in\Sp^1$, $d_1\not=d_2.$
Define
\be\label{eq23}
I^A_{z_0}(z):=\left|\frac{2\pi}{M}\left(\frac{2\pi}{N}\right)^2\sum_{i=1}^M\sum_{j=1}^N
\sum_{l=1,l\neq j}^N\left|u^\infty_{z_0}(\hat{x_i};d_{1j},d_{2l})\right|^2
e^{ik(z-z_0)\cdot d_{1j}}e^{-ik(z-z_0)\cdot d_{2l}}\right|.
\en
Then $I^A_{z_0}(z)$ is a good trapezoid quadrature approximation to the continuous imaging function $I_{z_0}(z)$.
Our direct imaging method is based on the formula (\ref{eq23}) and presented in the following algorithm.

\begin{algorithm}\label{al1}
Let $\Om_L$ be a large sampling domain containing the obstacle $D$.

Step 1. Locating a smaller domain containing the obstacle $D$.

1.1. Choose a point $z_0=z_{10}\in\R^2$. Collect the measurement data
$|u^\infty_{z_{10}}(\hat{x_i};d_{1j},d_{2l})|^2$, $i=1,\ldots,M$, $j,l=1,2,\ldots,N$,
corresponding to the incident plane waves $u^i_{z_{10}}(x,d_{1j},d_{2l})$,
$j,l=1,2,\ldots,N$, with $d_{1j}\neq d_{2l}$ for $j\not=l$.
For $z\in\Om_L$, compute the imaging function $I^A_{z_{10}}(z)$ by using the formula (\ref{eq23}).
Locate all the sampling points $z$ at which $I^A_{z_{10}}(z)$ takes a large value, representing the
reconstructed image of ${D}\cup {D}_{z_{10}}\cup z_{10}$ in the large sampling region $\Om_L$.

1.2. Choose another point $z_0=z_{20}\in\R^2$ which is away from the reconstructed image
of ${D}\cup {D}_{z_{10}}\cup z_{10}$ in the previous step (1.1).
Collect the measurement data $|u^\infty_{z_{20}}(\hat{x_i};d_{1j},d_{2l})|^2$,
$i=1,\dots,M$, $j,l=1,2,\ldots,N$, corresponding to the incident plane waves $u^i_{z_{20}}(x,d_{1j},d_{2l})$,
$j,l=1,2,\ldots,N$, where $d_{1j}\neq d_{2l}$ for $j\not=l$.
For $z\in\Om_L$, compute the imaging function $I^A_{z_{20}}(z)$ by using (\ref{eq23}) again.
Locate all the sampling point $z$ at which $I^A_{z_{20}}(z)$ takes a large value, representing the reconstructed
image of ${D}\cup {D}_{z_{20}}\cup z_{20}$ in the sampling region $\Om_L$.
Note that the the actual obstacle $D$ is independent of the choice of the points ${z_{10}}$ and ${z_{20}}$,
while the artifact images $D_{z_{10}}$ and $D_{z_{20}}$ depend on the points $z_{10}$ and $z_{20}$.
Thus, in this way, we can distinguish between the actual obstacle $D$ and its two artifact images $D_{z_{10}}$
and $D_{z_{20}}$. Denote by $\Om_S$ the smaller domain which only contains the actual obstacle $D$
but not its artifact images.

Step 2. Reconstructing the obstacle $D$. For $z\in\Om_S$ compute the imaging function $I^A_{z_{20}}(z)$
by using the formula (\ref{eq23}) and locate all the sampling point $z$ at which $I_{z_{20}}(z)$ takes
a large value to reconstruct the obstacle $D$.
\end{algorithm}

The procedure of Algorithm \ref{al1} will be presented in Example 1 in the next section on numerical experiments,
where the actual obstacle $D$ is accurately reconstructed.

\begin{remark}\label{rem3} {\rm
It is easy to see that our direct imaging method does not need to know the type of boundary conditions on
the obstacle in advance. Further, it is known from the numerical examples that the proposed
direct imaging method can also be used to reconstruct multiple obstacles with different boundary conditions.
}
\end{remark}

\section{Numerical experiments}\label{sec3}
\setcounter{equation}{0}

We present several numerical examples to illustrate the effectiveness of the direct imaging algorithm \ref{al1}.
Our algorithm will be compared with the direct imaging method using the full far-field data given in \cite{P10}
and based on the imaging function
\ben\label{eq23+}
I^A_F(z):=\frac{2\pi}{M}\frac{2\pi}{N}\sum_{i=1}^M\left|\sum_{j=1}^N
u^\infty(\hat{x_i};d_j)e^{ikz\cdot\hat{x_i}}\right|^2,
\enn
which is the trapezoid quadrature approximation to the continuous imaging function
\ben
I_F(z):=\int_{\Sp^1}\left|\int_{\Sp^1}u^\infty(\hat{x};d)e^{ik z\cdot\hat{x}}ds(\hat{x})\right|^2ds(d).
\enn
Here, $u^\infty(\hat{x};d)$ is the far-field pattern of the scattering solution generated by the incident
plane wave $u^i=e^{ikd\cdot x}$ with $\hat{x},d\in\Sp^1$,
$\hat{x_i}$, $i=1,\ldots,M$, are the measurement points uniformly distributed on $\Sp^1$,
and $d_j$, $j,l=1,\ldots,N$, are the incident directions uniformly distributed on $\Sp^1$.

To generate the synthetic data, we use the Nystr\"{o}m method \cite{CK13} to solve the direct scattering problem.
Unless otherwise stated, the far-field data are measured with $360$ incident and observed directions which are
uniformly distributed on $\Sp^1$ (that is, $M=N=360$).
Further, the noisy far-field data $u^\infty_{\delta}(\hat{x};d)$ and the noisy phaseless far-field data
$|u^\infty_{z_0,\delta}(\hat{x};d_1,d_2)|^2$ are given as follows:
\ben
&&u^\infty_{\delta}(\hat{x};d)=u^\infty(\hat{x};d)+\delta(\zeta_1+i\zeta_2)\max|u^\infty(\hat{x};d)|,\\
&&|u^\infty_{z_0,\delta}(\hat{x};d_1,d_2)|^2=|u^\infty_{z_0}(\hat{x};d_1,d_2)|^2
+\delta\zeta_3\max|u^\infty_{z_0}(\hat{x};d_1,d_2)|^2,
\enn
where $\delta$ is the noise ratio and $\zeta_1,\zeta_2,\zeta_3$ are the standard normal distributions.

The parametrization of the curves we used are given in Table \ref{table1}.
\begin{table}[h]
\centering
\begin{tabular}{ll}
\hline
Type & Parametrization\\
\hline
Circle &$(c_1,c_2)+r(\cos{t},\sin{t}),\;t\in[0,2\pi]$\\
Apple shaped &$(c_1,c_2)+[({0.5+0.4\cos{t}+0.1\sin(2t)})/({1+0.7\cos{t}})](\cos{t},\sin{t}),\;t\in[0,2\pi]$\\
Kite shaped &$(c_1,c_2)+(\cos{t}+0.65\cos(2t)-0.65,1.5\sin{t}),\;t\in[0,2\pi]$\\
Peanut shaped&$(c_1,c_2)+\sqrt{\cos^2{t}+0.25\sin^2{t}}(\cos{t},\sin{t})$, $t\in[0,2\pi]$\\
Rounded square & $(c_1,c_2)+({3}/{2})(\cos^3{t}+\cos{t},\sin^3{t}+\sin{t}),\;t\in[0,2\pi]$\\
Rounded triangle &$(c_1,c_2)+(2+0.3\cos(3t))(\cos{t},\sin{t}),\;t\in[0,2\pi]$\\
\hline
\end{tabular}
\caption{Parametrization of the curves}\label{table1}
\end{table}

\textbf{Example 1: Imaging of a sound-soft obstacle.}

We consider the scattering problem by an apple-shaped, sound-soft obstacle.
The wave number is chosen to be $k=20$.
We first show how to use Step 1 of Algorithm \ref{al1} to locate the small region which contains
the actual obstacle. Choose the large sampling region ${K_L}$ to be $[-12,2]\times[-12,2]$
which contains the actual obstacle (see Figure \ref{fig1-1}).
First choose the point $z_0=z_{10}:=(-1,-5)^T$.
Figure \ref{fig1-2} presents the corresponding imaging result of $I^A_{z_{10}}(z)$, $z\in{K_L}$,
from the phaseless far-field data with $10\%$ noise.
Next choose the point $z_0=z_{20}:=(-5,-4)^T$.
Figure \ref{fig1-3} presents the corresponding imaging result of $I^A_{z_{20}}(z)$, $z\in{K_L}$,
from the phaseless far-field data with $10\%$ noise.
We observe that $I^A_{z_0}$ takes a large value in the neighborhood of the point $z_0$, the actual obstacle
and its central symmetry obstacle with respect to $z_0$ (i.e., the artifact image $D_{z_0}$),
which is consistent with Theorem \ref{thm2} and the discussion in Section \ref{sec3}.
From Figures \ref{fig1-2} and \ref{fig1-3} we see that the reconstructed obstacle in the
region $[-1,1]\times[-1,1]$ is the actual obstacle which is independent of the choice of the point $z_0$,
while the reconstructed artifact image depends on the choice of the point $z_0$ and thus moves in the
large region $K_L$ with different $z_0$.
Therefore, we can locate the small sampling region ${K_S}=[-1,1]\times[-1,1]$ containing the actual obstacle.
Figure \ref{fig1-4} presents the imaging result of $I^A_{z_{20}}(z),\,z\in{K_S}$, from the phaseless far-field
data with $10\%$ noise which gives a satisfactory reconstruction of the actual obstacle.

For simplicity, we only present the reconstruction results by using Step 2 of Algorithm \ref{al1}
for the remaining examples.

We now compare the reconstruction results obtained by using the phaseless and full far-field
measurement data. We choose the sampling region to be $[-1,1]\times[-1,1]$ and $z_0=(-5,-4)^T$.
Figure \ref{fig2} presents the exact curve, the imaging results of $I^A_{z_0}(z)$ and $I_F(z)$
from the measured data without noise, with $10\%$ noise and with $20\%$ noise, respectively.
It is shown that the reconstruction results given by the two imaging methods are stable,
accurate and robust to noise in the data.
It is remarked that, due to the influence of $I^{(2)}_{z_0}$ and $I^{(3)}_{z_0}$
in the formula (\ref{eq9}), the imaging results of $I^A_{z_0}(z)$ with the phaseless far-field data
are not as good as those of $I^A_F(z)$ with the full far-field data.

\begin{figure}[htbp]
  \centering
  \subfigure[\textbf{Exact curve}]{\label{fig1-1}
    \includegraphics[width=3in]{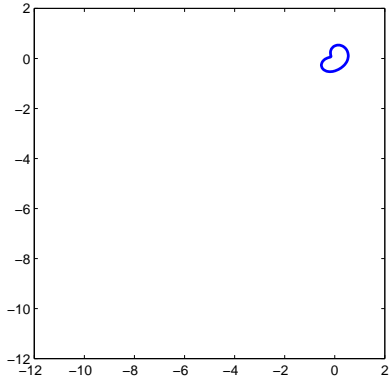}}
  \subfigure[\textbf{Step 1: 10\% noise, k=20}]{\label{fig1-2}
    \includegraphics[width=3in]{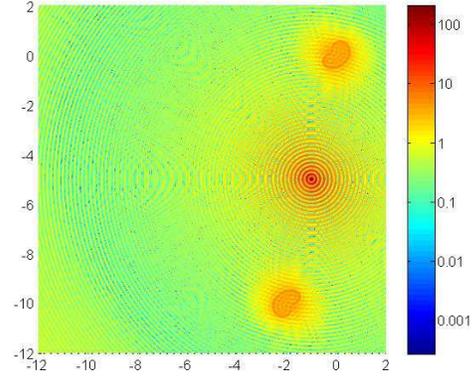}}
  \subfigure[\textbf{Step 2: 10\% noise, k=20}]{\label{fig1-3}
    \includegraphics[width=3in]{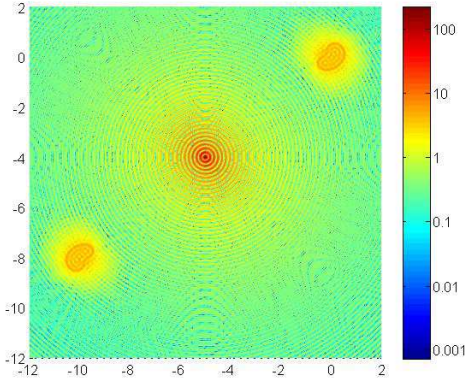}}
  \subfigure[\textbf{Step 3: 10\% noise, k=20}]{\label{fig1-4}
    \includegraphics[width=3in]{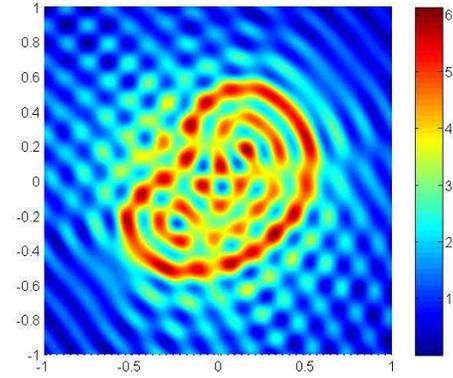}}
\caption{Imaging results of an apple-shaped, sound-soft obstacle with Algorithm \ref{al1}
using phaseless far-field data with $10\%$ noise.
(a) gives the exact curve, (b) presents the imaging result of $I_{z_{10}}(z),\,z\in {K_L}$ with
$z_{10}=(-1,-5)^T$, (c) shows the imaging result of $I_{z_{20}}(z),\,z\in {K_L}$ with $z_{20}=(-5,-4)^T$,
and (d) presents the imaging result of $I_{z_{20}}(z),\,z\in {K_S}$.
The large sampling region ${K_L}=[-12,2]\times[-12,2]$, the small sampling region ${K_S}=[-1,1]\times[-1,1]$,
and the wave number $k=20$.
}\label{fig1}
\end{figure}

\begin{figure}[htbp]
  \centering
  \subfigure[\textbf{Exact curve}]{
    \includegraphics[width=1.5in]{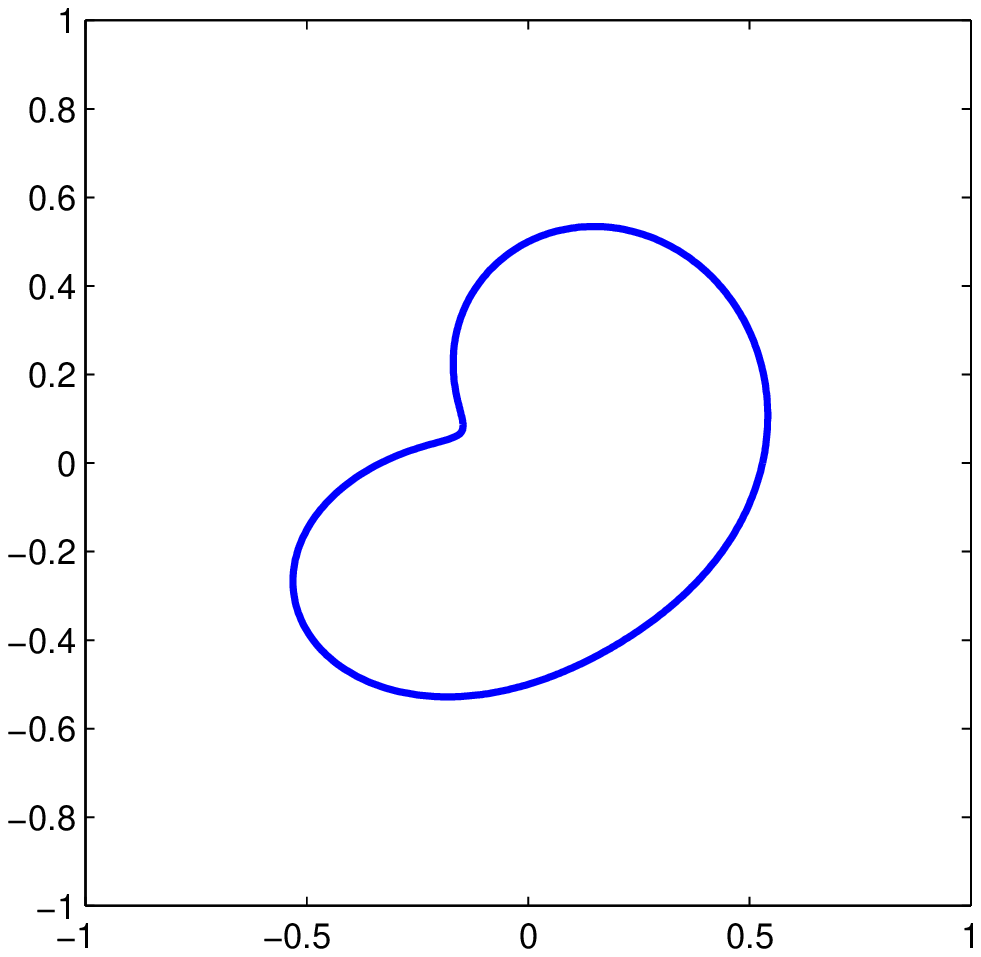}}
  \subfigure[\textbf{No noise, k=20}]{
    \includegraphics[width=1.5in]{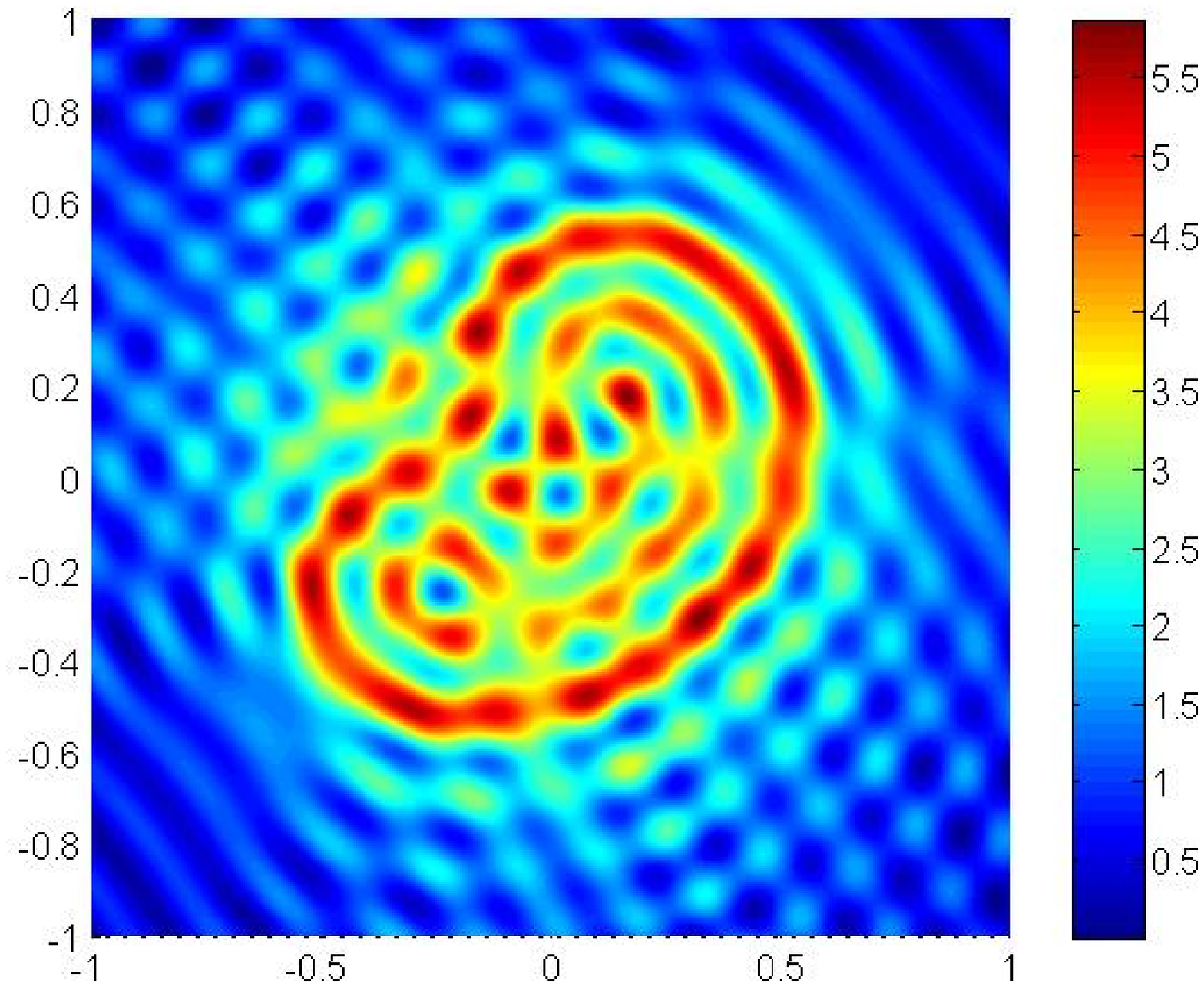}}
  \subfigure[\textbf{10\% noise, k=20}]{
    \includegraphics[width=1.5in]{pic/example1/Phaseless/percent_10_error.eps}}
  \subfigure[\textbf{20\% noise, k=20}]{
    \includegraphics[width=1.5in]{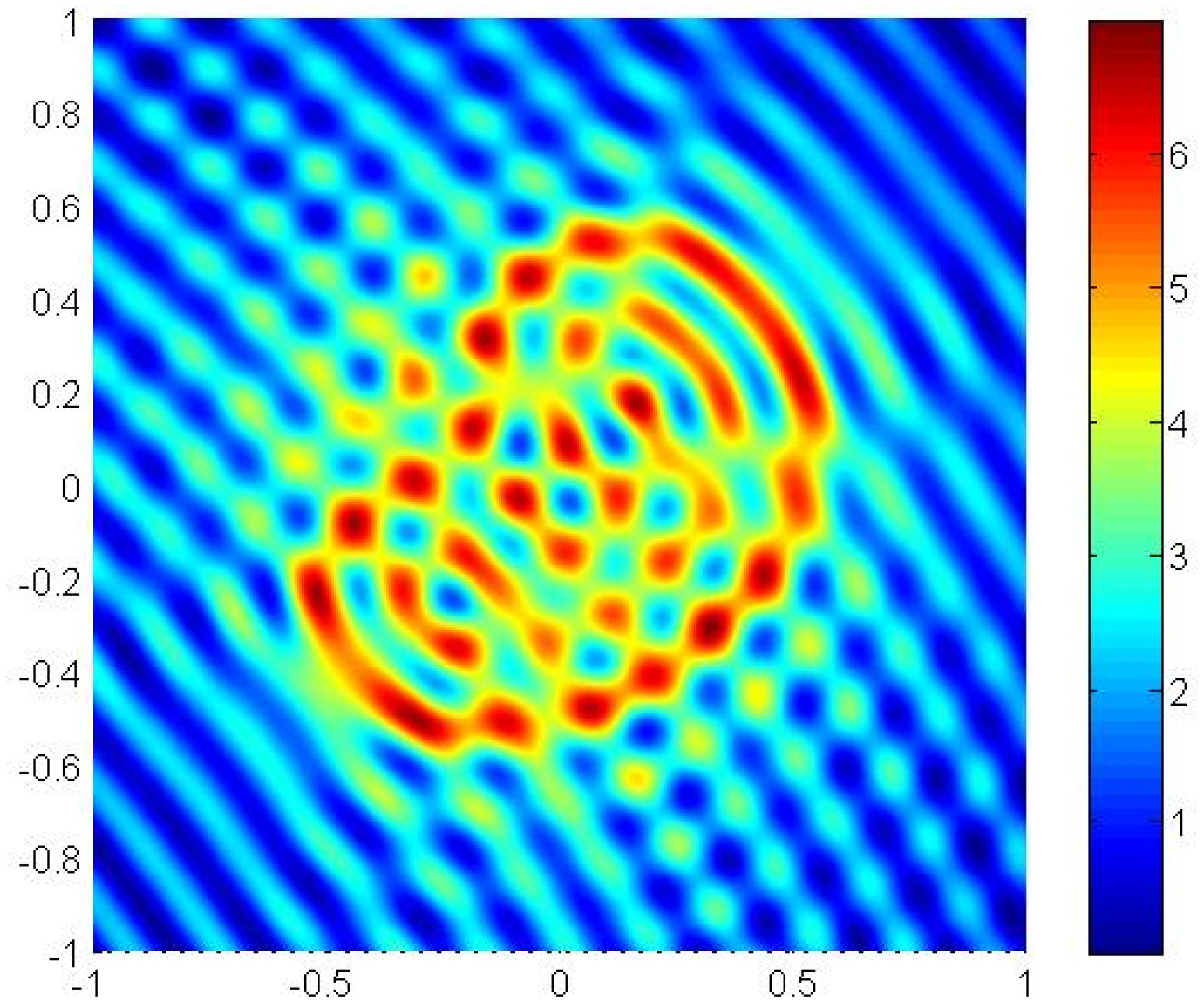}}
  \subfigure[\textbf{No noise, k=20}]{
    \includegraphics[width=1.5in]{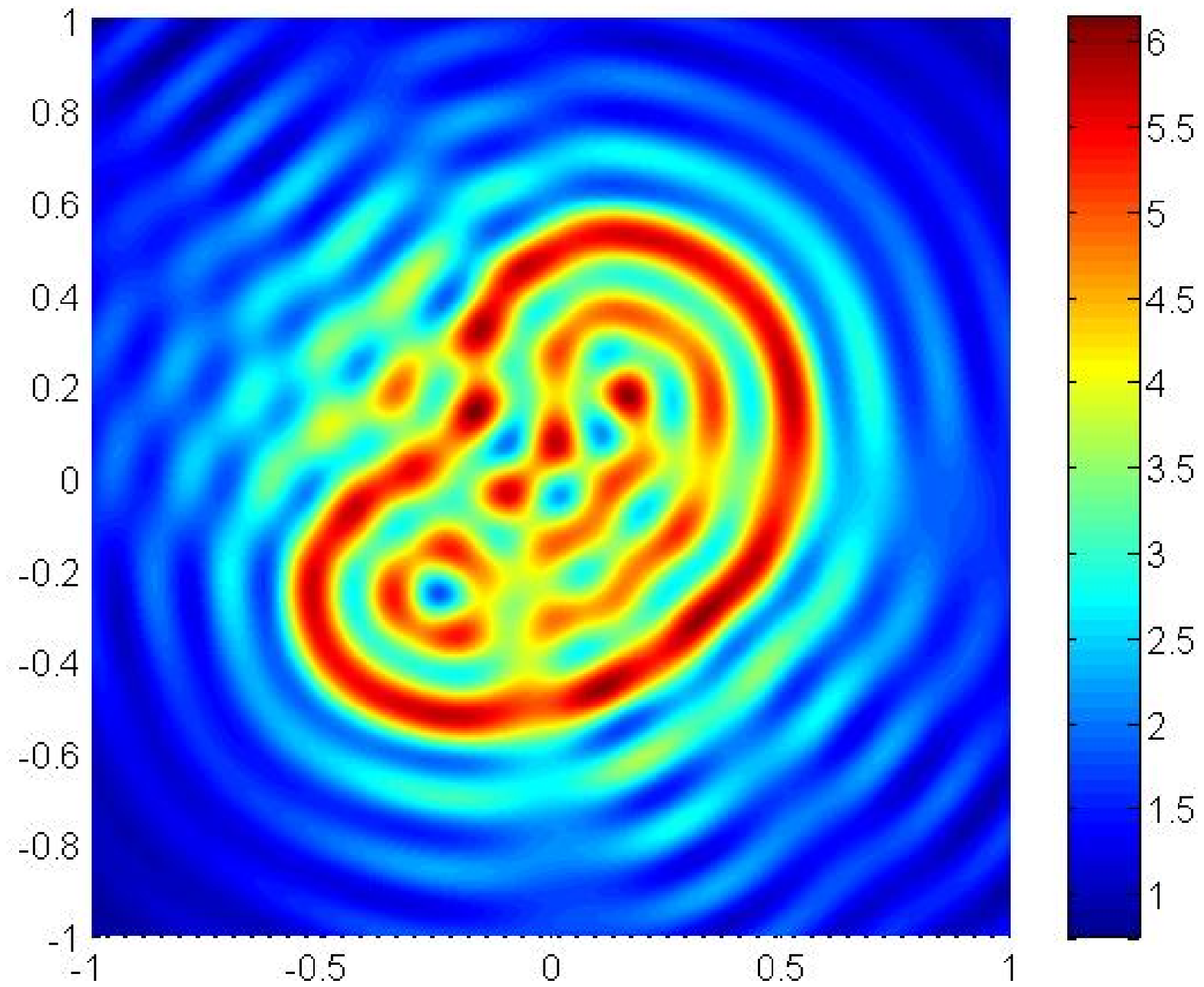}}
  \subfigure[\textbf{10\% noise, k=20}]{
    \includegraphics[width=1.5in]{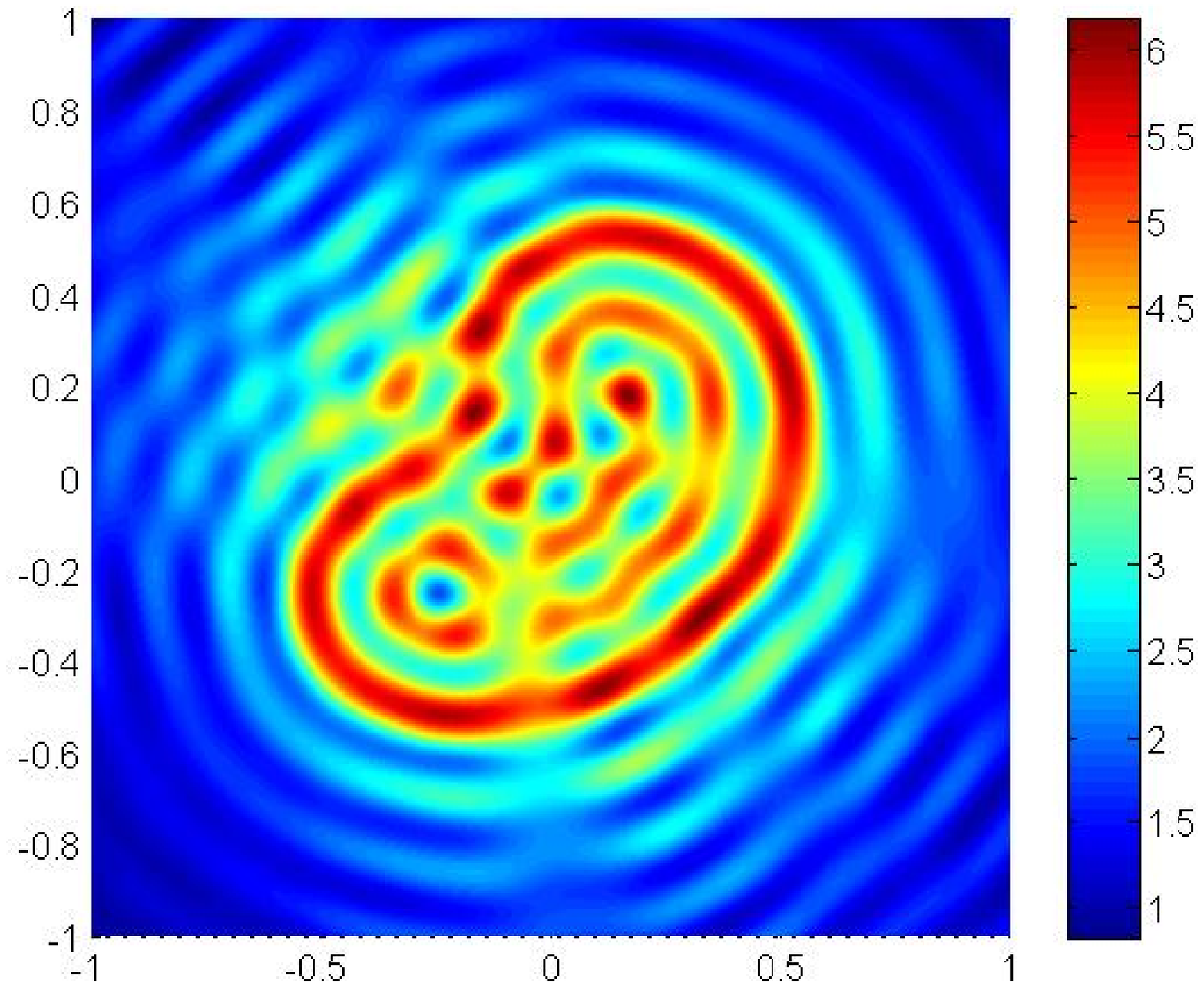}}
  \subfigure[\textbf{20\% noise, k=20}]{
    \includegraphics[width=1.5in]{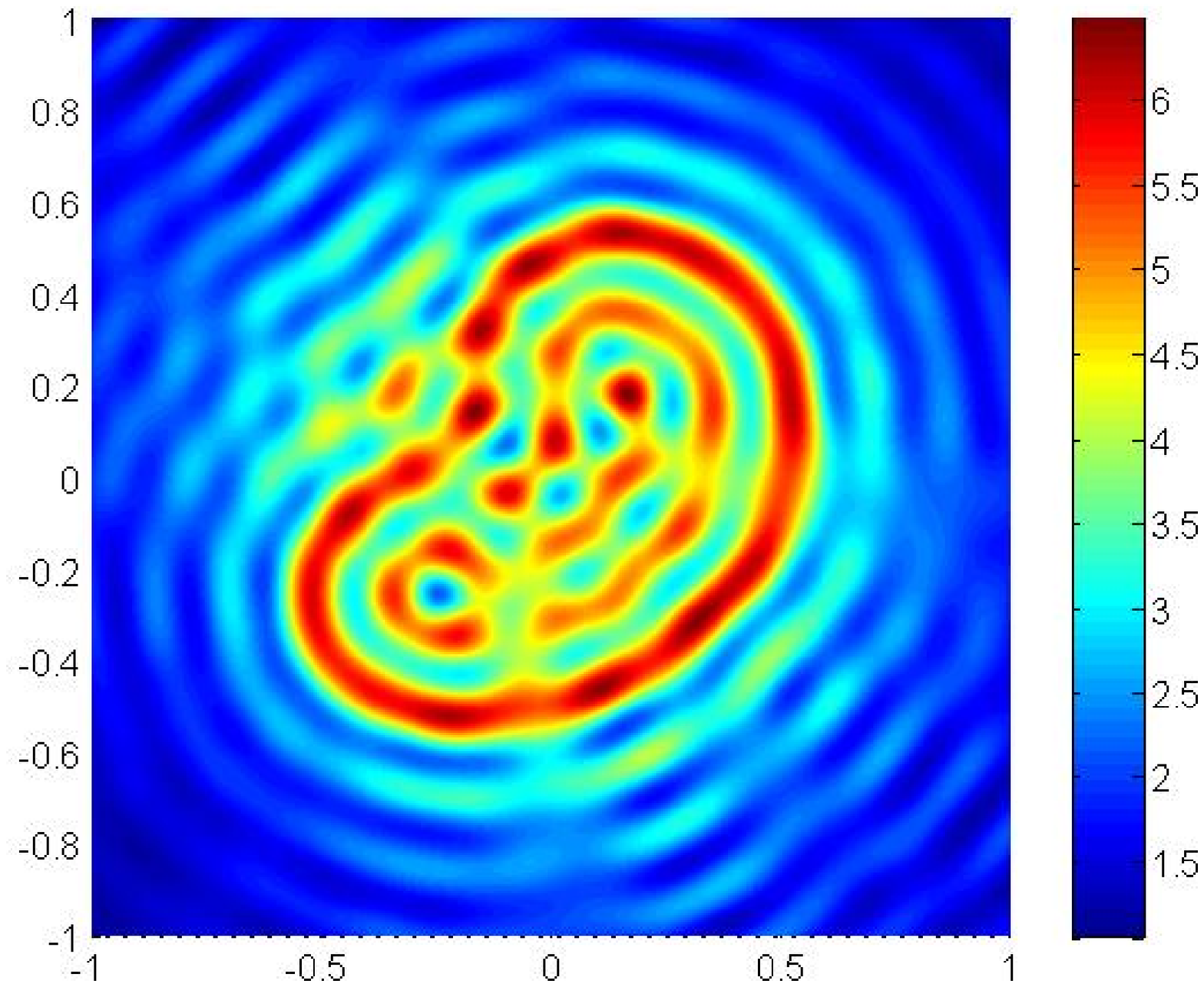}}
\caption{Imaging results of an apple-shaped, sound-soft obstacle given by Algorithm \ref{al1}
with phaseless data (top row) and
by the imaging algorithm with $I^A_F(z)$ in \cite{P10} with full data (bottom row), respectively.
}\label{fig2}
\end{figure}

\textbf{Example 2: Reconstruction of an impedance obstacle.}

We now consider a kite-shaped, impedance obstacle. The sampling region is assumed to be $[-3,3]\times[-3,3]$.
We choose $z_0=(12,0)^T$ for our algorithm. The wave number is chosen to be $k=20$.
Figure \ref{fig3} presents the exact curve, the imaging results of $I_{z_0}(z)$
and $I^A_F(z)$ from the measured data without noise, with $10\%$ noise and with $20\%$ noise, respectively,
for the case when the impedance function $\rho=0$ (i.e., the Neumann boundary condition).
Figure \ref{fig4} presents the exact curve, the imaging results of $I^A_{z_0}(z)$ and $I^A_F(z)$ from the
measured data without noise, with $10\%$ noise and with $20\%$ noise, respectively, for the case
when the impedance function $\rho(x(t))=2+0.5\sin(t),t\in[0,2\pi]$.
The reconstruction results in Figures \ref{fig3} and \ref{fig4} are satisfactory and similar to those
in Figure \ref{fig1}. Note that in Figure \ref{fig3} the imaging function $I^A_{z_0}(z)$
takes a small value on the boundary of the obstacle $D$.
This is consistent with the discussion in Remark \ref{rem1}.

\begin{figure}[htbp]
  \centering
  \subfigure[\textbf{Exact curve}]{
    \includegraphics[width=1.5in]{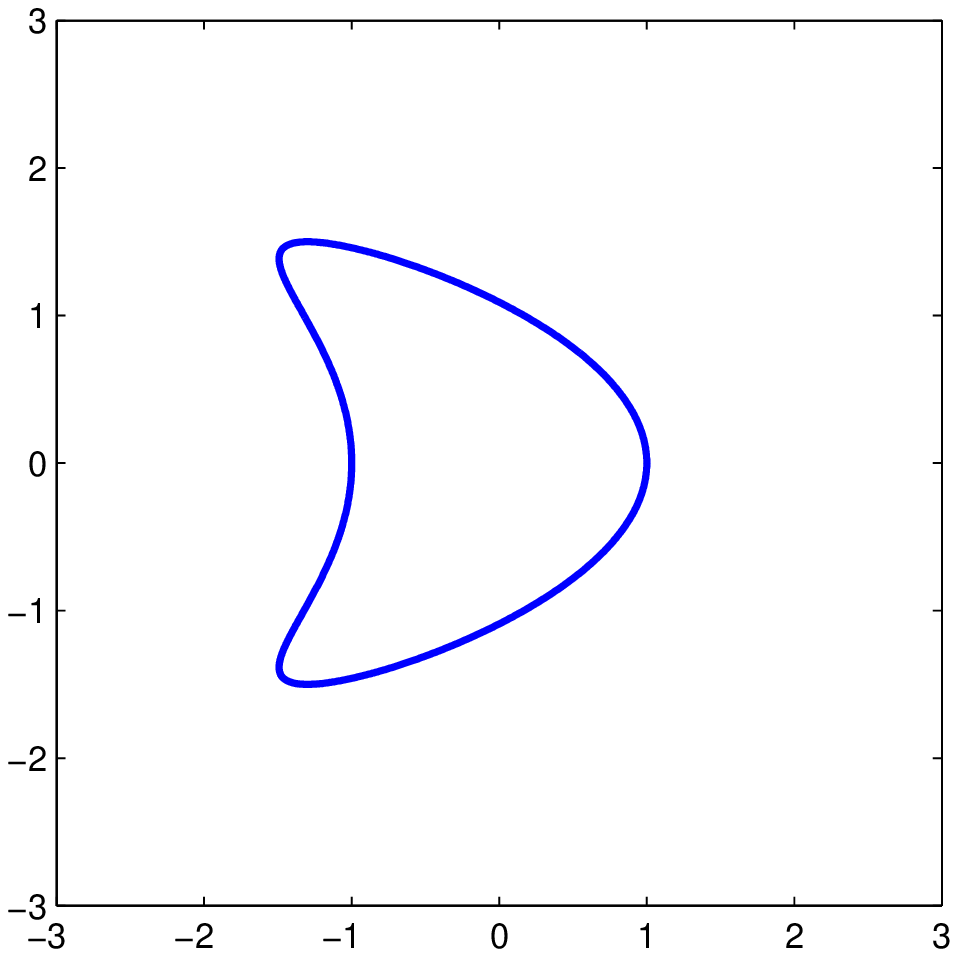}}
  \subfigure[\textbf{No noise, k=20}]{
    \includegraphics[width=1.5in]{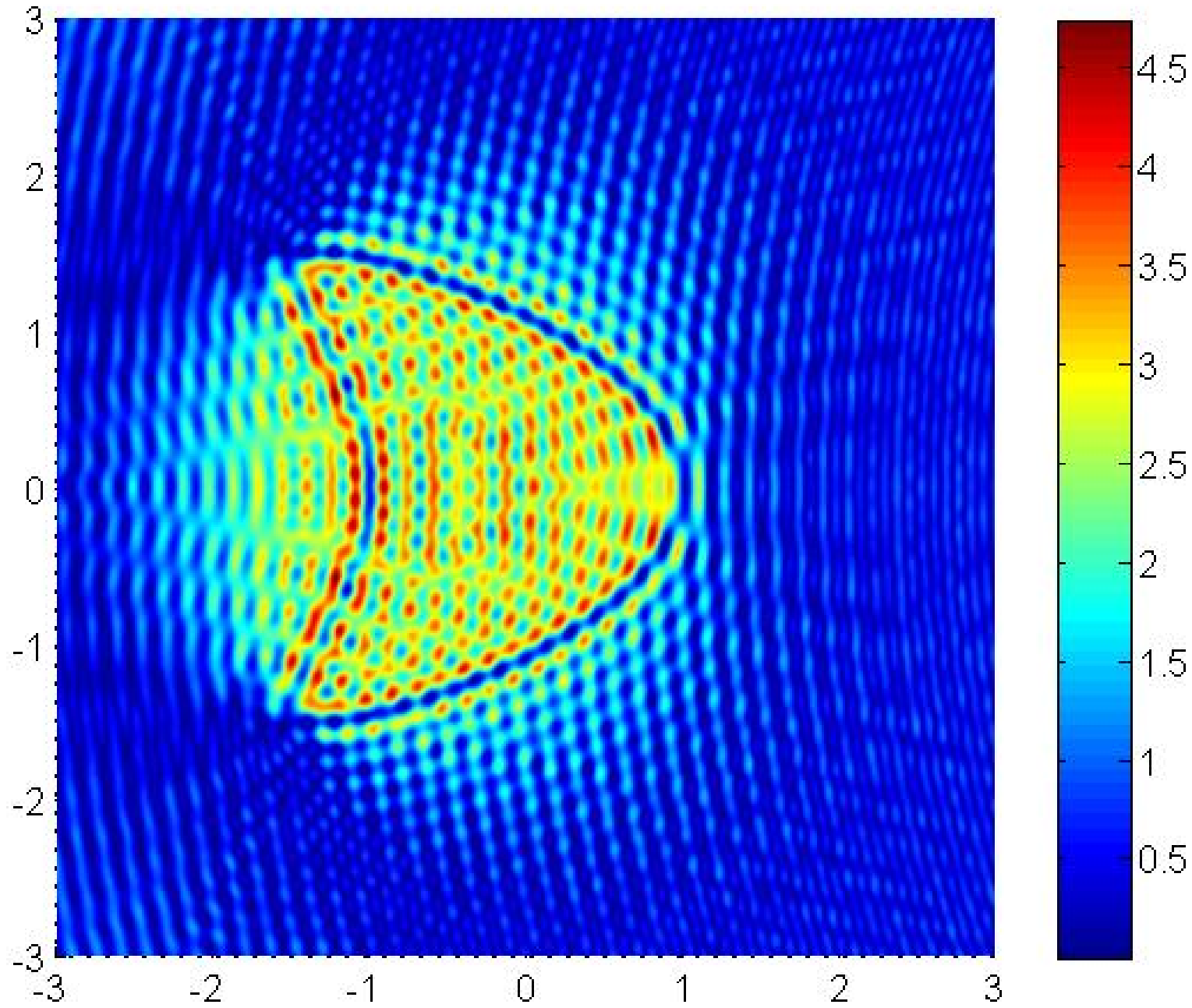}}
  \subfigure[\textbf{10\% noise, k=20}]{
    \includegraphics[width=1.5in]{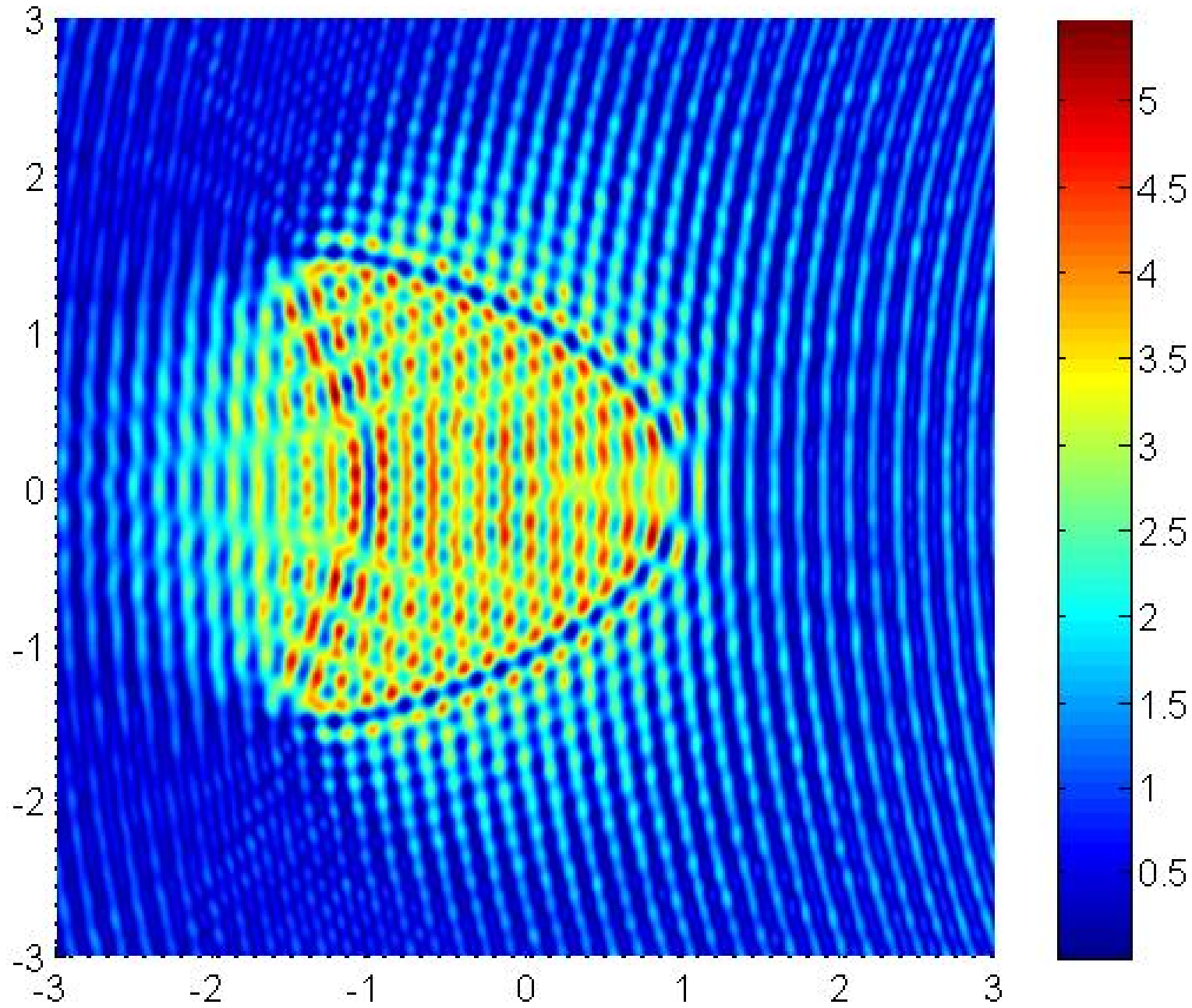}}
  \subfigure[\textbf{20\% noise, k=20}]{
    \includegraphics[width=1.5in]{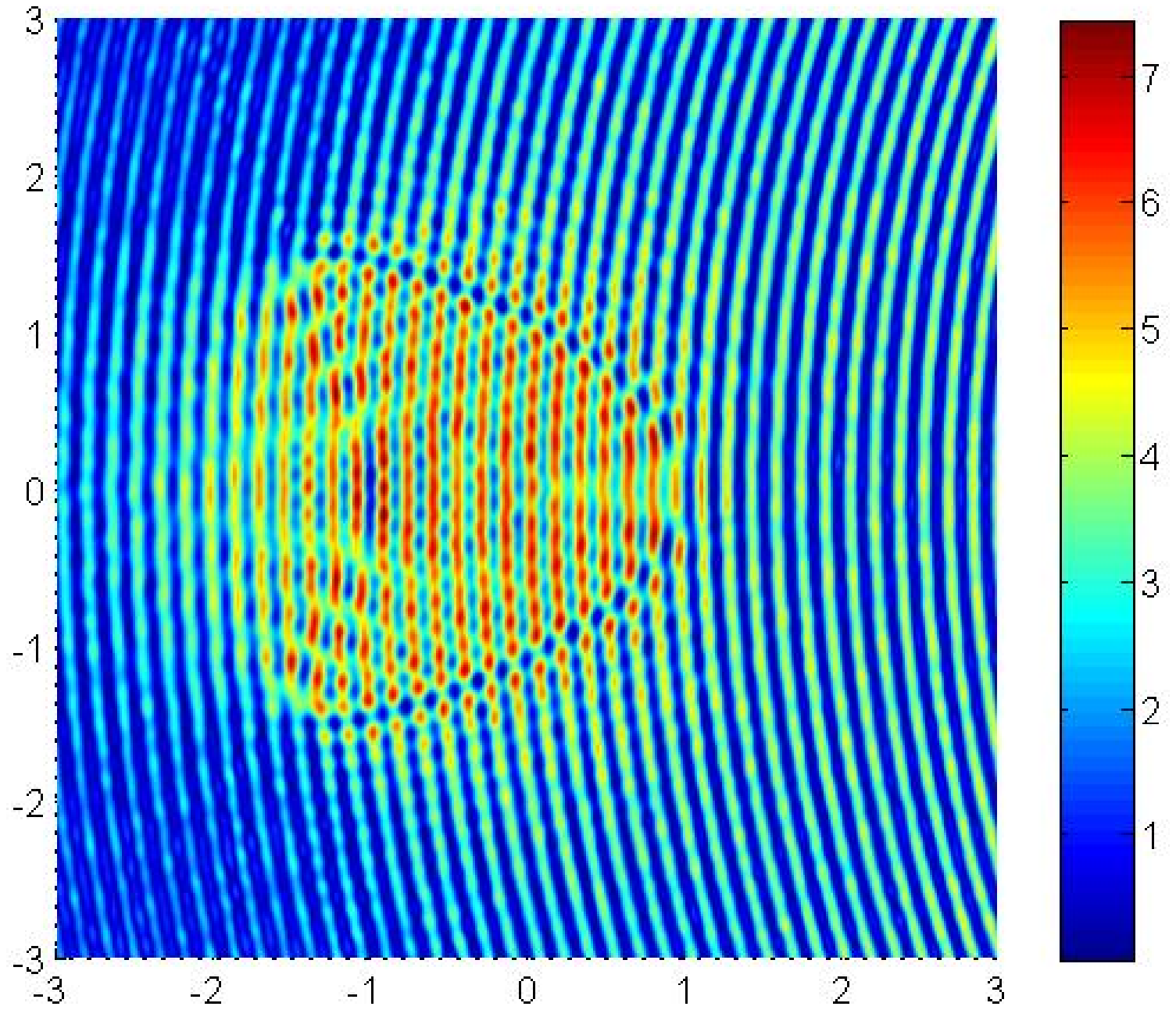}}
  \subfigure[\textbf{No noise, k=20}]{
    \includegraphics[width=1.5in]{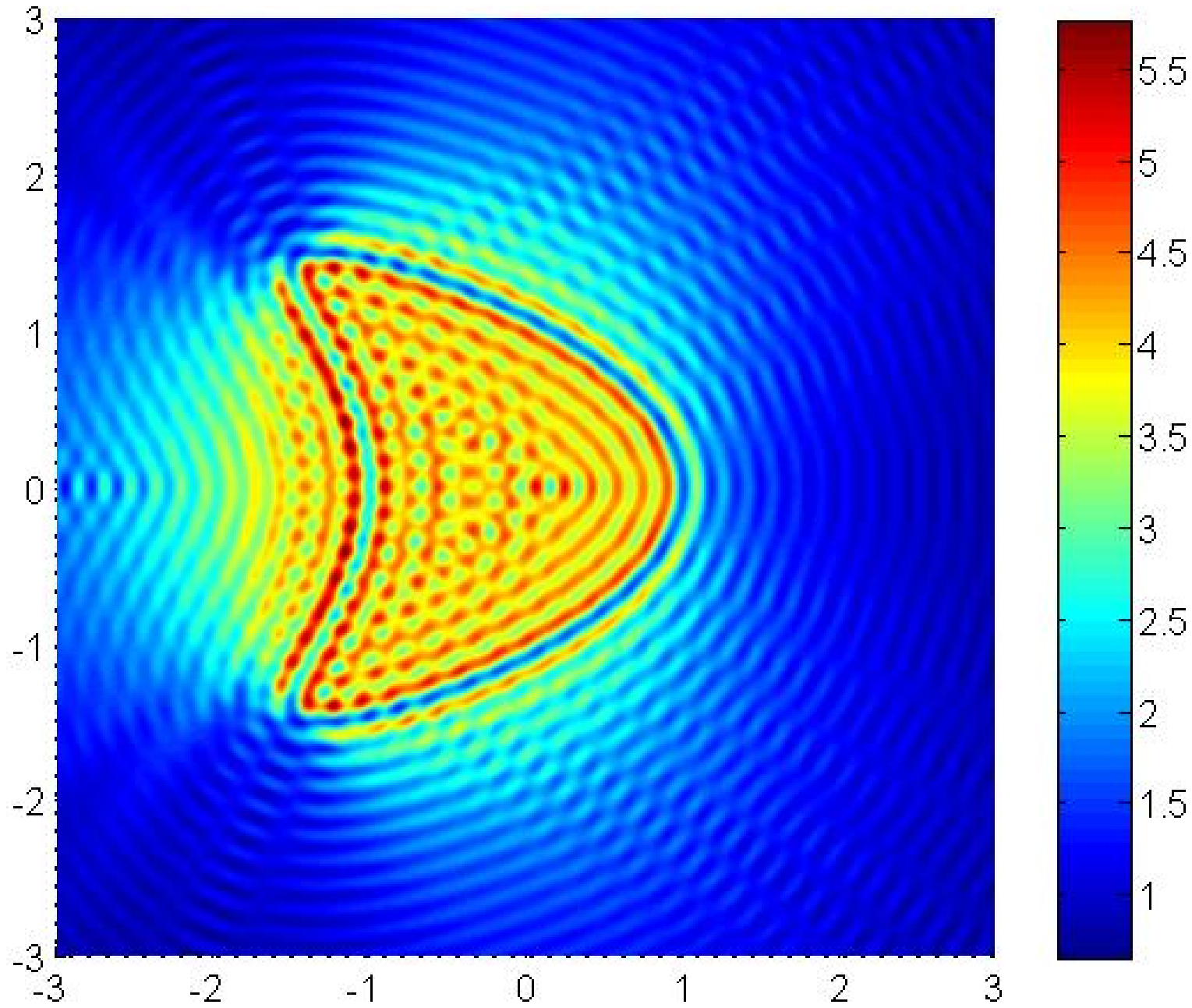}}
  \subfigure[\textbf{10\% noise, k=20}]{
    \includegraphics[width=1.5in]{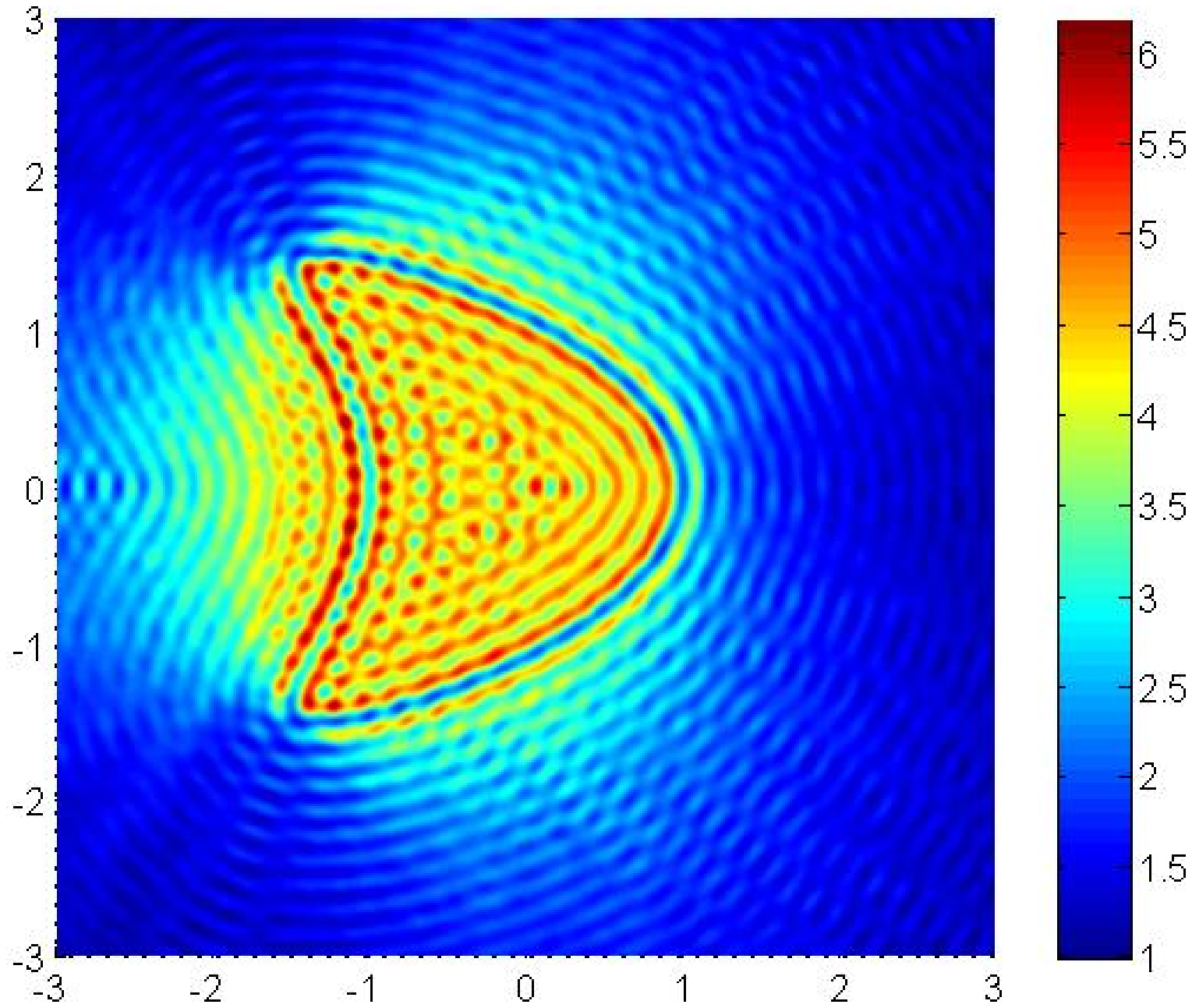}}
  \subfigure[\textbf{20\% noise, k=20}]{
    \includegraphics[width=1.5in]{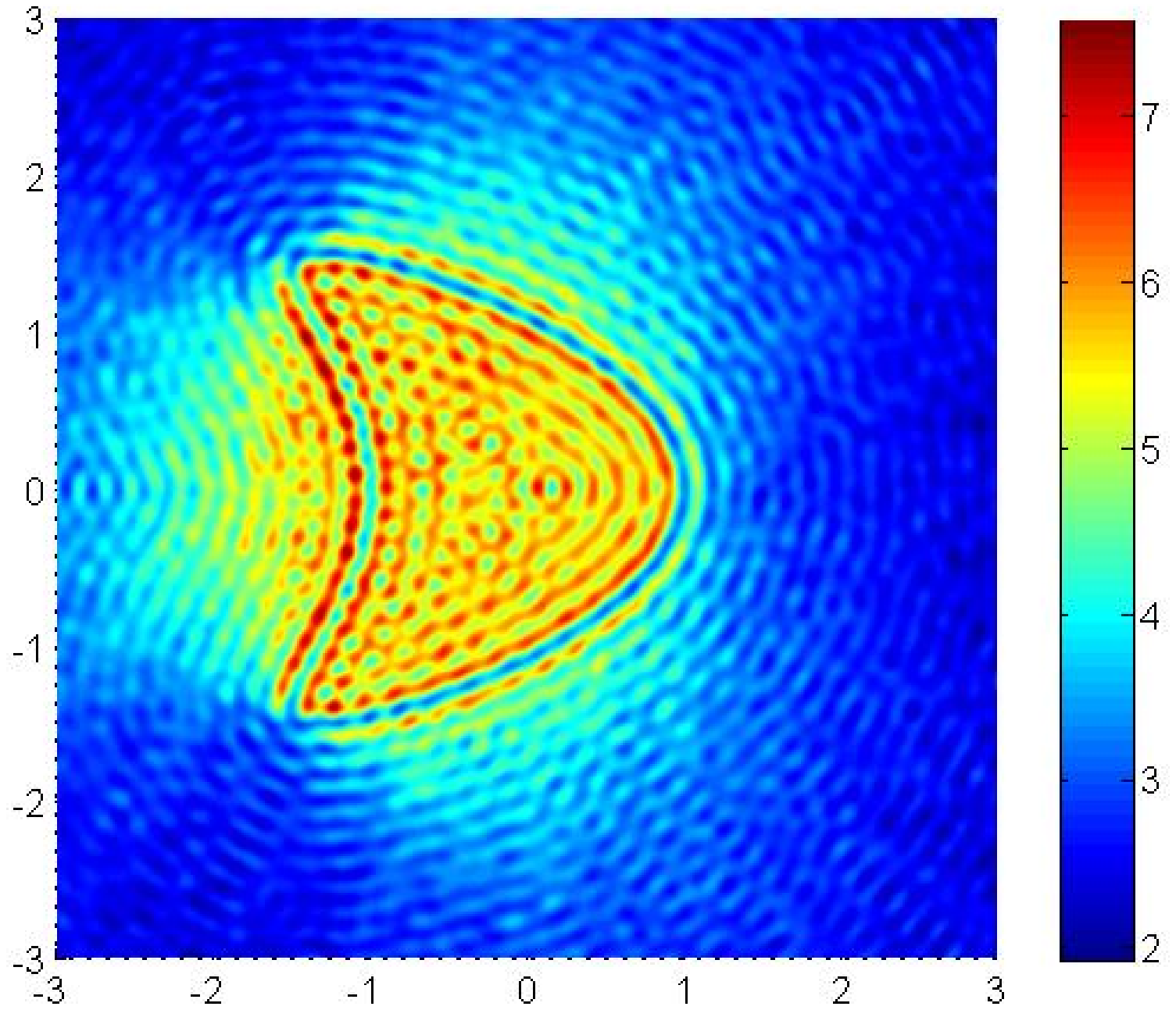}}
\caption{Imaging results of a kite-shaped, sound-hard obstacle given by Algorithm \ref{al1}
with phaseless data (top row) and
by the imaging algorithm with $I^A_F(z)$ in \cite{P10} with full data (bottom row), respectively.
}\label{fig3}
\end{figure}

\begin{figure}[htbp]
  \centering
  \subfigure[\textbf{Exact curve}]{
    \includegraphics[width=1.5in]{pic/example2/kite.eps}}
  \subfigure[\textbf{No noise, k=20}]{
    \includegraphics[width=1.5in]{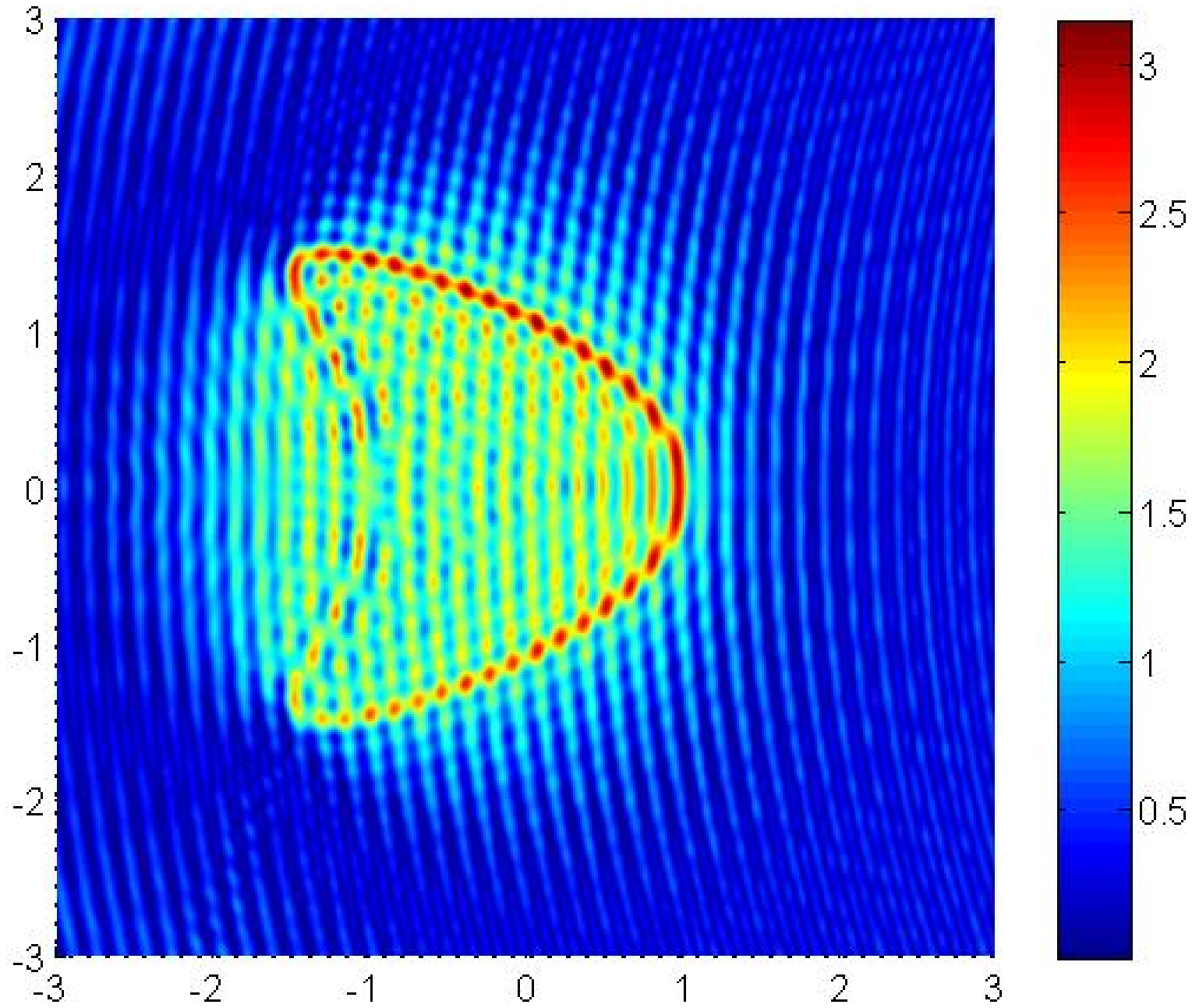}}
  \subfigure[\textbf{10\% noise, k=20}]{
    \includegraphics[width=1.5in]{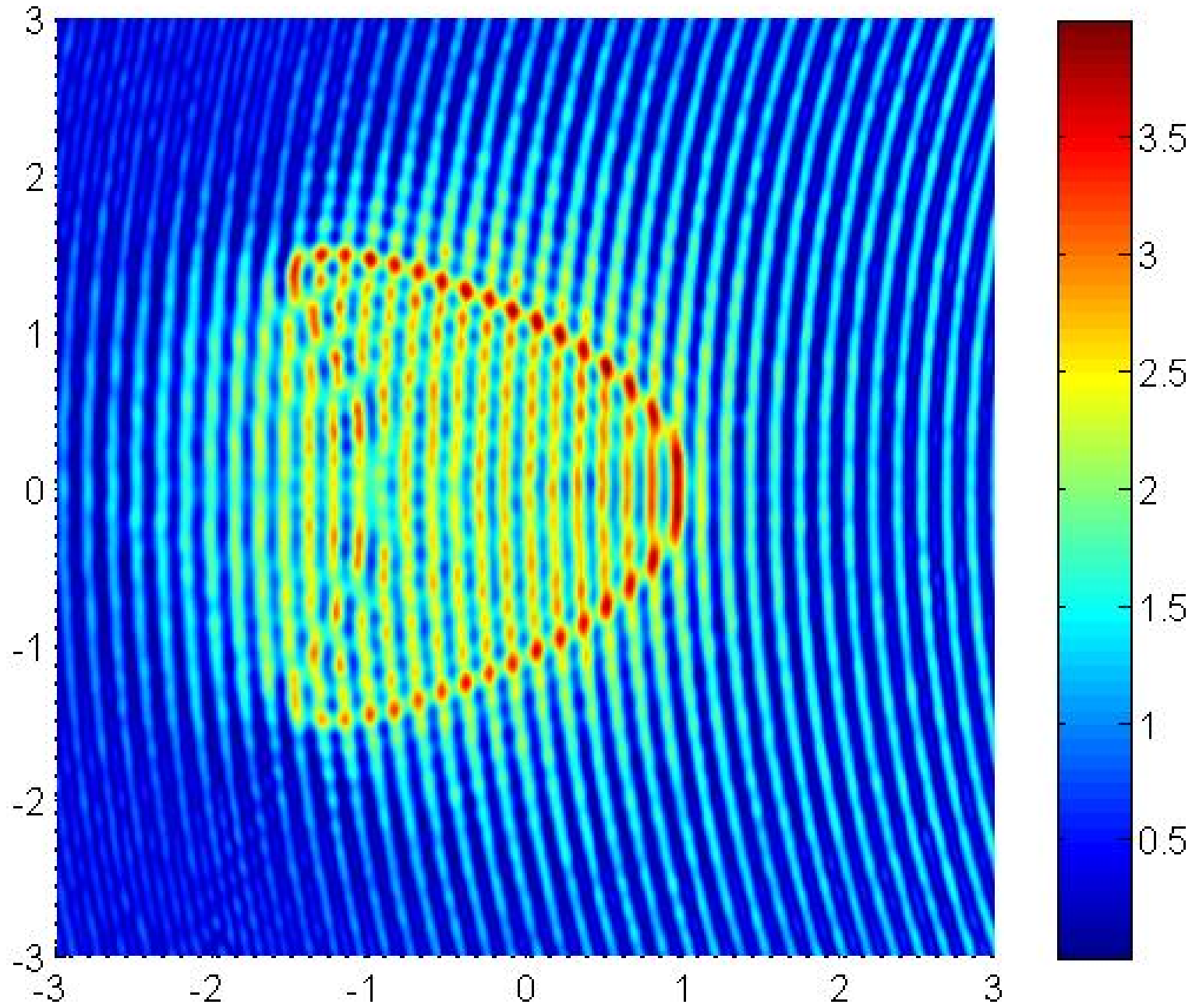}}
  \subfigure[\textbf{20\% noise, k=20}]{
    \includegraphics[width=1.5in]{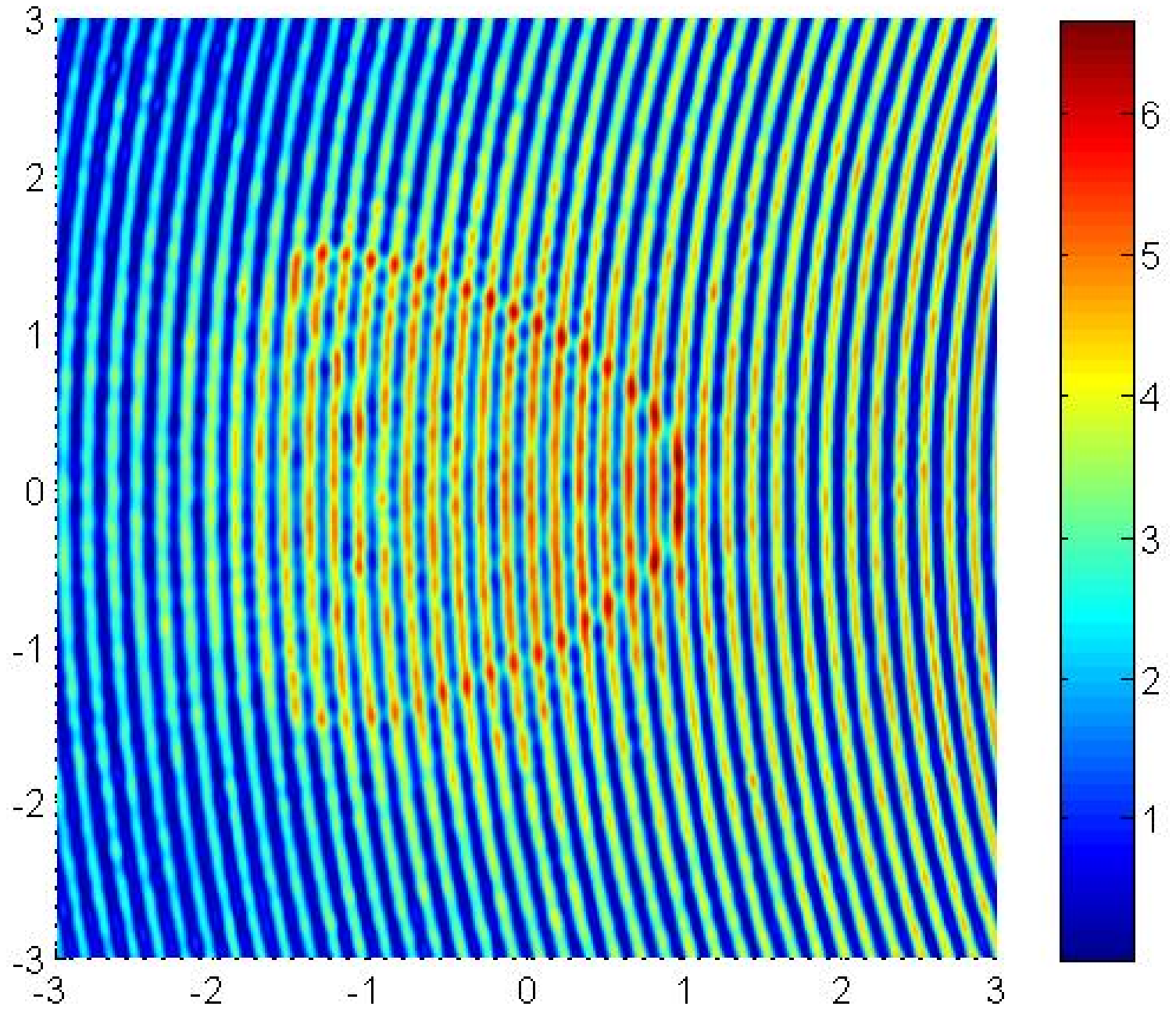}}
  \subfigure[\textbf{No noise, k=20}]{
    \includegraphics[width=1.5in]{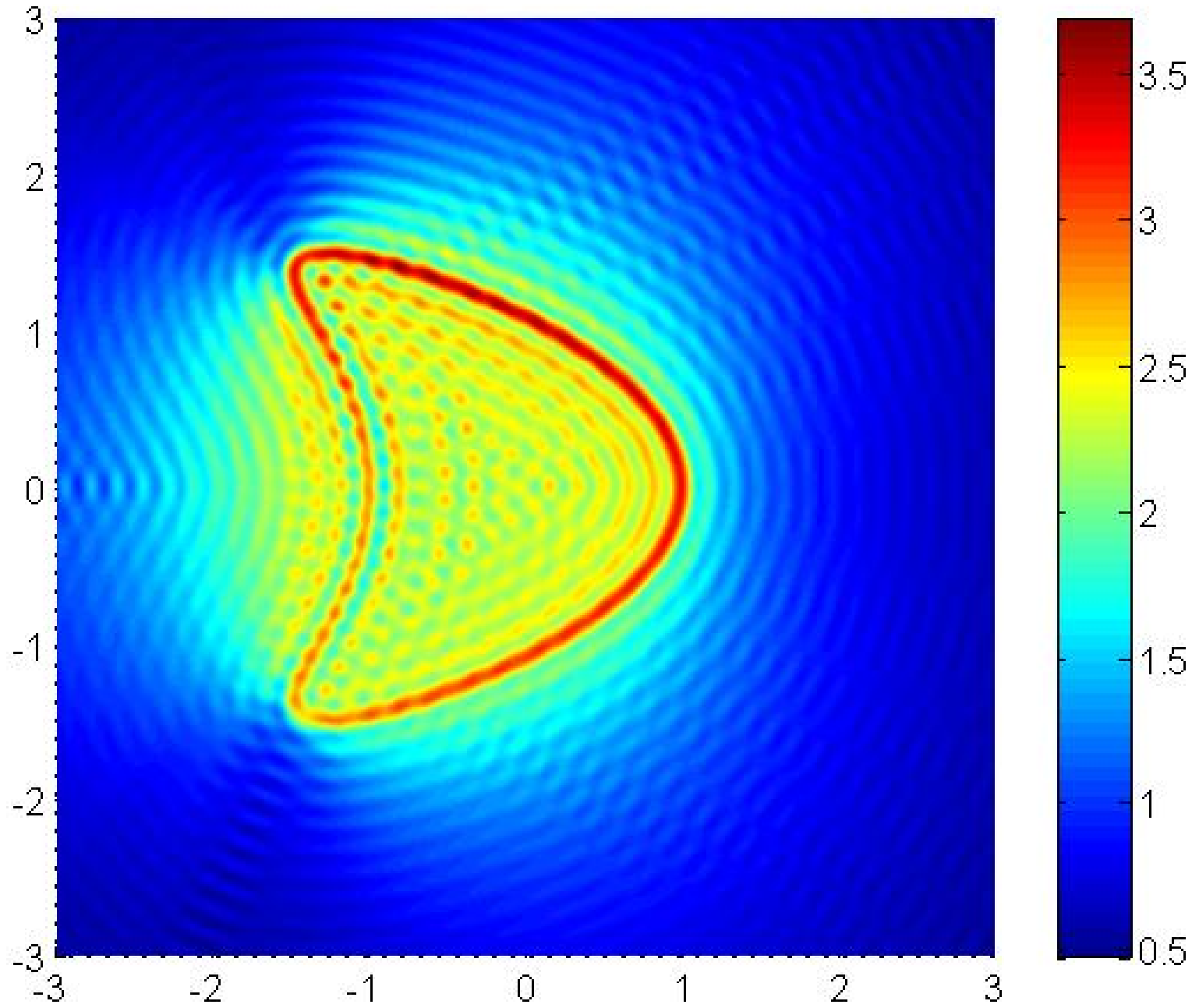}}
  \subfigure[\textbf{10\% noise, k=20}]{
    \includegraphics[width=1.5in]{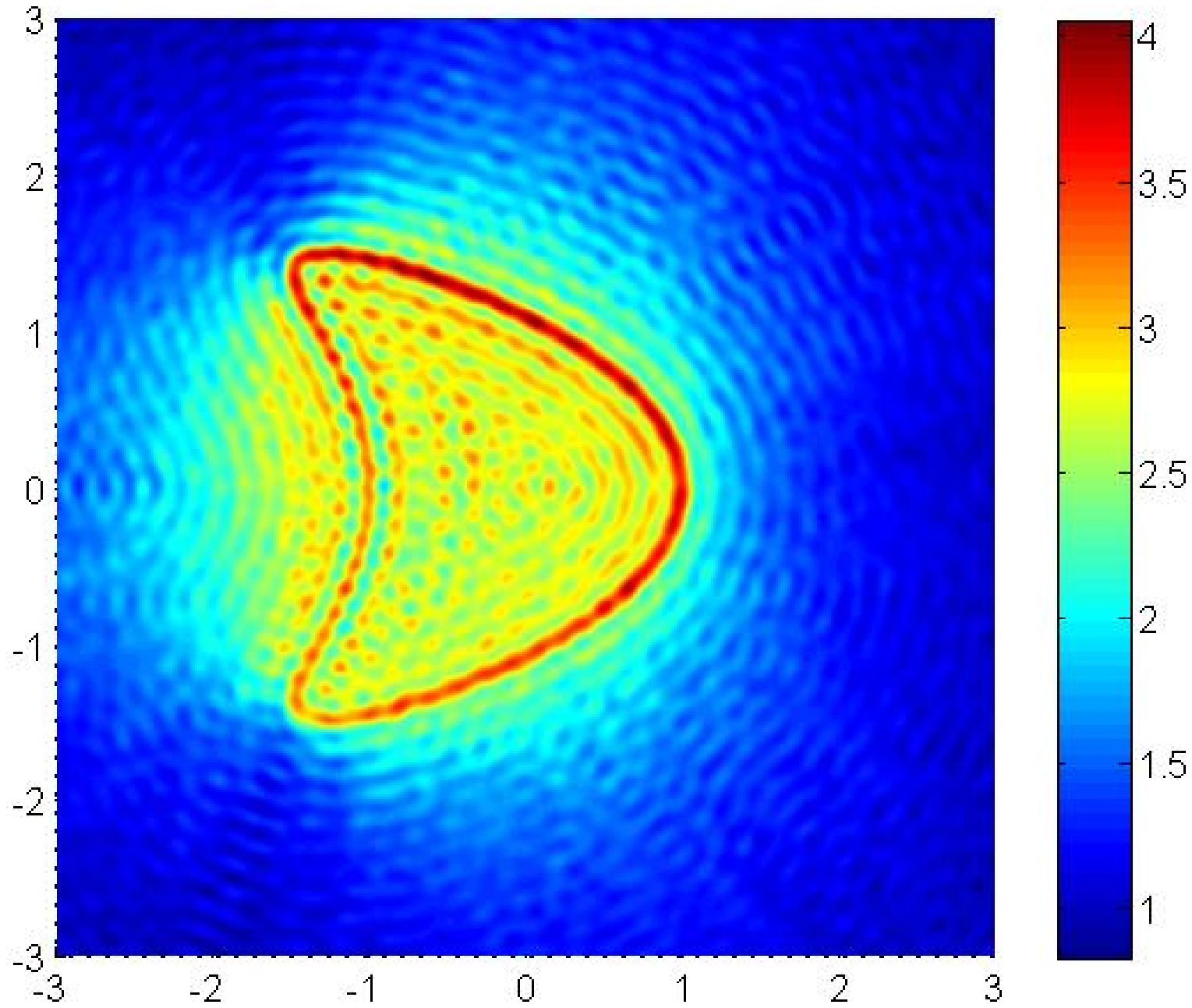}}
  \subfigure[\textbf{20\% noise, k=20}]{
    \includegraphics[width=1.5in]{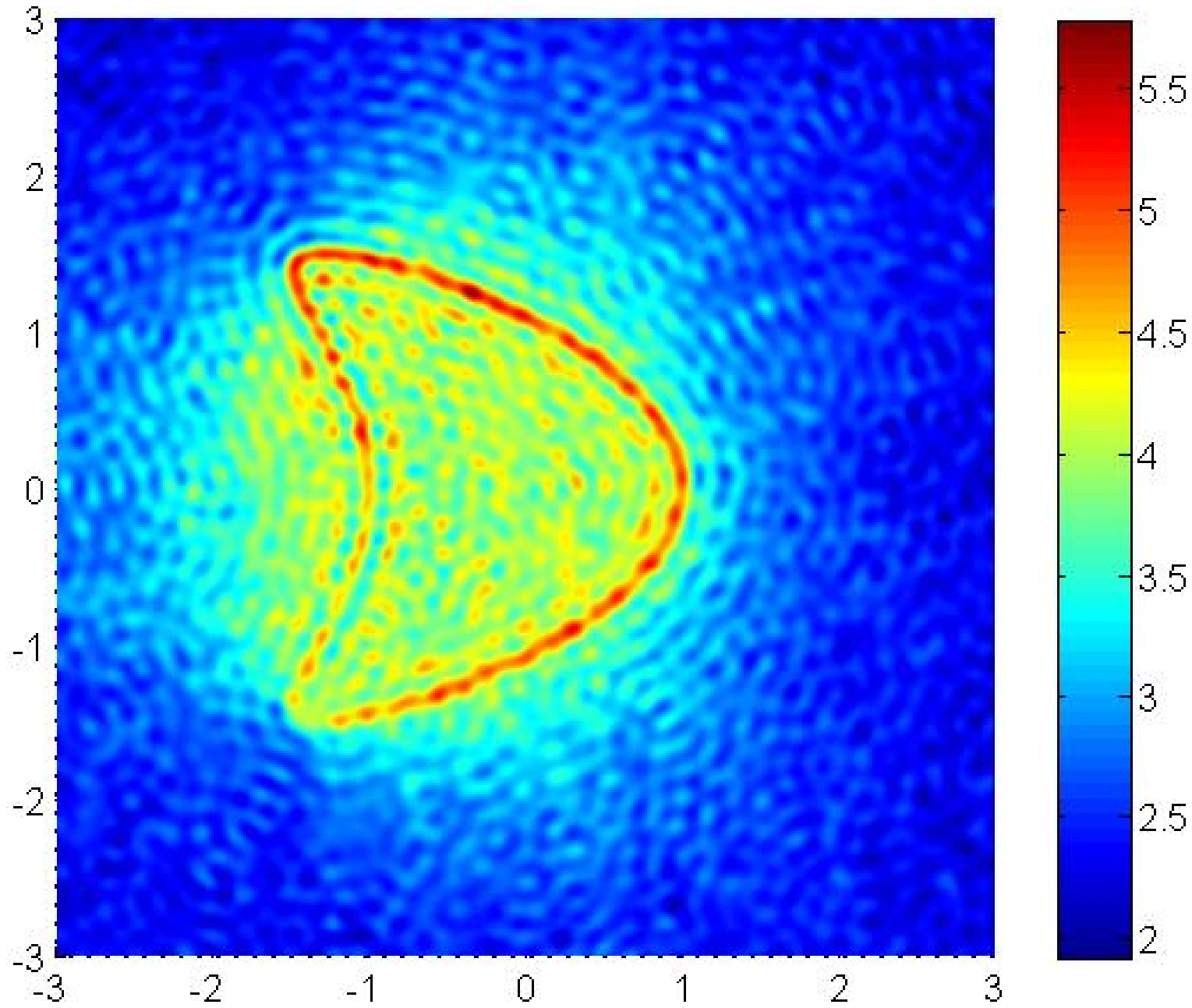}}
\caption{Imaging results of a kite-shaped, impedance obstacle with the impedance function $\rho(x(t))=2+0.5\sin(t),t\in[0,2\pi]$ given by Algorithm \ref{al1} with phaseless data (top row)
and by the imaging algorithm with $I^A_F(z)$ in \cite{P10} with full data (bottom row),
respectively.
}\label{fig4}
\end{figure}

\textbf{Example 3: Reconstruction of a penetrable obstacle.}

We consider the reconstruction of a rounded triangle-shaped, penetrable obstacle.
The sampling region is assumed to be $[-3,3]\times[-3,3]$.
We take $z_0=(9,9)^T$ for our algorithm. The wave number is chosen to be $k=20$.
Figure \ref{fig5} gives the exact curve, the imaging results of $I^A_{z_0}(z)$ and $I^A_F(z)$ from the
measured data without noise, with $5\%$ noise and with $10\%$ noise, respectively,
for the case when the refractive index $n(x)=4$ and the transmission constant $\la=1$.
Figure \ref{fig6} shows the exact curve, the imaging  results of $I^A_{z_0}(z)$ and $I^A_F(z)$ from the
measured data without noise, with $5\%$ noise and with $10\%$ noise, respectively,
for the case when the refractive index $n(x)=0.64$ and the transmission constant $\la=2$.
It is observed that the reconstructed results for penetrable obstacles are not as good as
those for impenetrable obstacles but are still satisfactory.

\begin{figure}[htbp]
  \centering
  \subfigure[\textbf{Exact curve}]{
    \includegraphics[width=1.5in]{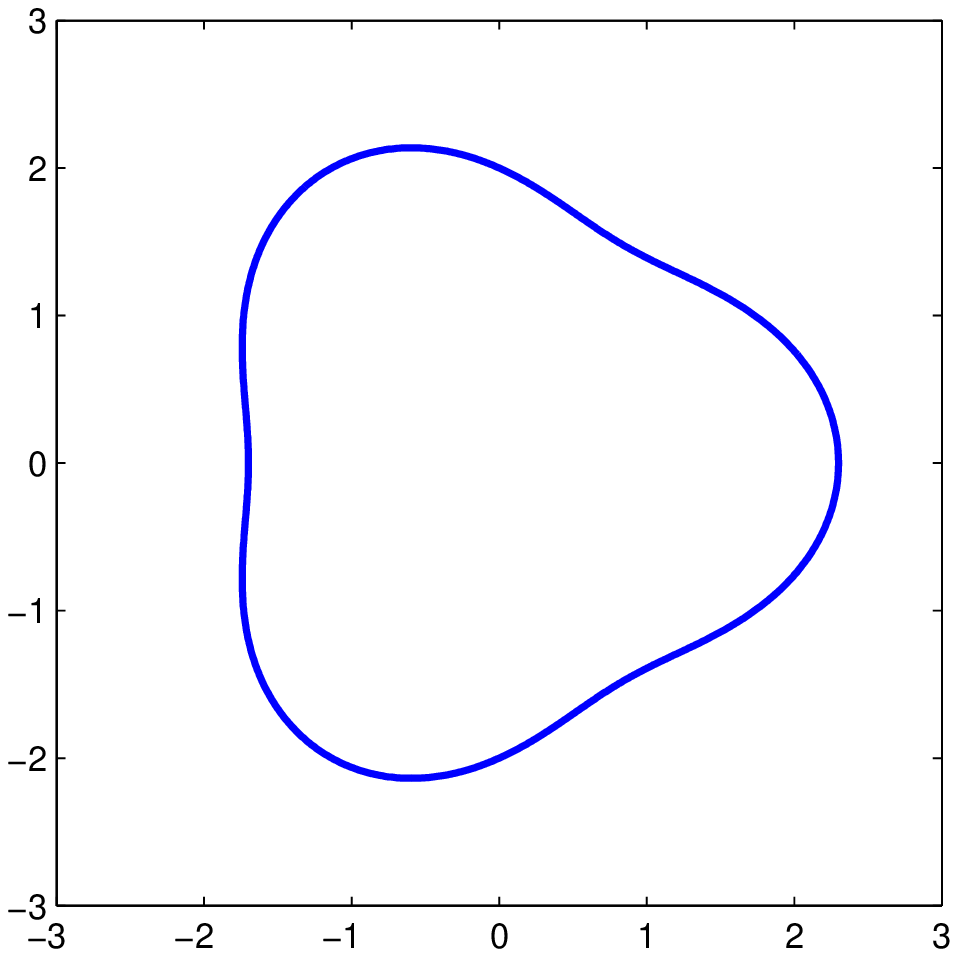}}
  \subfigure[\textbf{No noise, k=20}]{
    \includegraphics[width=1.5in]{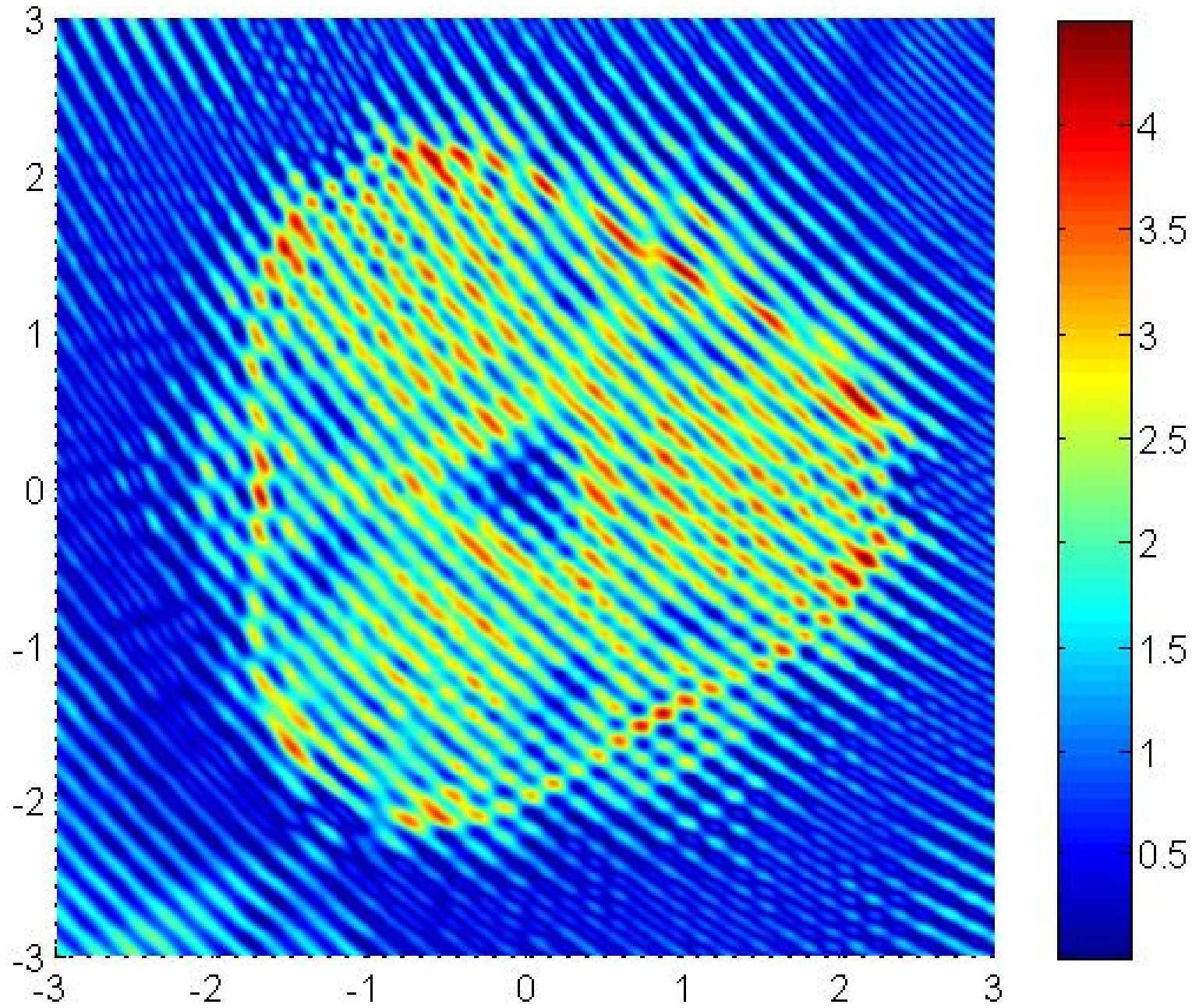}}
  \subfigure[\textbf{5\% noise, k=20}]{
    \includegraphics[width=1.5in]{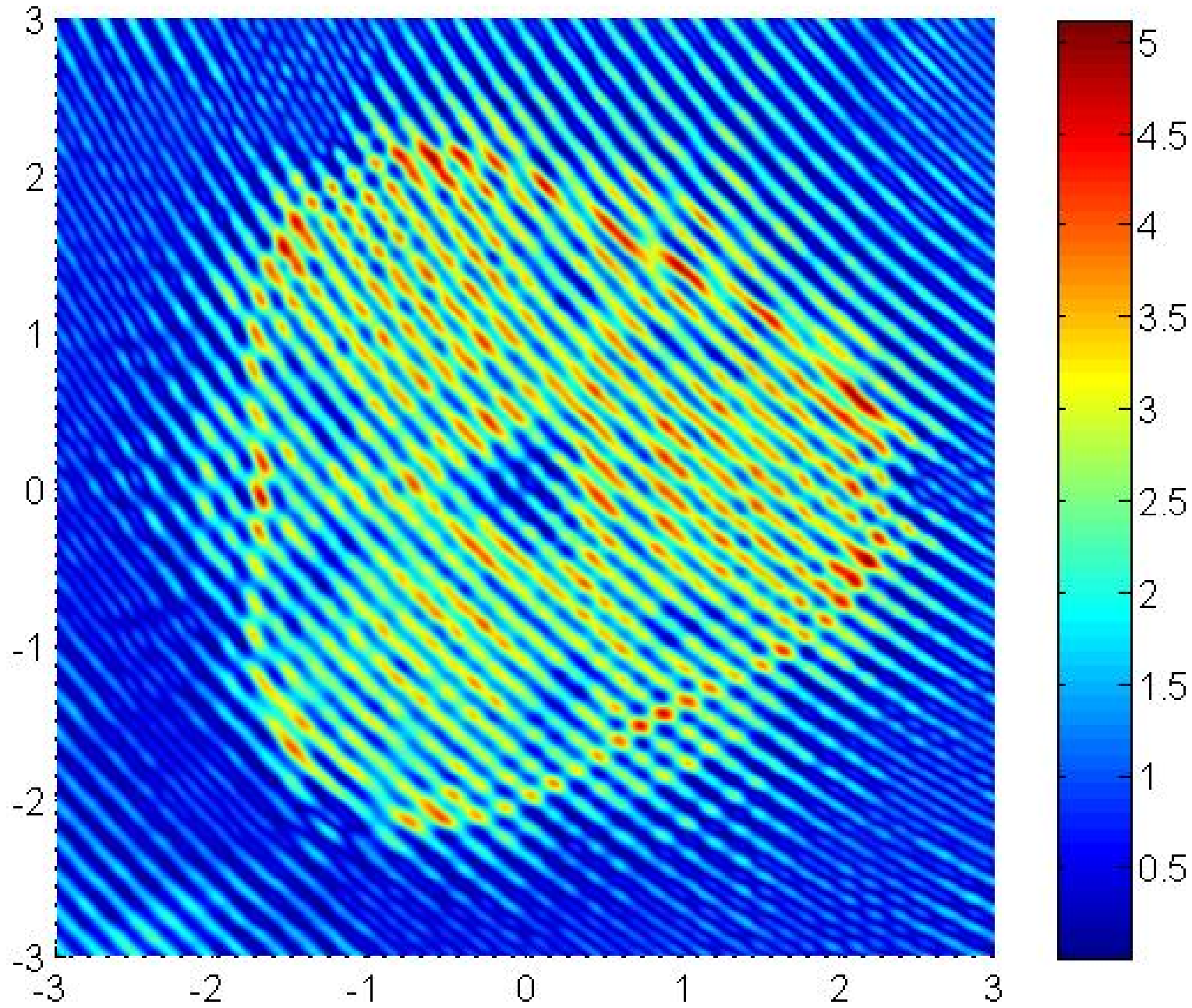}}
  \subfigure[\textbf{10\% noise, k=20}]{
    \includegraphics[width=1.5in]{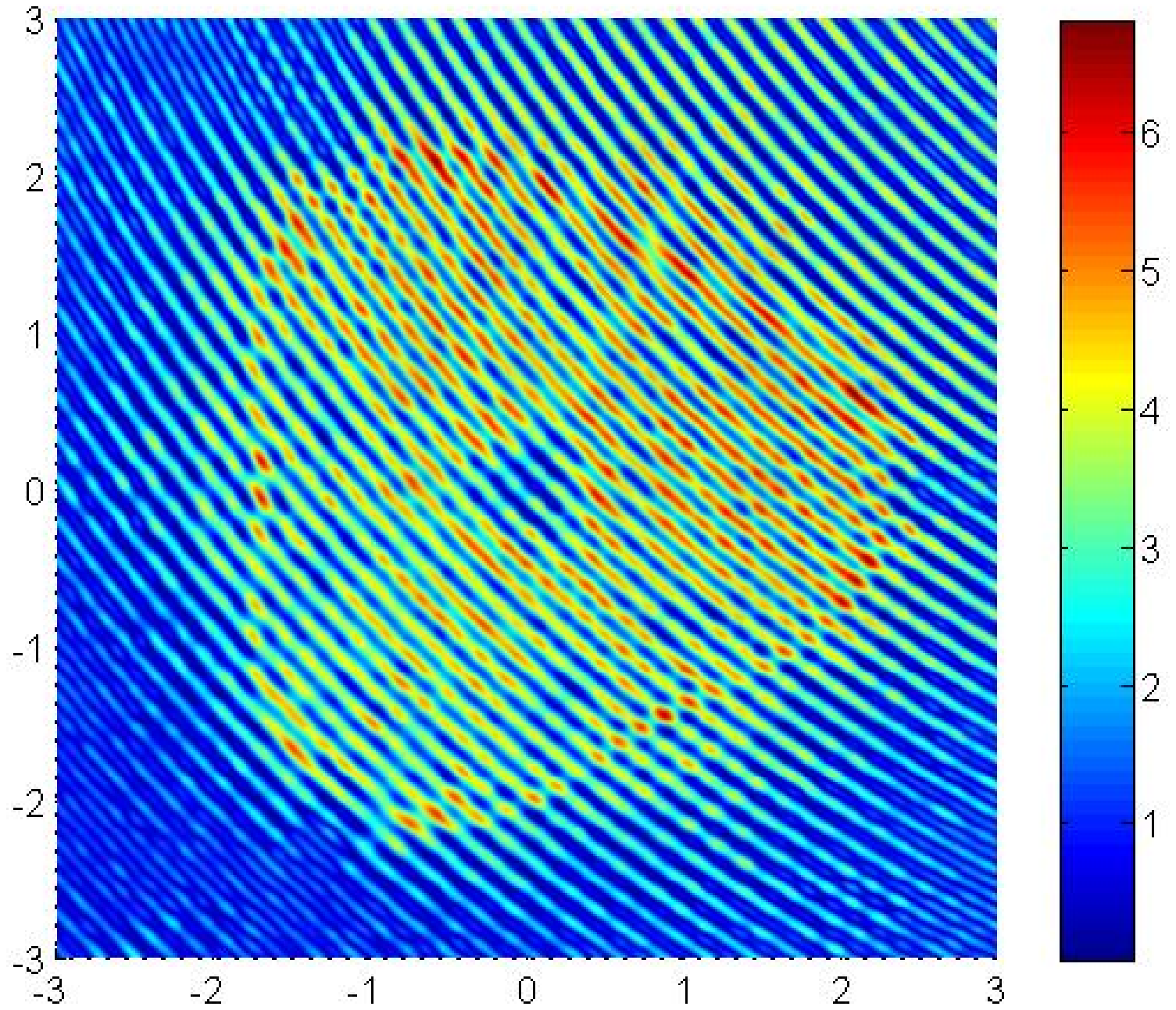}}
  \subfigure[\textbf{No noise, k=20}]{
    \includegraphics[width=1.5in]{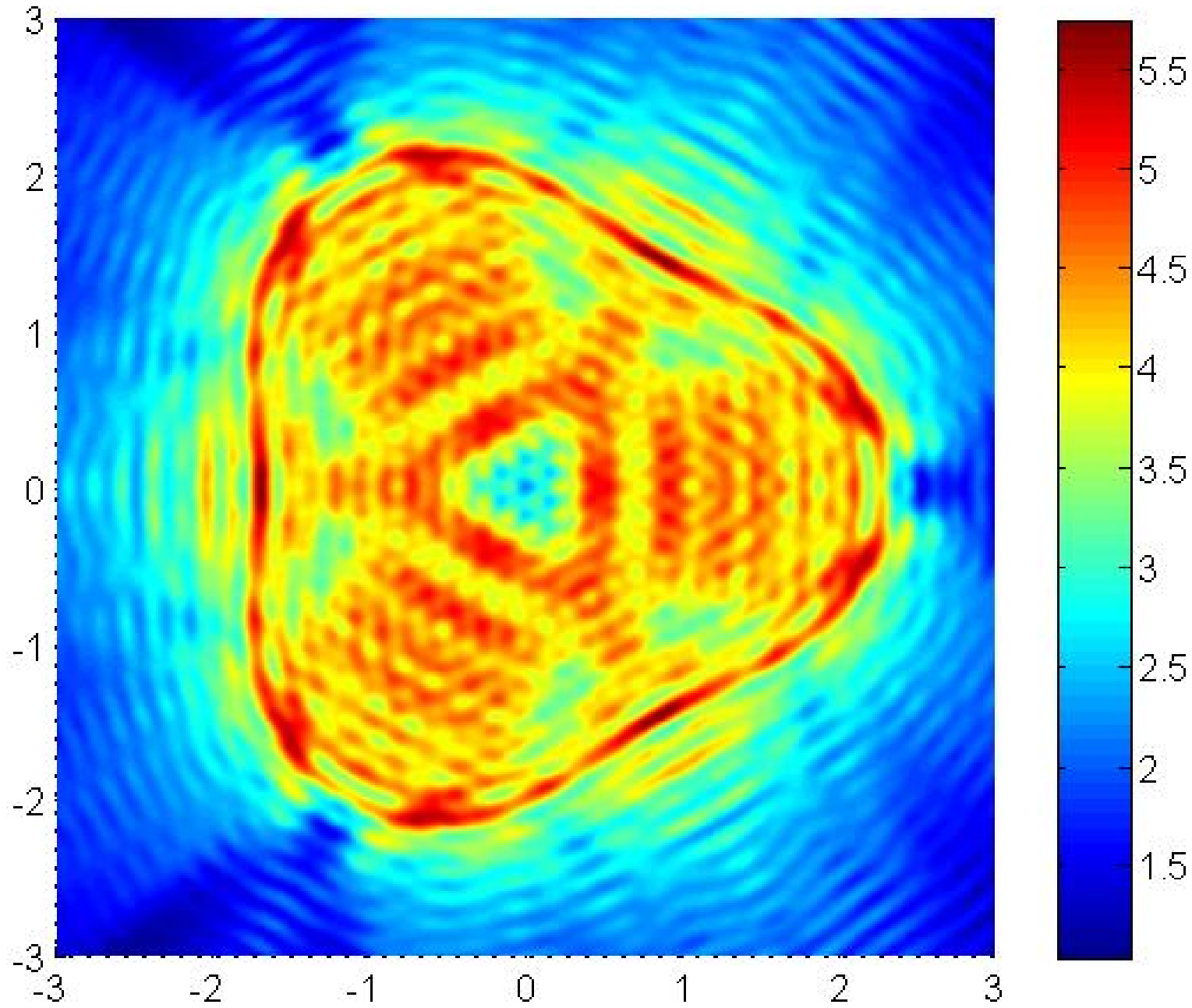}}
  \subfigure[\textbf{5\% noise, k=20}]{
    \includegraphics[width=1.5in]{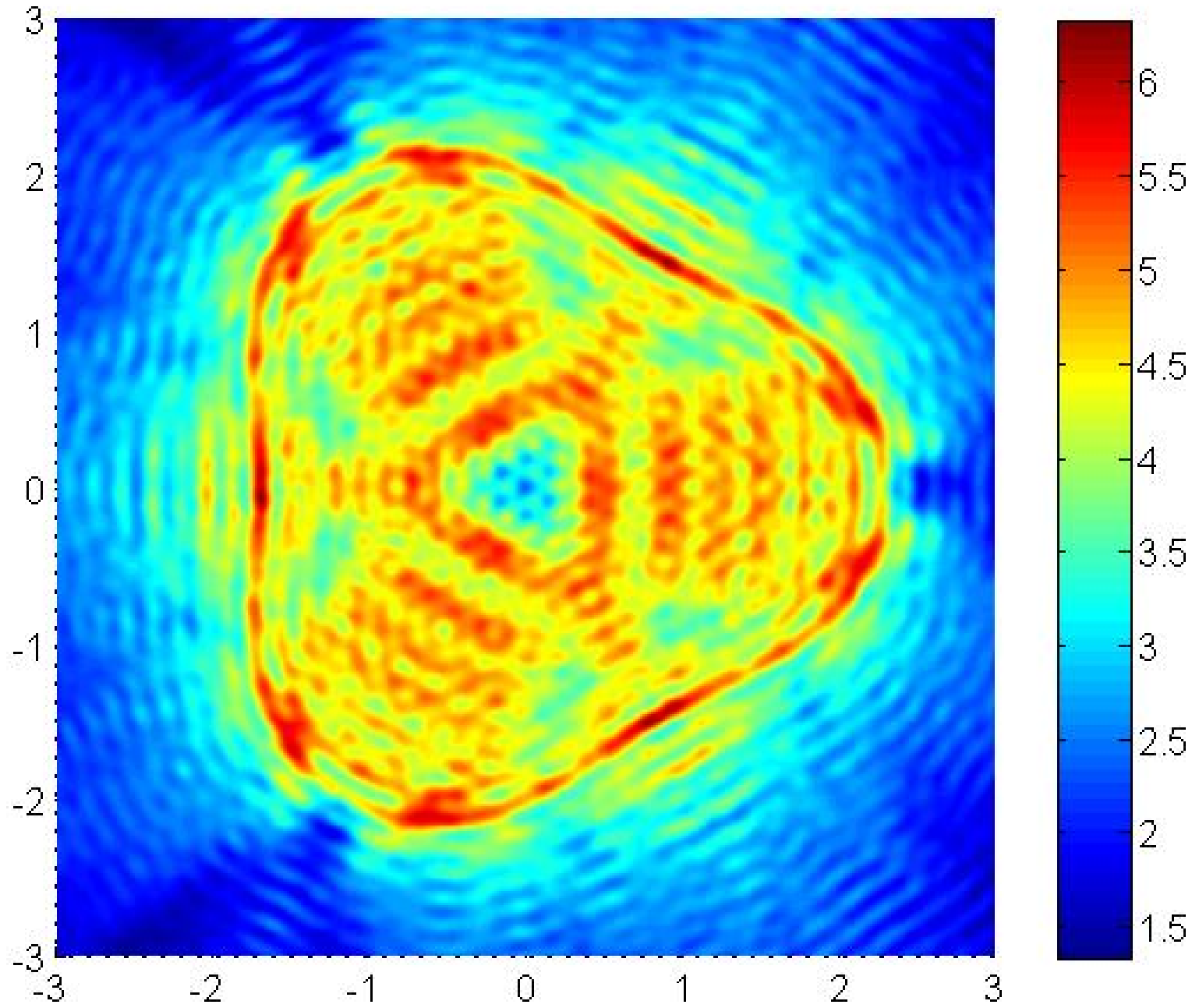}}
  \subfigure[\textbf{10\% noise, k=20}]{
    \includegraphics[width=1.5in]{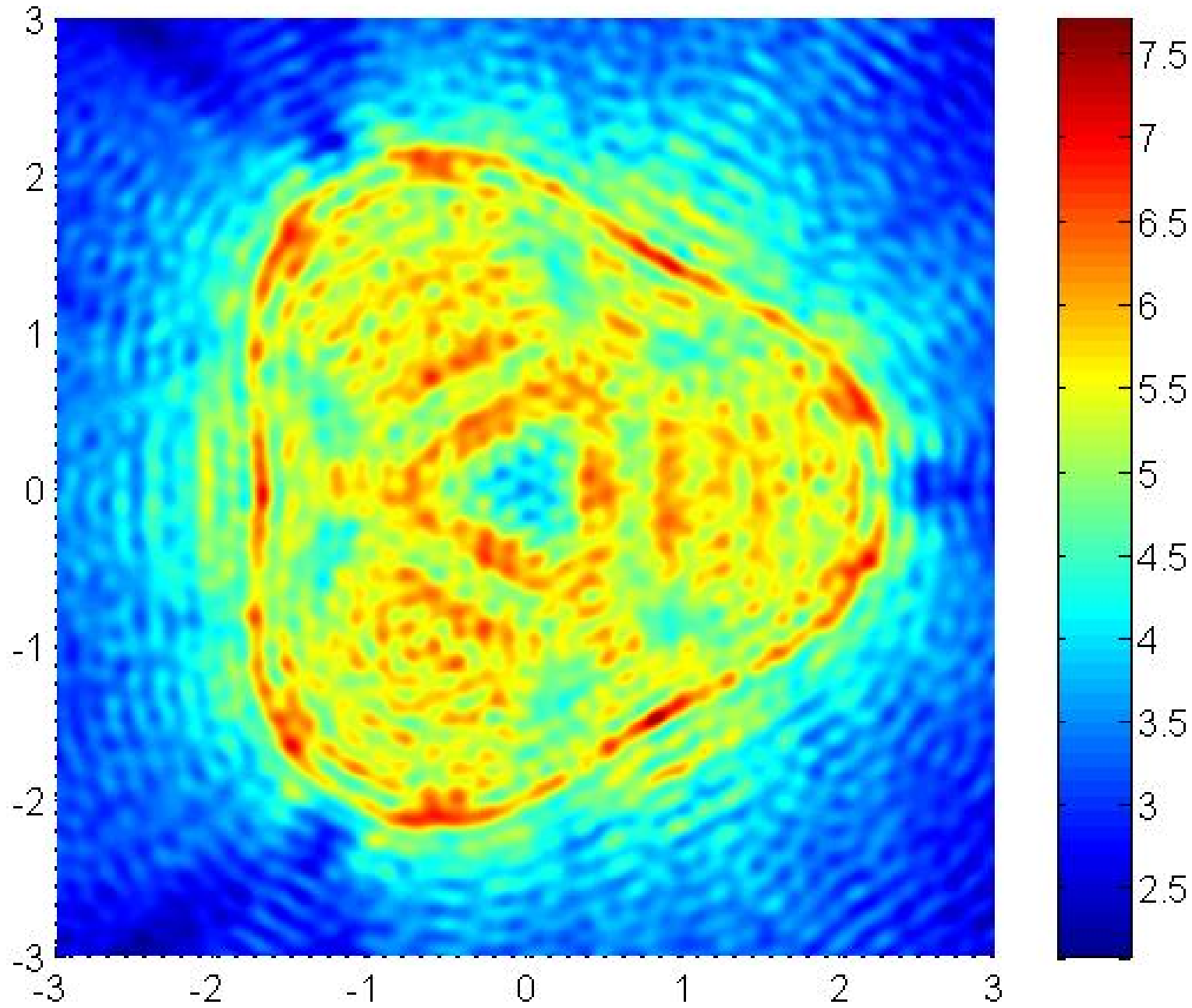}}
\caption{Imaging results of a rounded triangle-shaped, penetrable obstacle with the refractive index $n(x)=4$
and the transmission constant $\la=1$ given by Algorithm \ref{al1} with phaseless data (top row) and
by the imaging algorithm with $I^A_F(z)$ in \cite{P10} with full data (bottom row), respectively.
}\label{fig5}
\end{figure}

\begin{figure}[htbp]
  \centering
  \subfigure[\textbf{Exact curve}]{
    \includegraphics[width=1.5in]{pic/example3/triangle.eps}}
  \subfigure[\textbf{No noise, k=20}]{
    \includegraphics[width=1.5in]{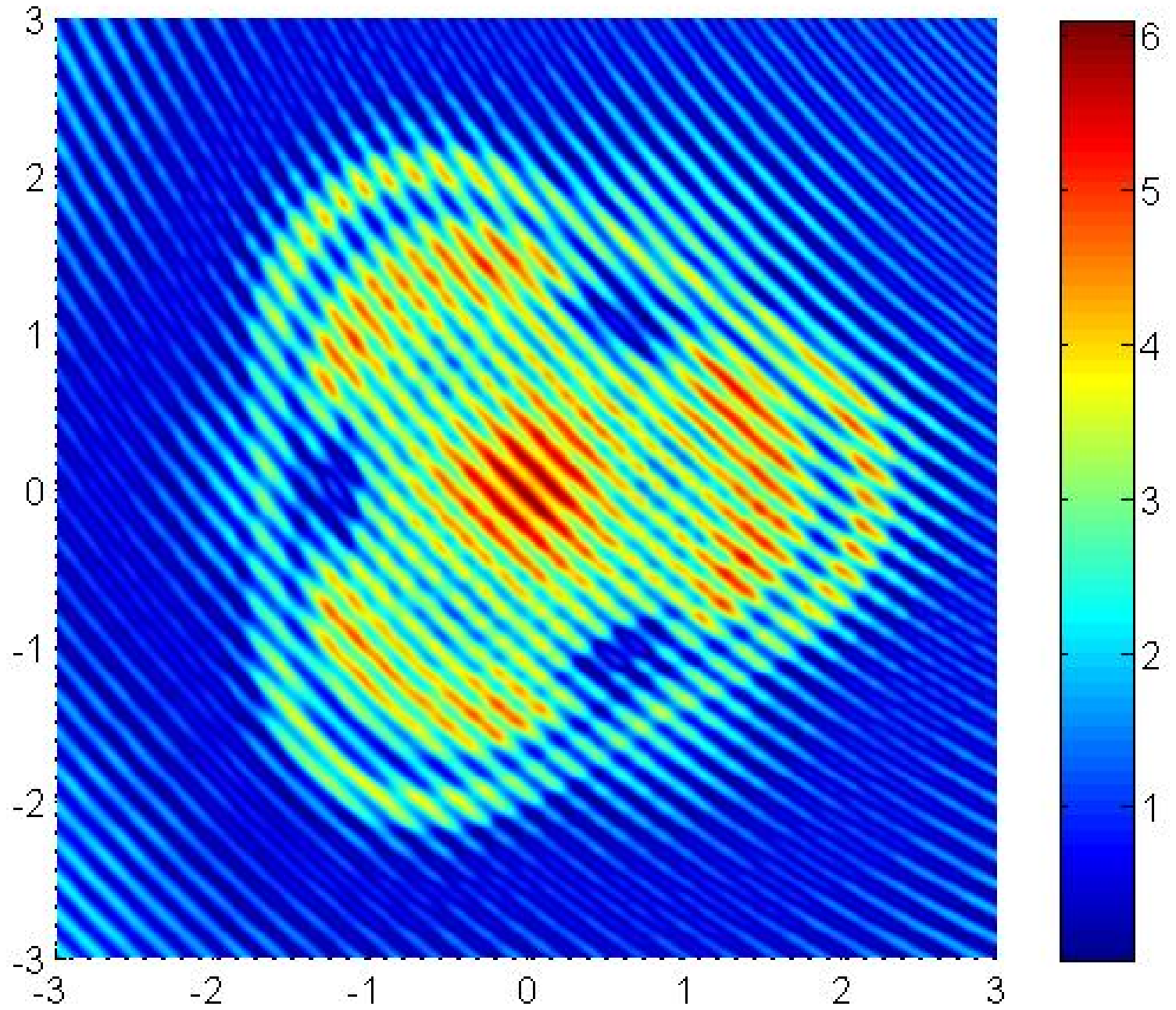}}
  \subfigure[\textbf{5\% noise, k=20}]{
    \includegraphics[width=1.5in]{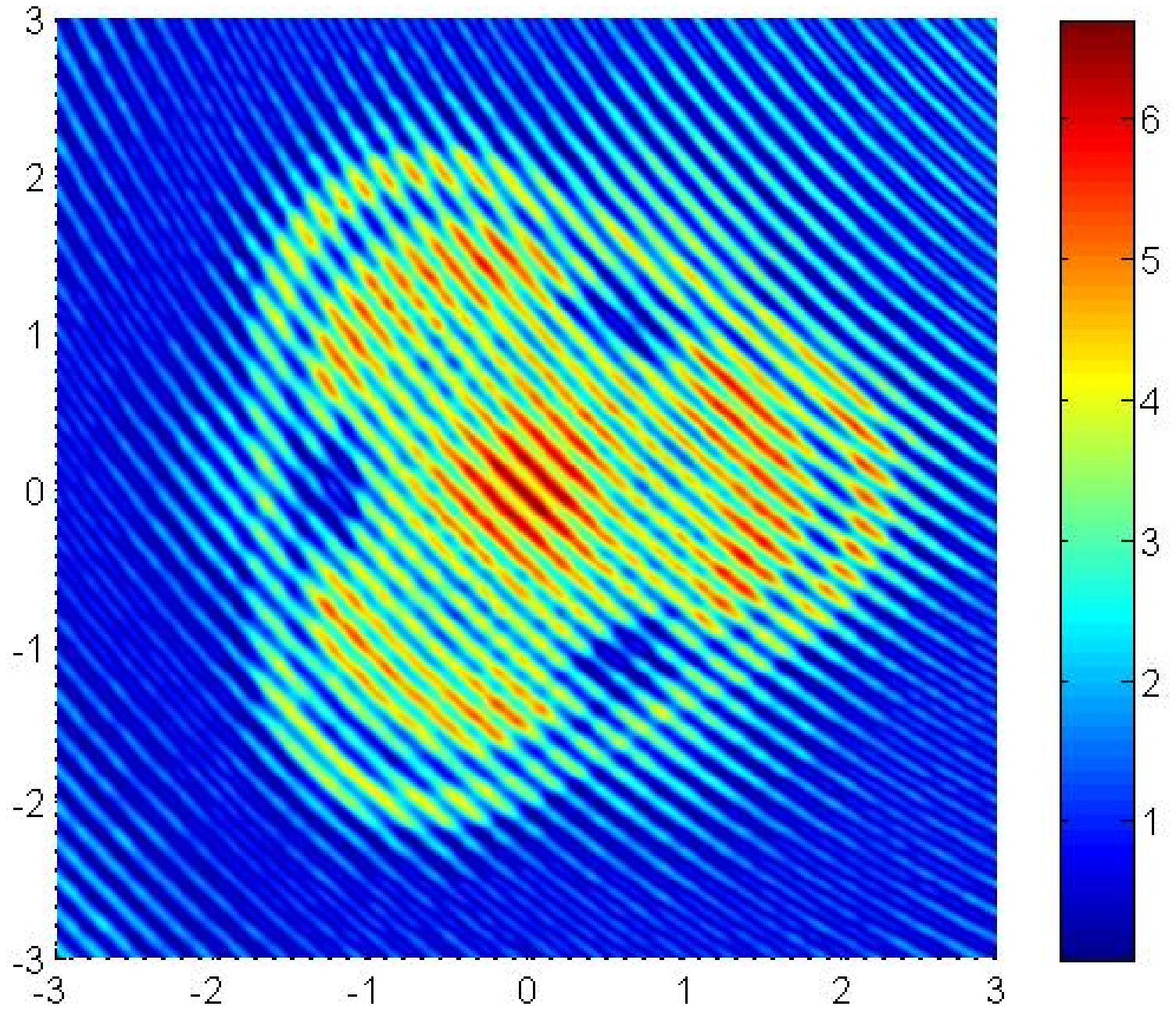}}
  \subfigure[\textbf{10\% noise, k=20}]{
    \includegraphics[width=1.5in]{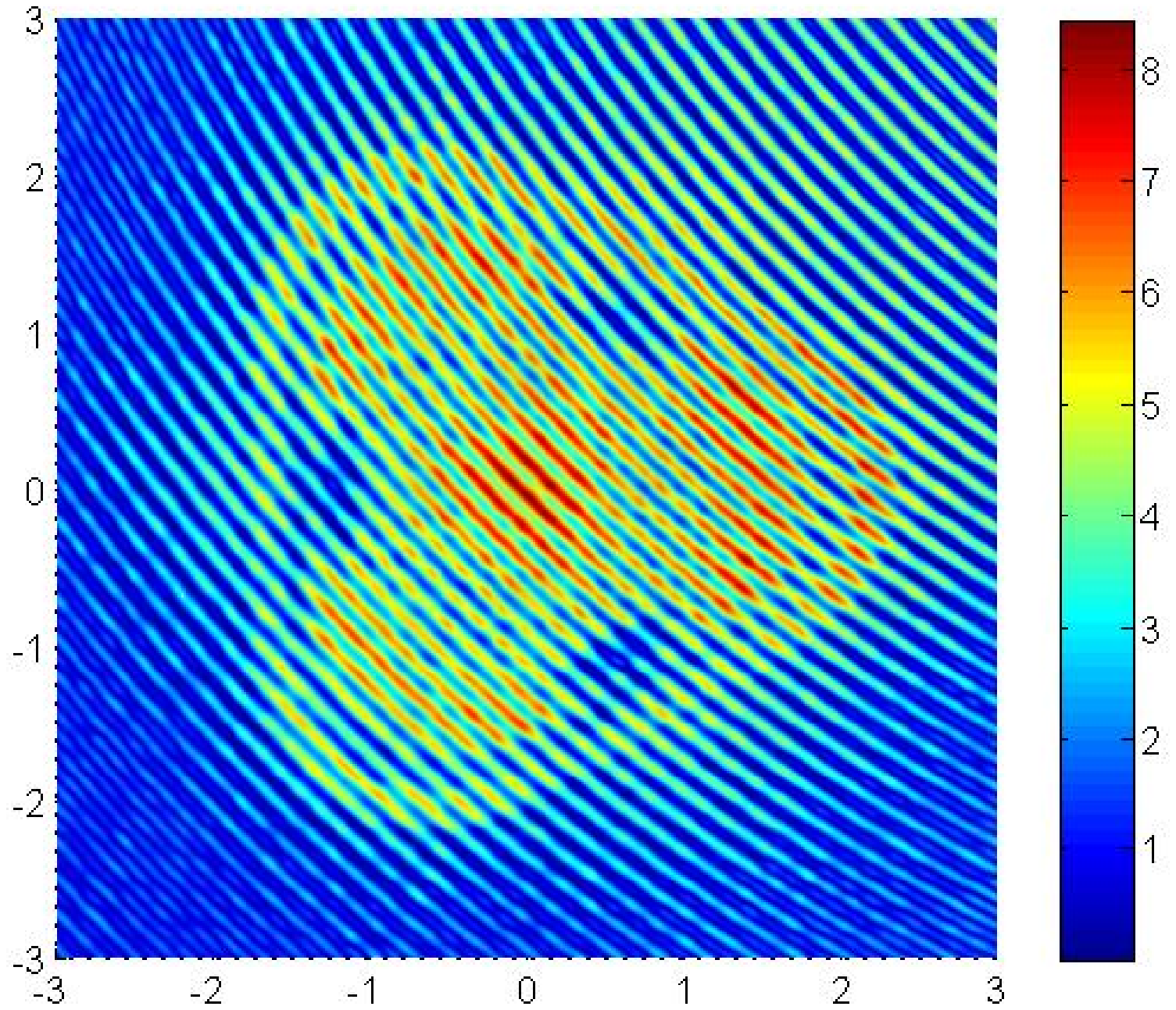}}
  \subfigure[\textbf{No noise, k=20}]{
    \includegraphics[width=1.5in]{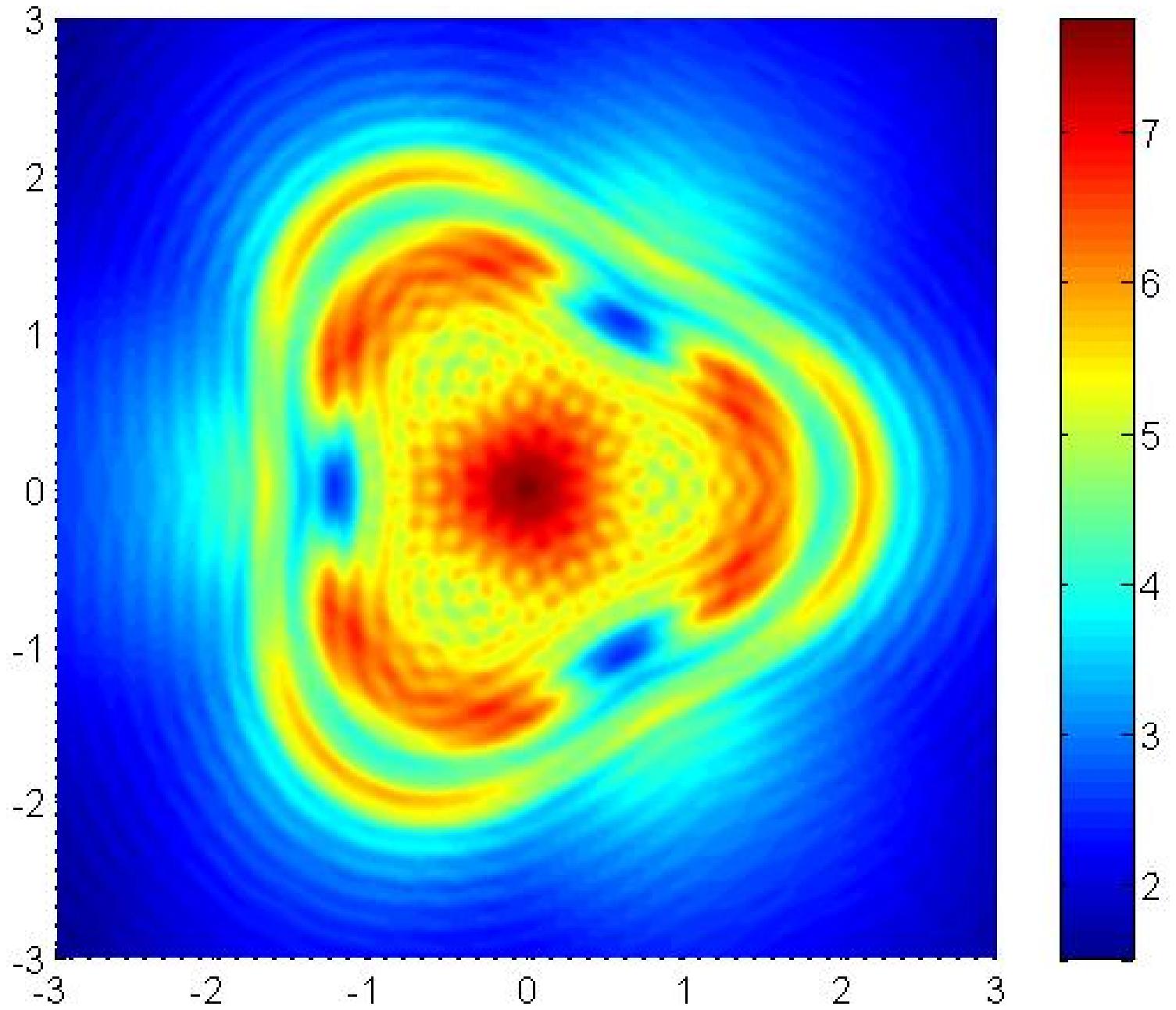}}
  \subfigure[\textbf{5\% noise, k=20}]{
    \includegraphics[width=1.5in]{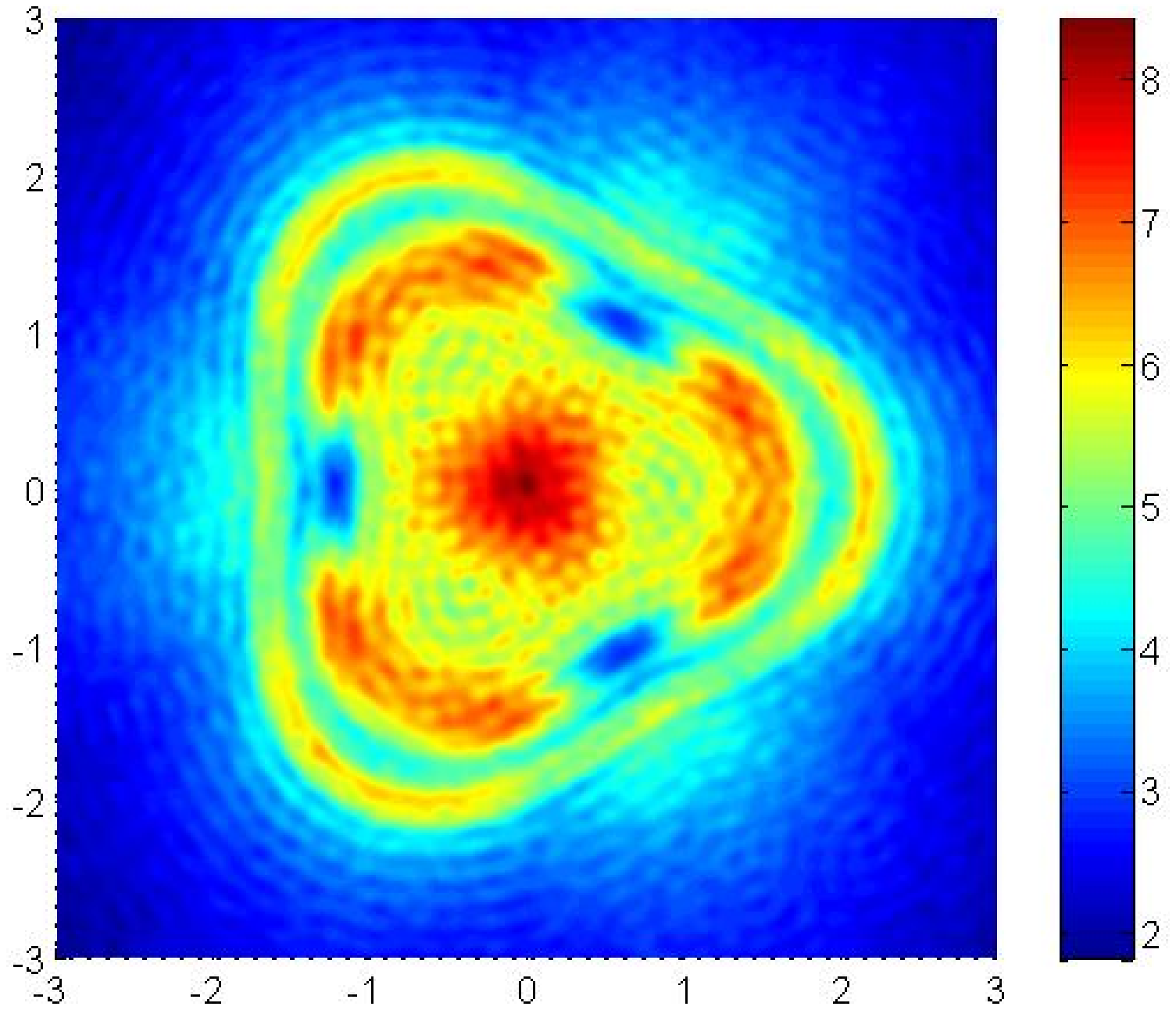}}
  \subfigure[\textbf{10\% noise, k=20}]{
    \includegraphics[width=1.5in]{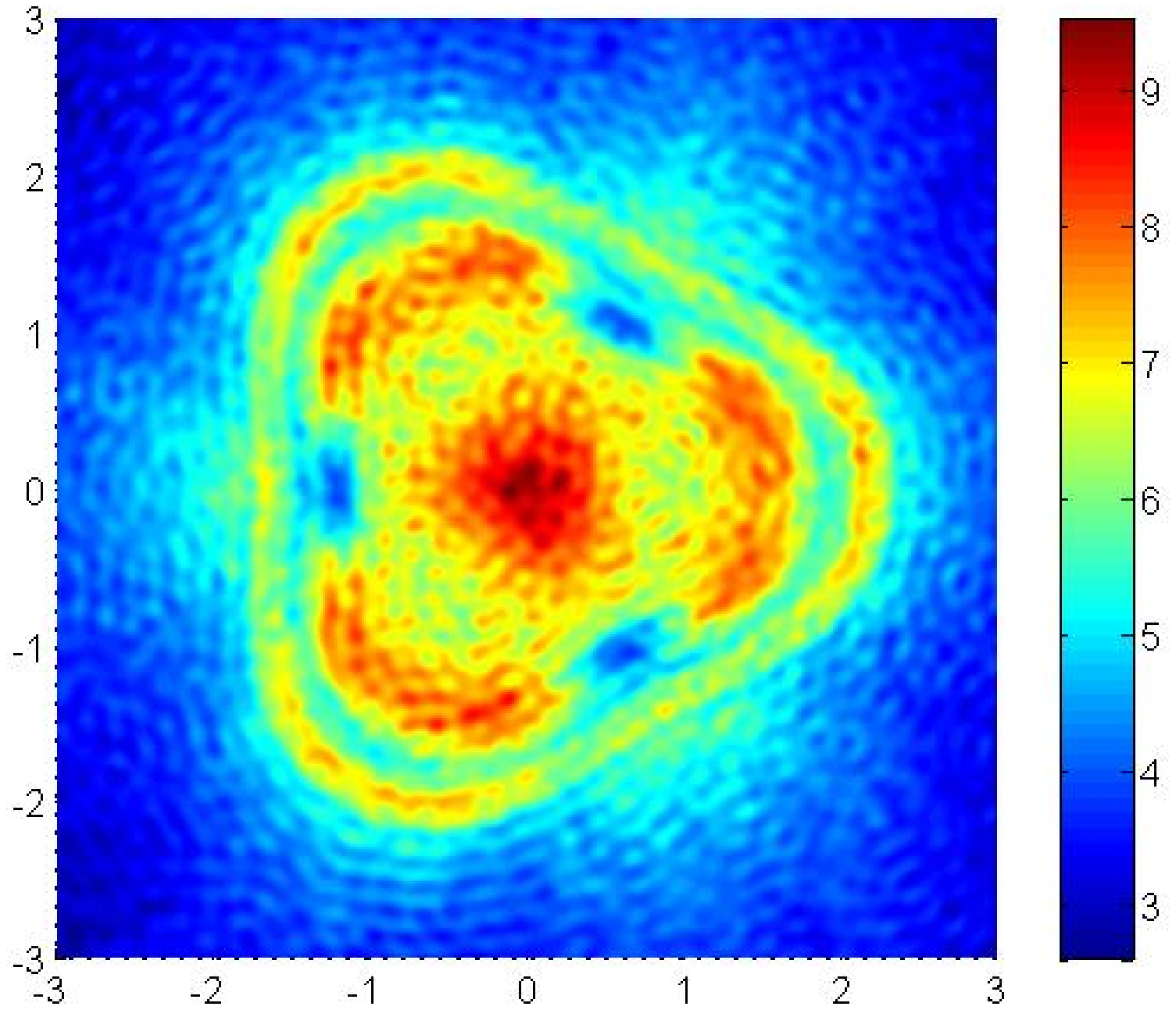}}
\caption{Imaging results of a rounded triangle-shaped, penetrable obstacle with the refractive
index $n(x)=0.64$ and the transmission constant $\la=2$ obtained by Algorithm \ref{al1} with phaseless
data (top row) and by the imaging algorithm with $I^A_F(z)$ in \cite{P10} with full data (bottom row),
respectively.
}\label{fig6}
\end{figure}

\textbf{Example 4: Reconstruction of two obstacles with different boundary conditions.}

This example considers the imaging of two obstacles $D_1$ and $D_2$ with different boundary conditions.
We study the influence of different wave numbers on the imaging results.
We first consider the case where $D_1$ is a rounded triangle-shaped, sound-soft obstacle and $D_2$ is
a circle-shaped, penetrable obstacle of radius $r=2$.
The size of the two obstacles $D_1$ and $D_2$ is comparable. The medium in $D_2$ is characterized by
the refractive index $n(x)=0.25$, and the transmission constant on the boundary $\pa D_2$ is $\la=0.5$
(see Figure \ref{fig7-1}). The sampling region is chosen to be $[-5,5]\times[-5,5]$.
We take $z_0=(-12,0)^T$ for our imaging algorithm.
Figures \ref{fig7}, \ref{fig11} and \ref{fig12} present the imaging results of $I^A_{z_0}(z)$ and $I^A_F(z)$
with the wave number $k=5,10,20$, respectively, from the exact, $5\%$ noisy and $10\%$ noisy measured data.

Next we consider the case where the size of the two obstacles is incompatible.
$D_1$ is a very small, circle-shaped, sound-soft obstacle of radius $r=0.1$ and $D_2$ is a much larger,
rounded square-shaped, impedance obstacle with the impedance function $\rho(x)=5$ (see Figure \ref{fig8-1}).
The searching region is chosen to be $[-4,4]\times[-4,4]$. We choose $z_0=(13,0)^T$ for our imaging algorithm.
Figures \ref{fig8}, \ref{fig13} and \ref{fig14} present the imaging results of $I^A_{z_0}(z)$ and $I^A_F(z)$
with the wave number $k=5,10,20$, respectively, from the exact, $5\%$ noisy and $10\%$ noisy measured data.

It is observed that in both cases the two obstacles can be reconstructed accurately as long as sufficiently
high-frequency data are used.

\begin{figure}[htbp]
  \centering
  \subfigure[\textbf{Exact curves}]{\label{fig7-1}
    \includegraphics[width=1.5in]{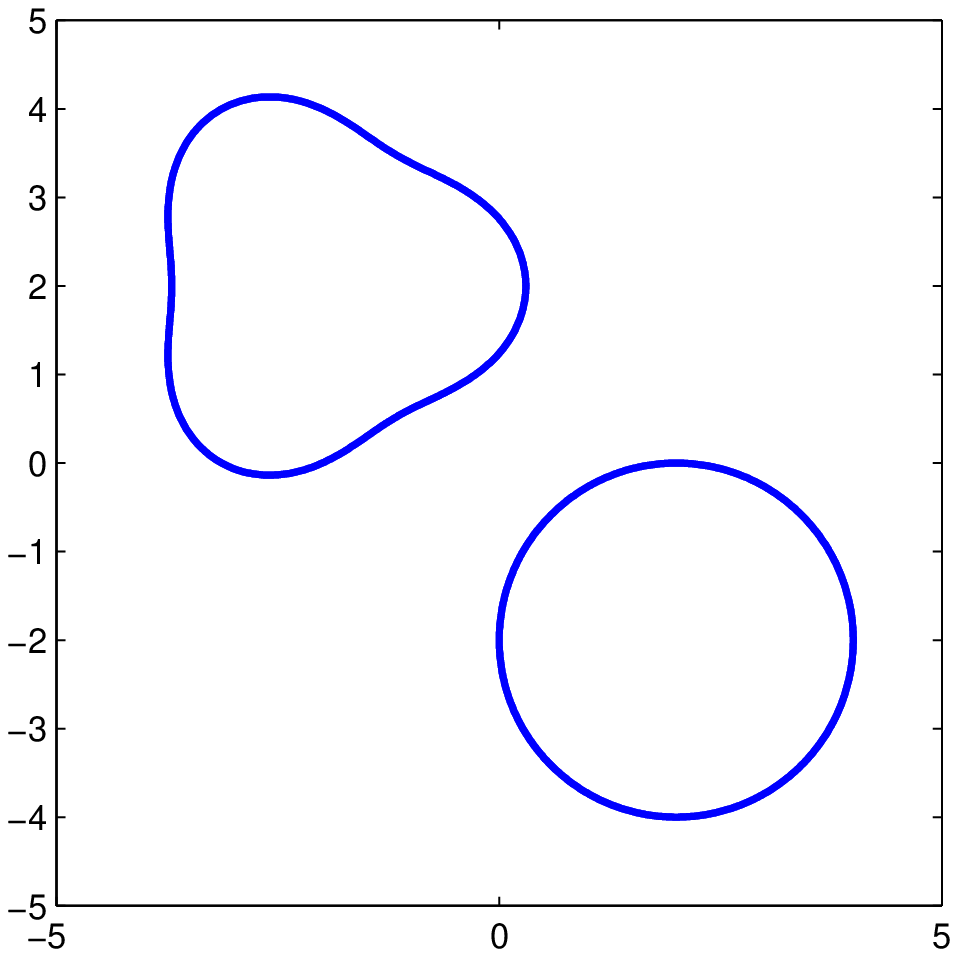}}
  \subfigure[\textbf{No noise, k=5}]{
    \includegraphics[width=1.5in]{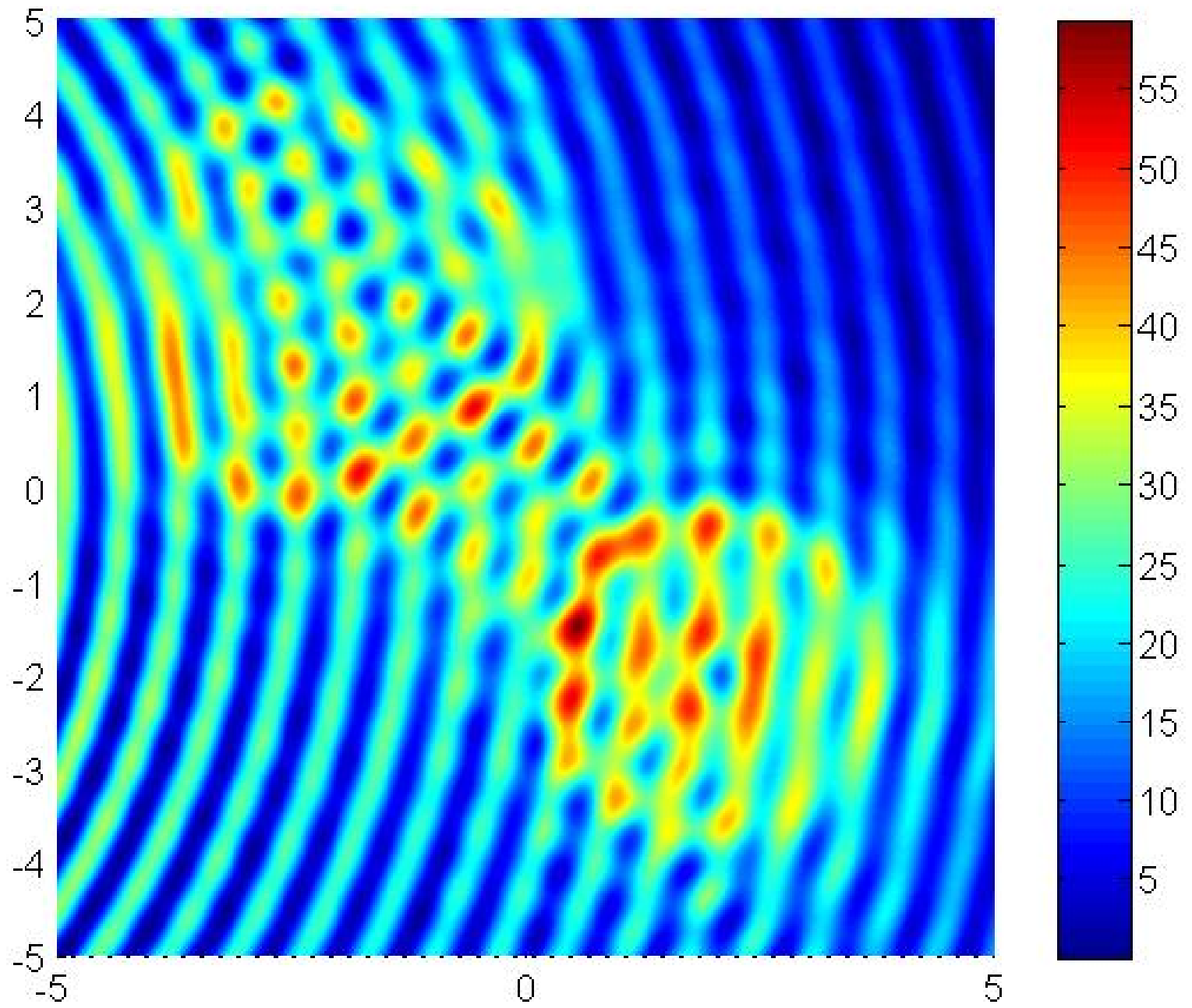}}
  \subfigure[\textbf{5\% noise, k=5}]{
    \includegraphics[width=1.5in]{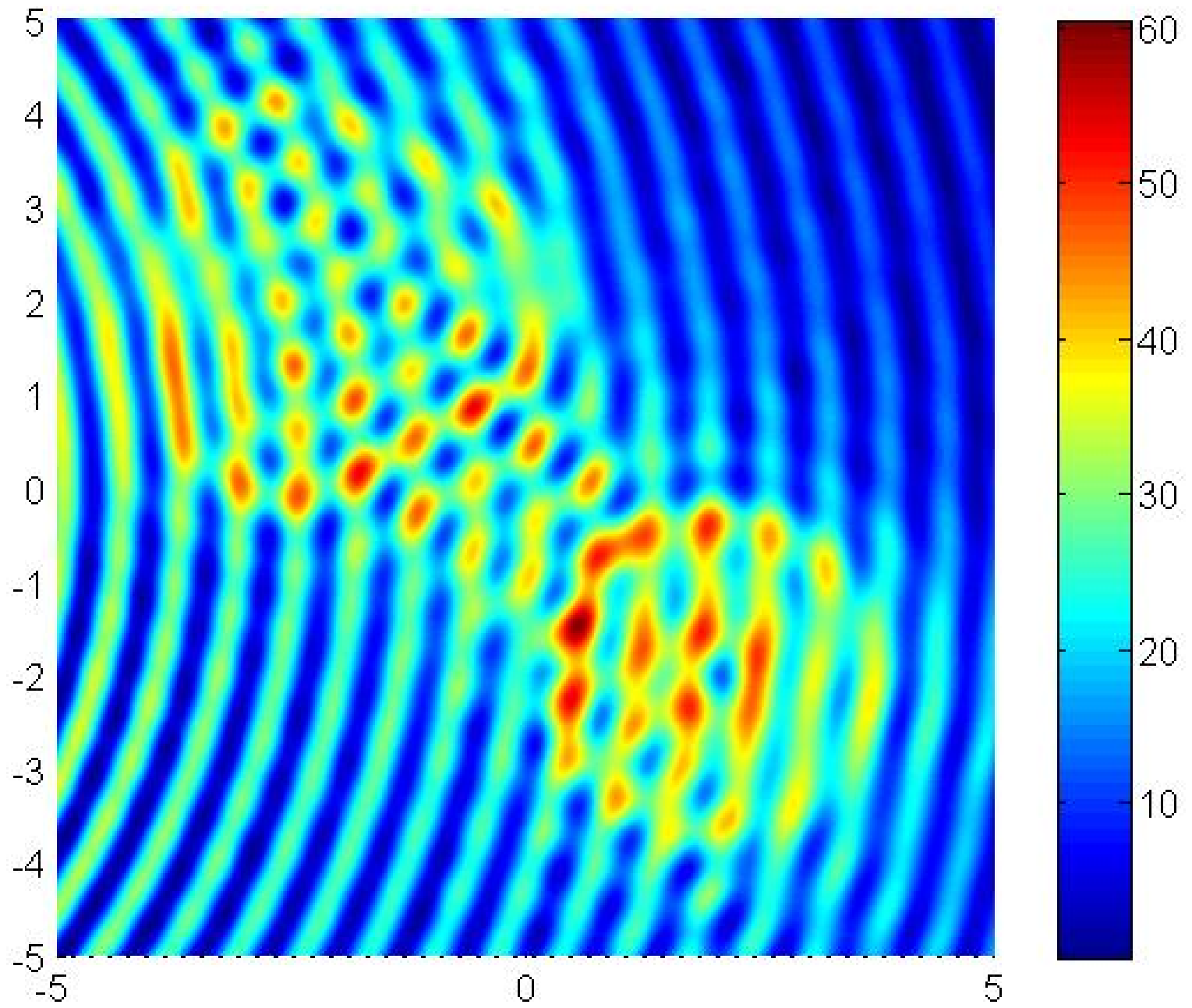}}
  \subfigure[\textbf{10\% noise, k=5}]{
    \includegraphics[width=1.5in]{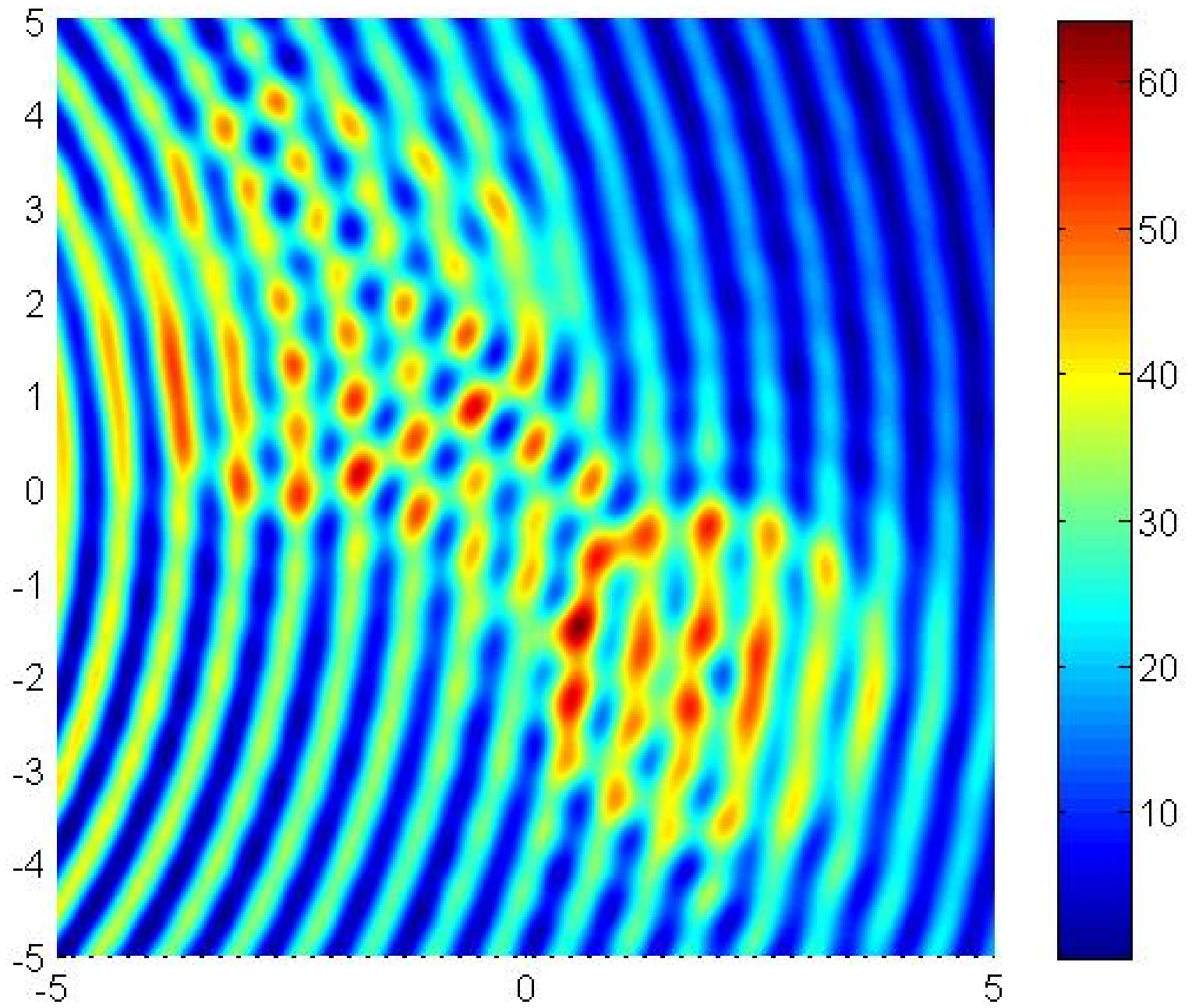}}
  \subfigure[\textbf{No noise, k=5}]{
    \includegraphics[width=1.5in]{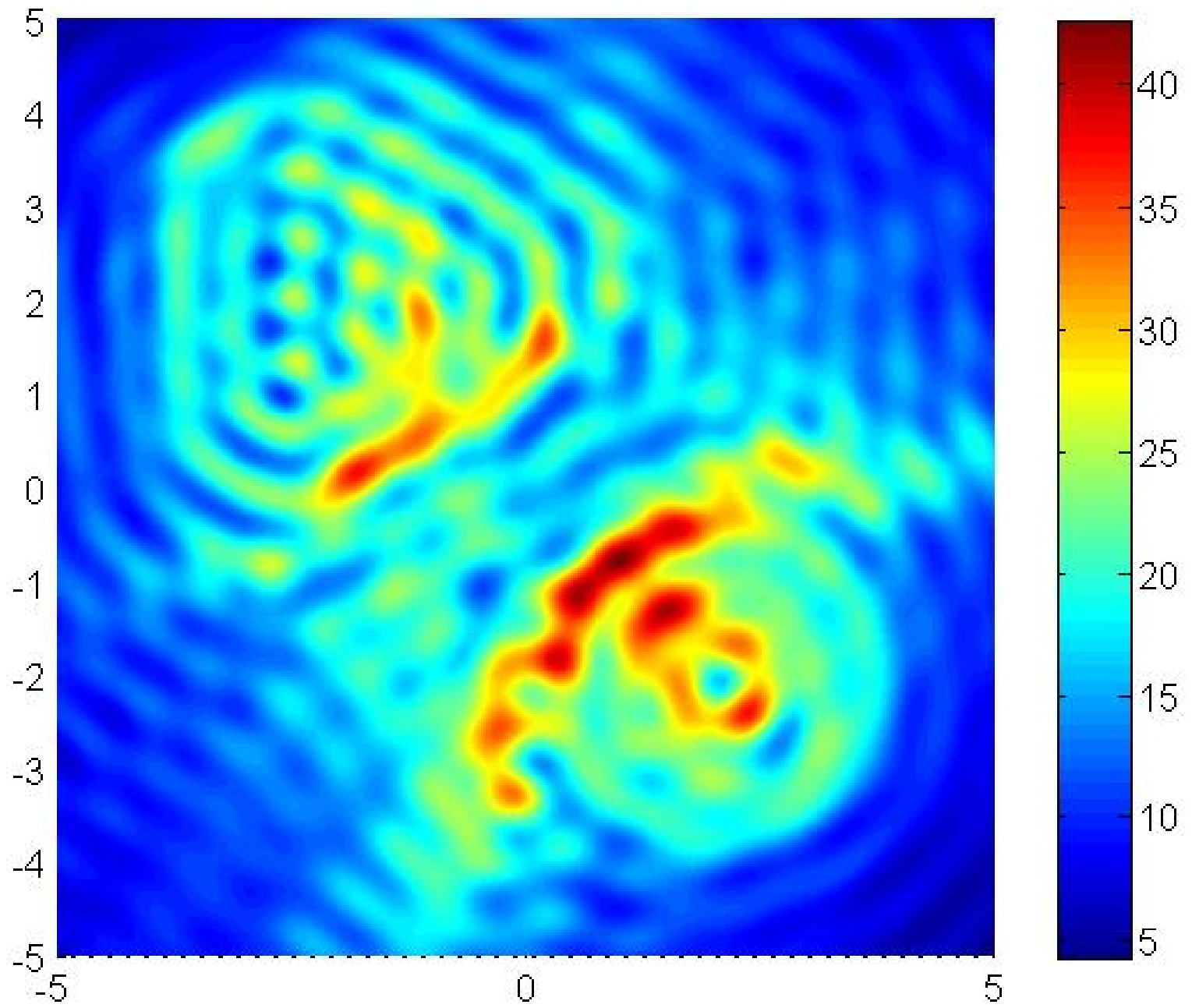}}
  \subfigure[\textbf{5\% noise, k=5}]{
    \includegraphics[width=1.5in]{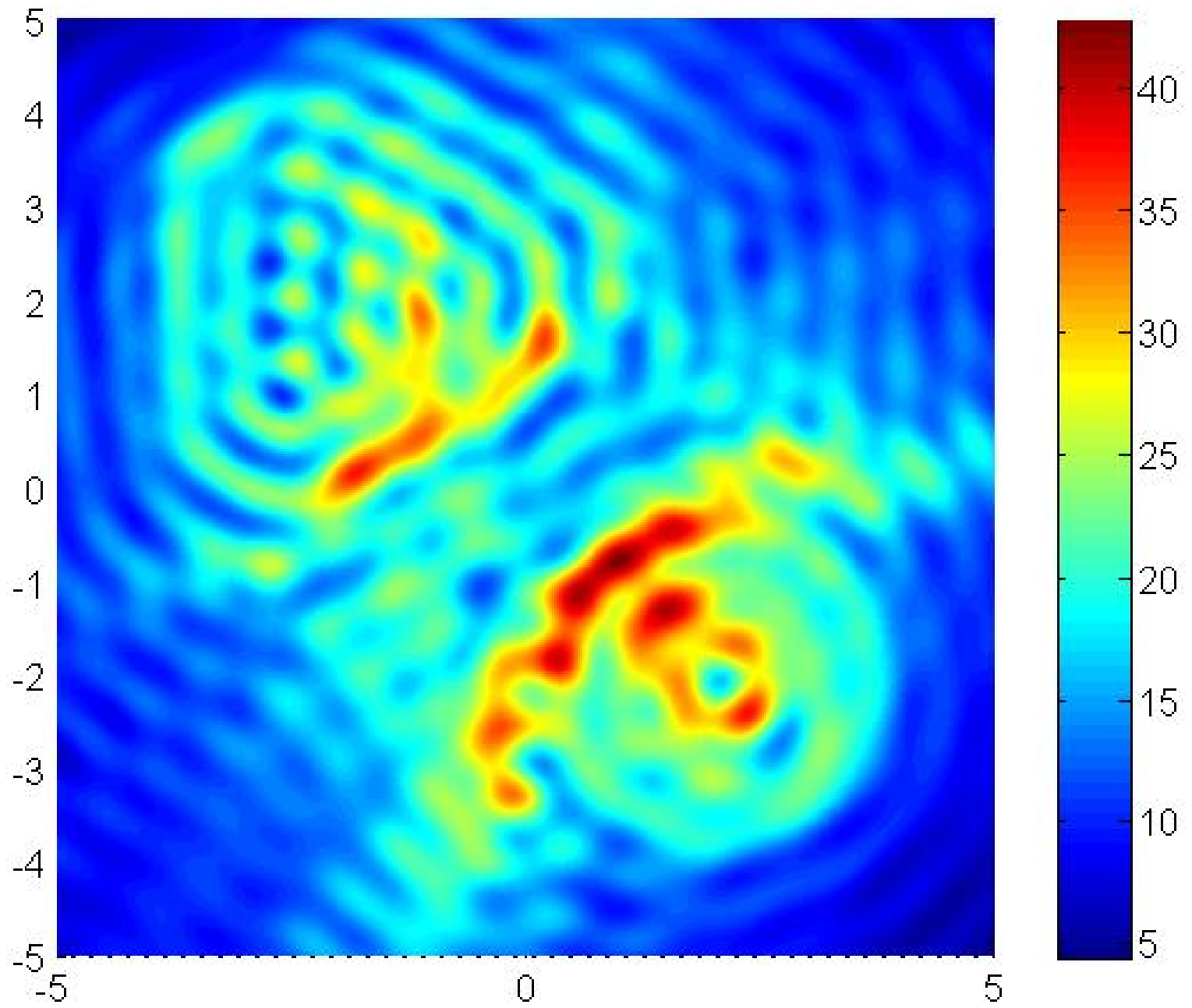}}
  \subfigure[\textbf{10\% noise, k=5}]{
    \includegraphics[width=1.5in]{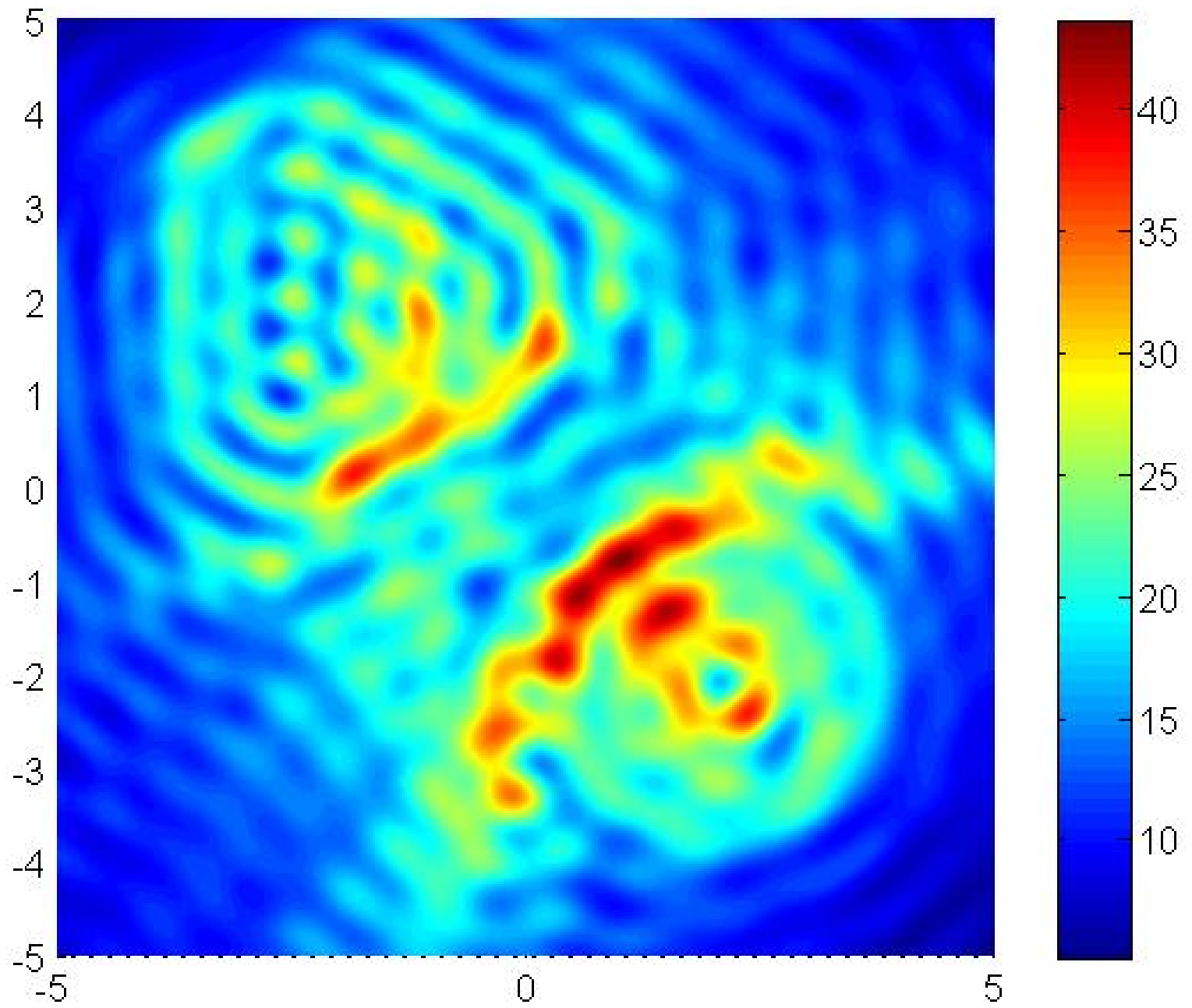}}
\caption{Imaging results of a rounded triangle-shaped, sound-soft obstacle and a circle-shaped,
penetrable obstacle of radius $r=2$ with the refractive index $n(x)=0.25$ and
the transmission constant $\la=0.5$ obtained by Algorithm \ref{al1} with phaseless data (top row)
and by the imaging algorithm with $I^A_F(z)$ in \cite{P10} with full data (bottom row), respectively.
}\label{fig7}
\end{figure}

\begin{figure}[htbp]
  \centering
  \subfigure[\textbf{Exact curves}]{
    \includegraphics[width=1.5in]{pic/example4/case1/triangle_circle.eps}}
  \subfigure[\textbf{No noise, k=10}]{
    \includegraphics[width=1.5in]{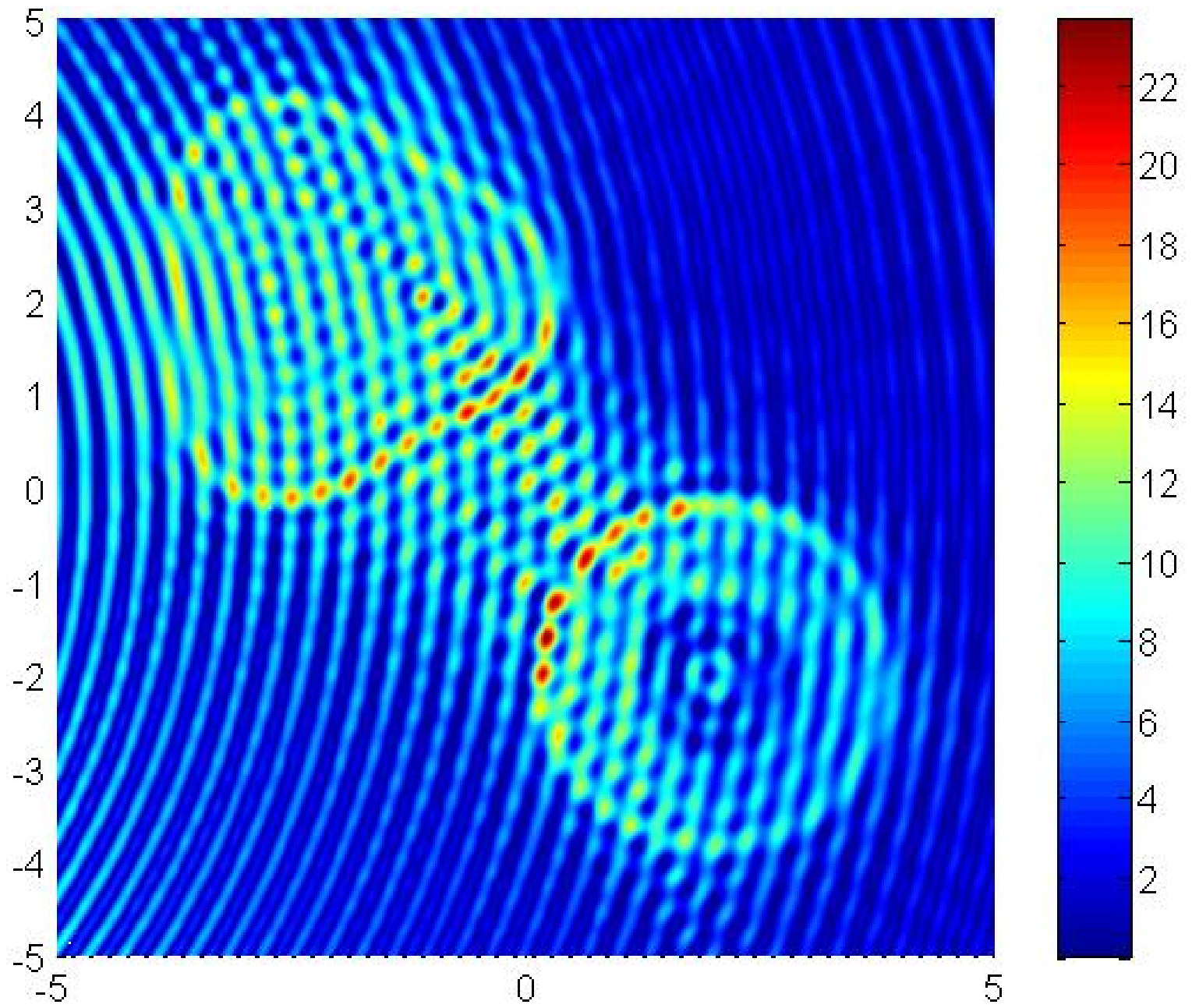}}
  \subfigure[\textbf{5\% noise, k=10}]{
    \includegraphics[width=1.5in]{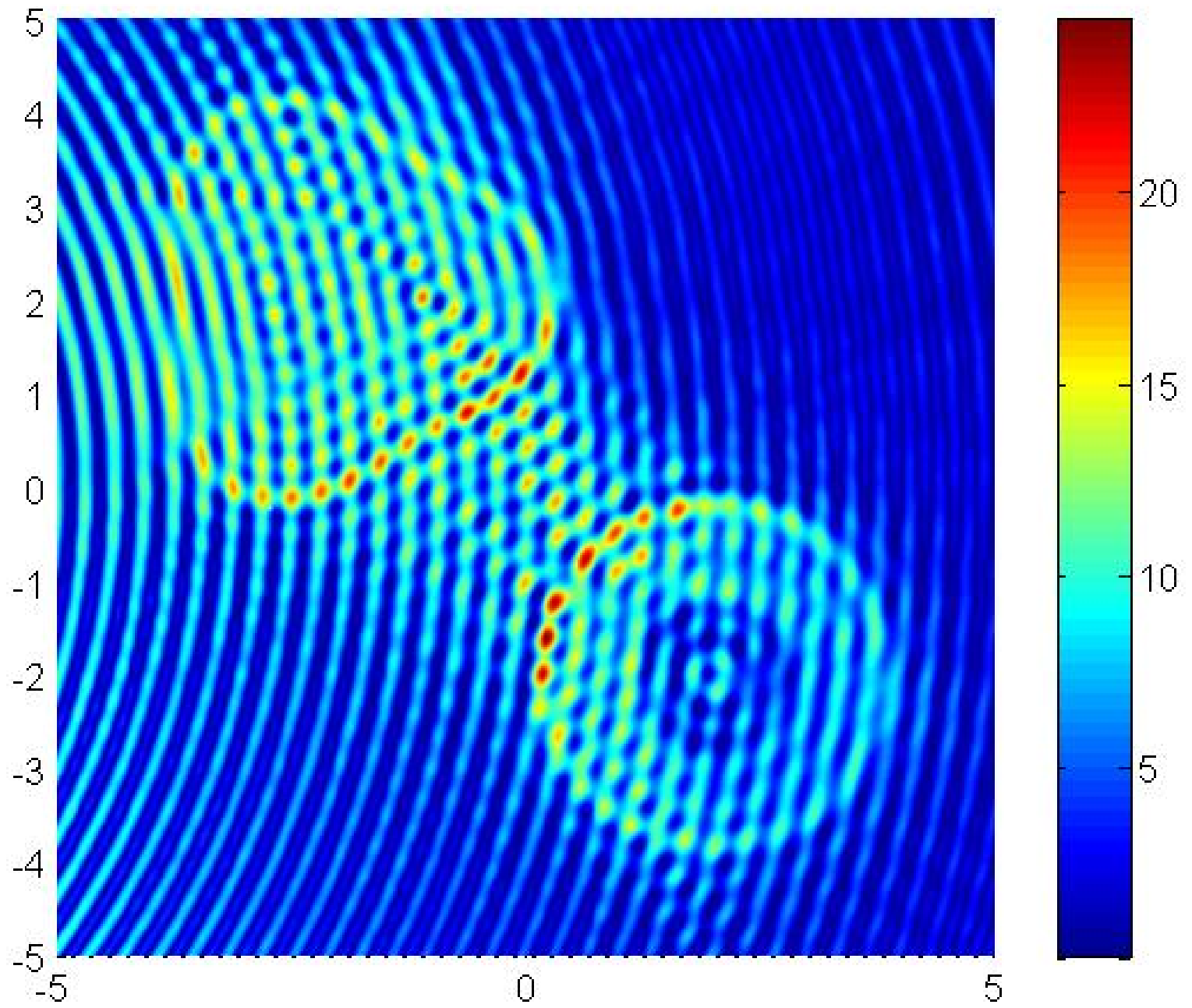}}
  \subfigure[\textbf{10\% noise, k=10}]{
    \includegraphics[width=1.5in]{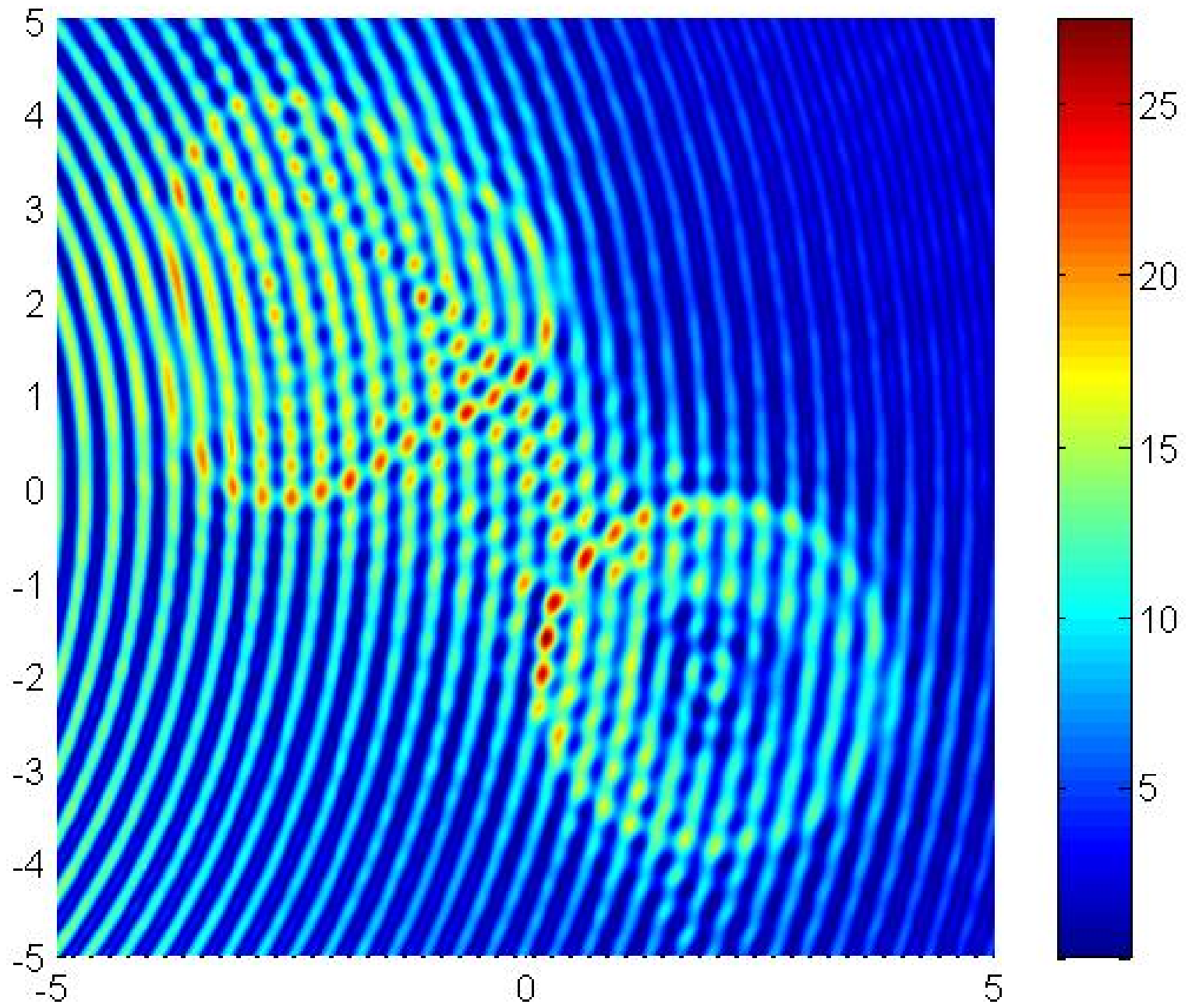}}
  \subfigure[\textbf{No noise, k=10}]{
    \includegraphics[width=1.5in]{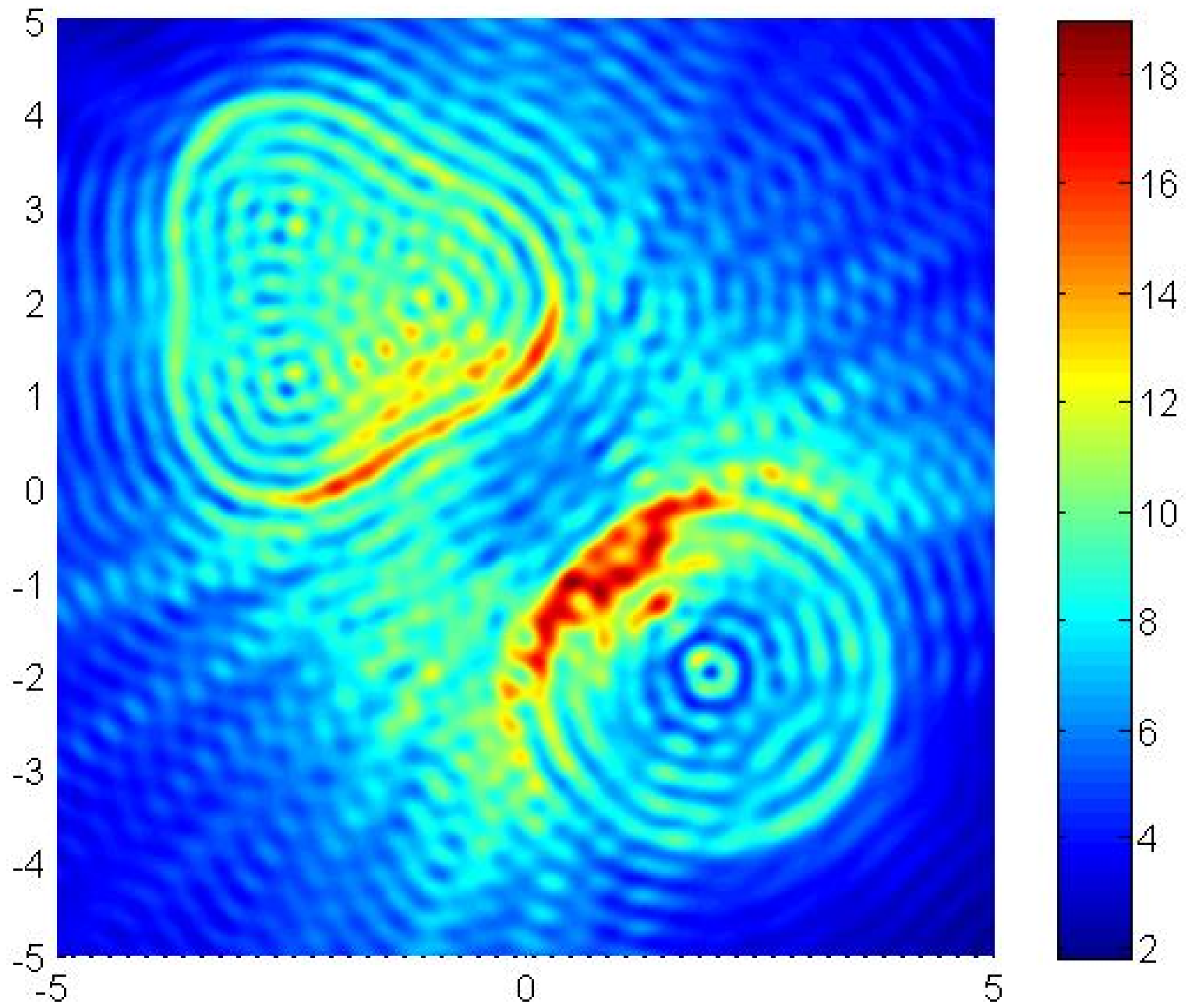}}
  \subfigure[\textbf{5\% noise, k=10}]{
    \includegraphics[width=1.5in]{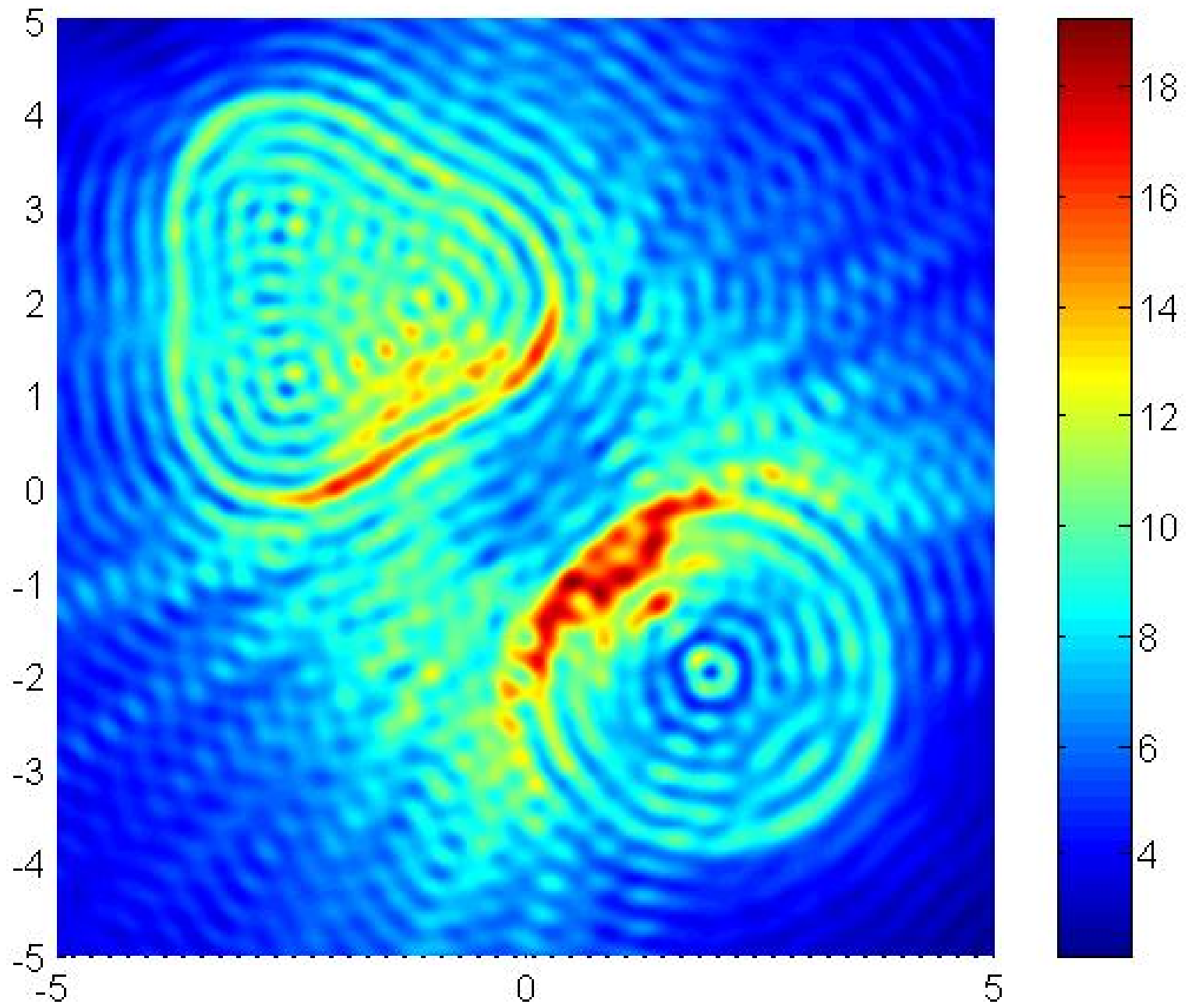}}
  \subfigure[\textbf{10\% noise, k=10}]{
    \includegraphics[width=1.5in]{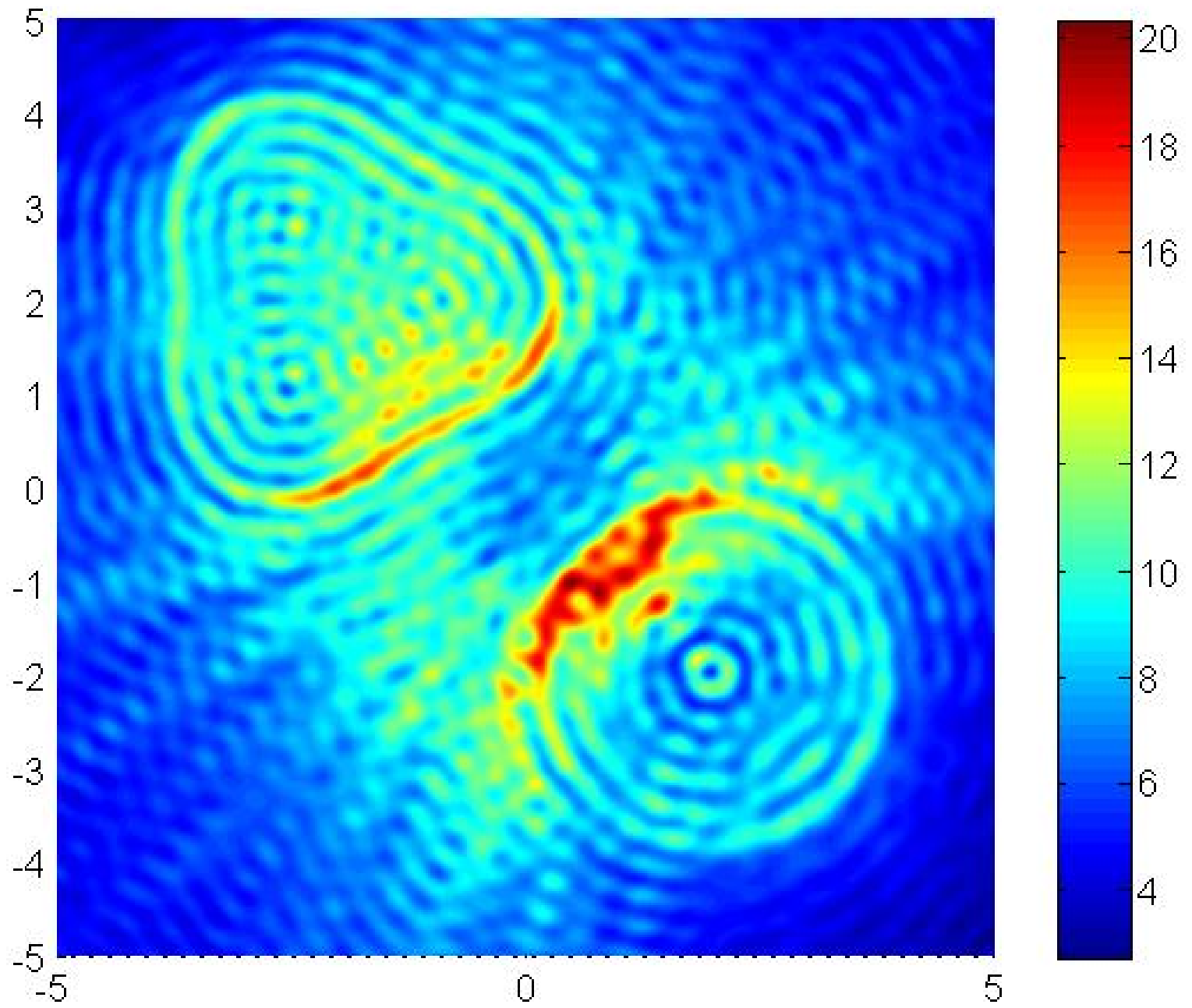}}
\caption{Imaging results of a rounded triangle-shaped, sound-soft obstacle and a circle-shaped,
penetrable obstacle of radius $r=2$ with the refractive index $n(x)=0.25$ and the transmission
constant $\la=0.5$ obtained by Algorithm \ref{al1} with phaseless data (top row) and
by the imaging algorithm with $I^A_F(z)$ in \cite{P10} with full data (bottom row), respectively.
}\label{fig11}
\end{figure}

\begin{figure}[htbp]
  \centering
  \subfigure[\textbf{Exact curves}]{
    \includegraphics[width=1.5in]{pic/example4/case1/triangle_circle.eps}}
  \subfigure[\textbf{No noise, k=20}]{
    \includegraphics[width=1.5in]{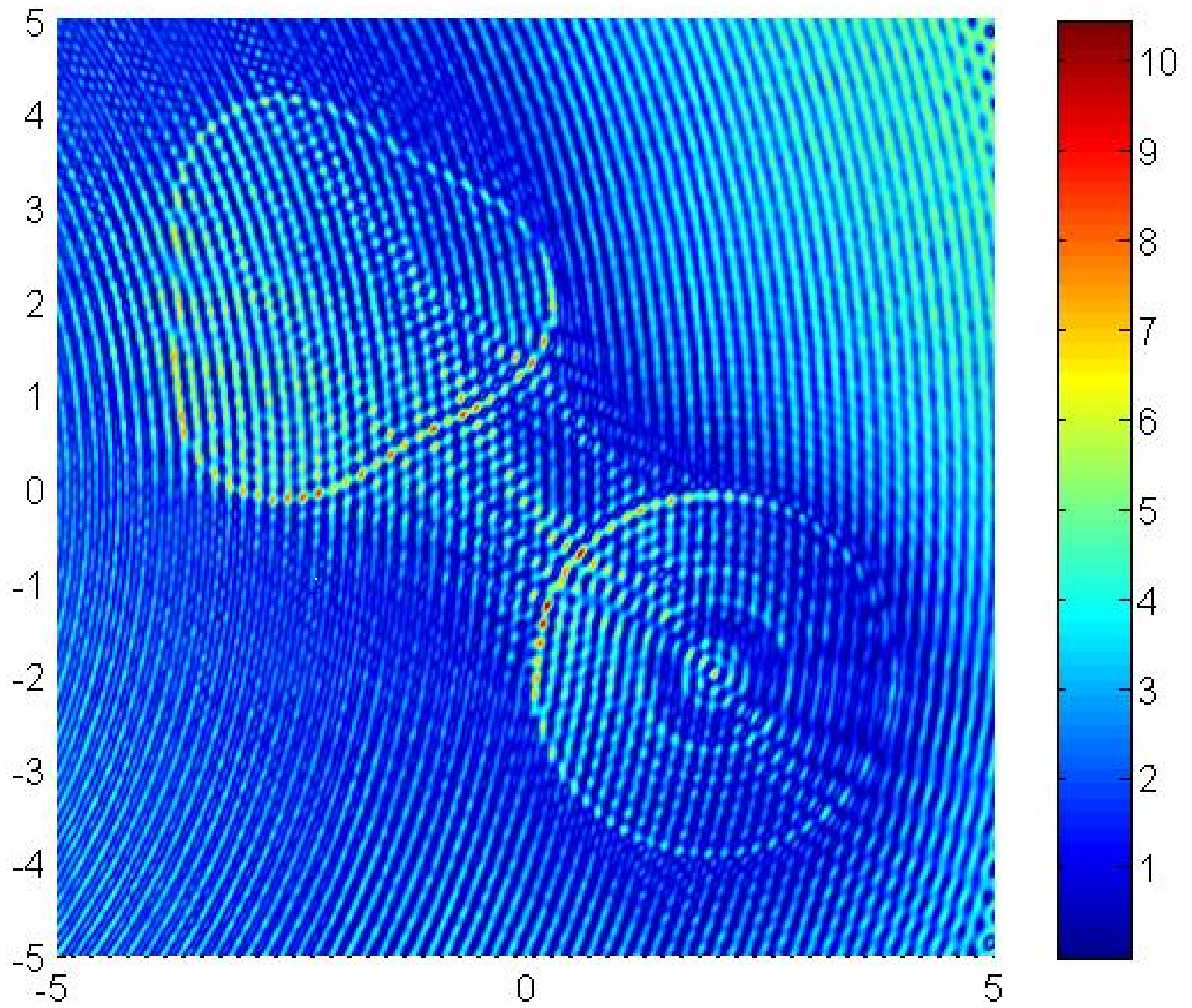}}
  \subfigure[\textbf{5\% noise, k=20}]{
    \includegraphics[width=1.5in]{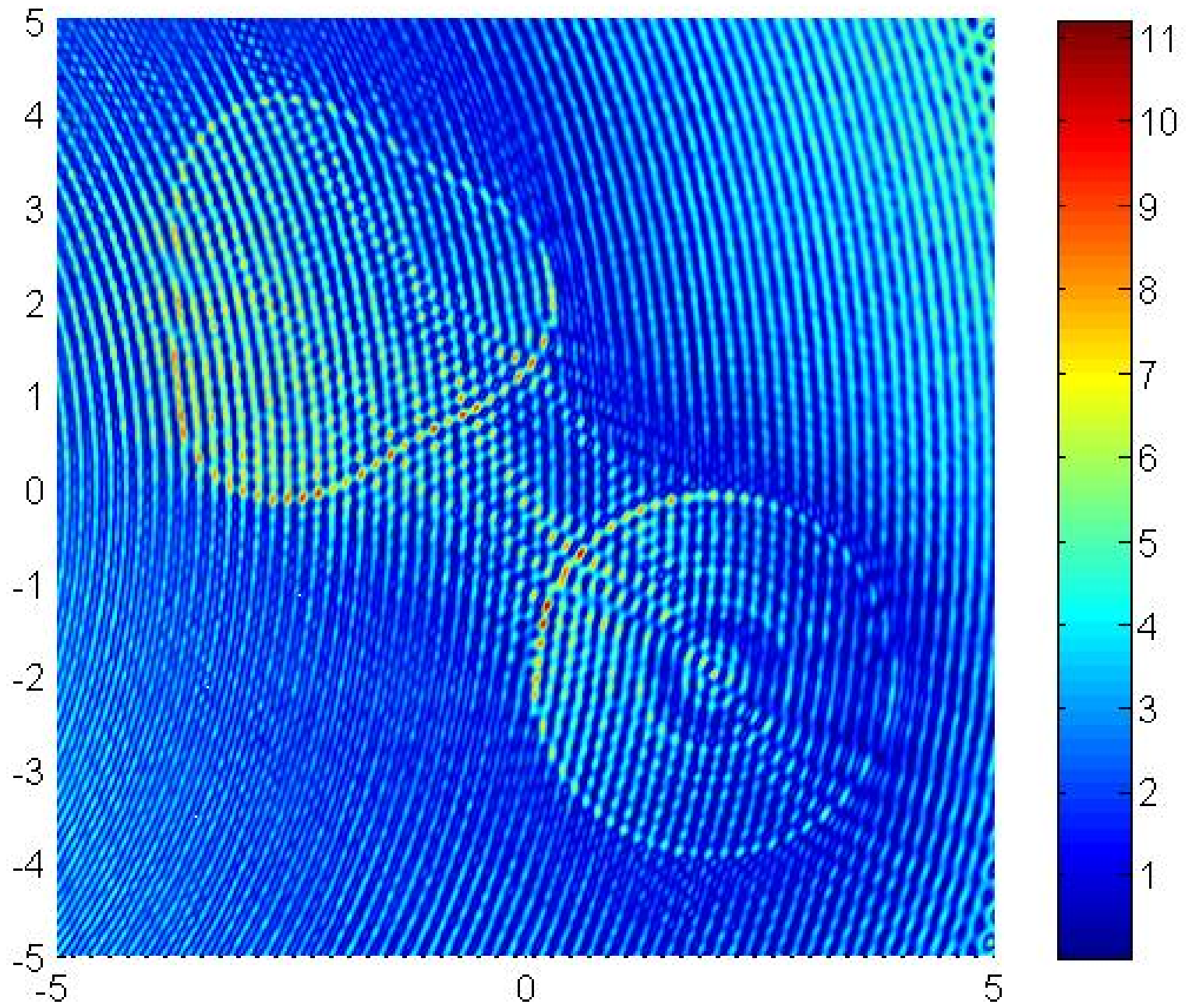}}
  \subfigure[\textbf{10\% noise, k=20}]{
    \includegraphics[width=1.5in]{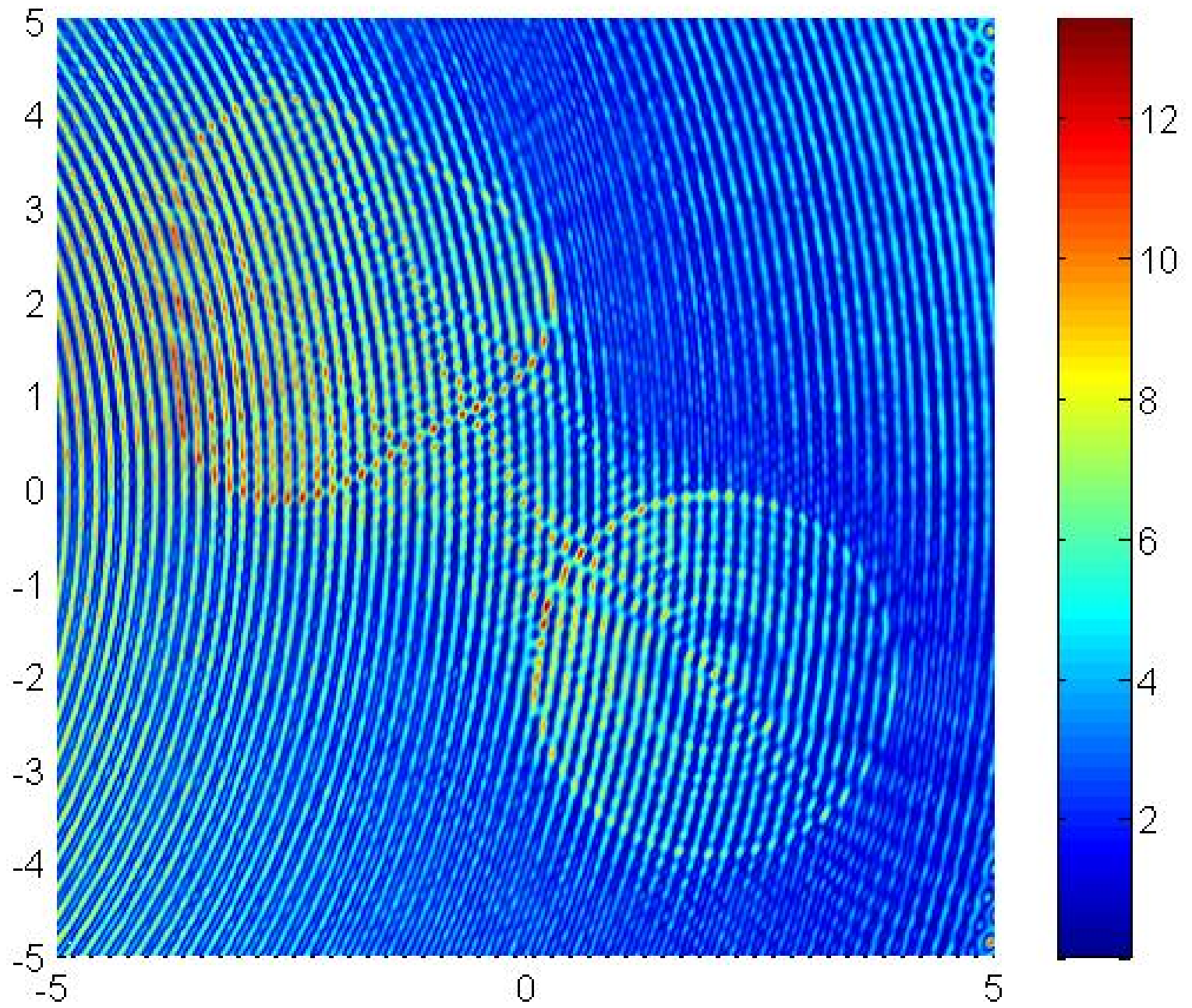}}
  \subfigure[\textbf{No noise, k=20}]{
    \includegraphics[width=1.5in]{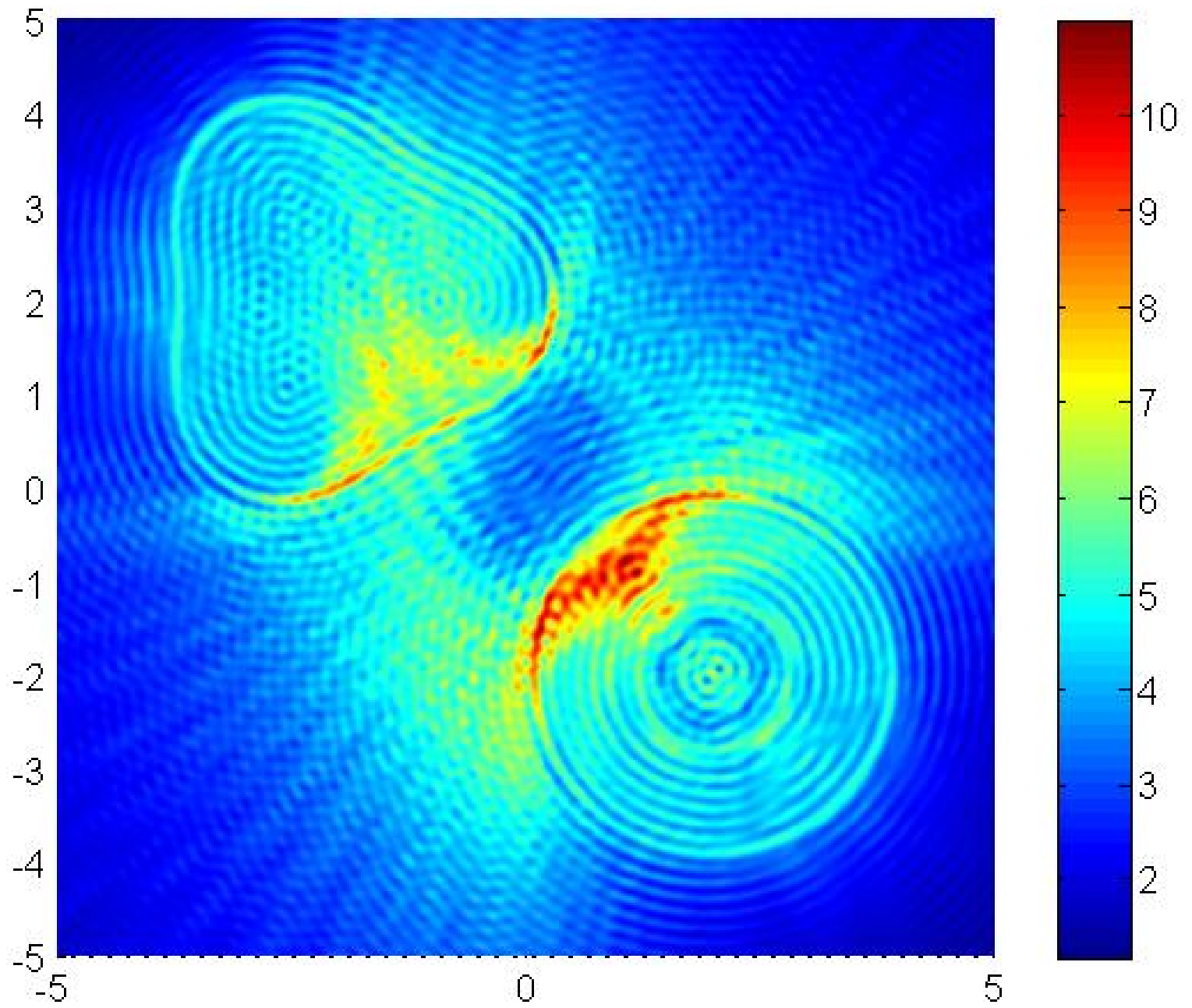}}
  \subfigure[\textbf{5\% noise, k=20}]{
    \includegraphics[width=1.5in]{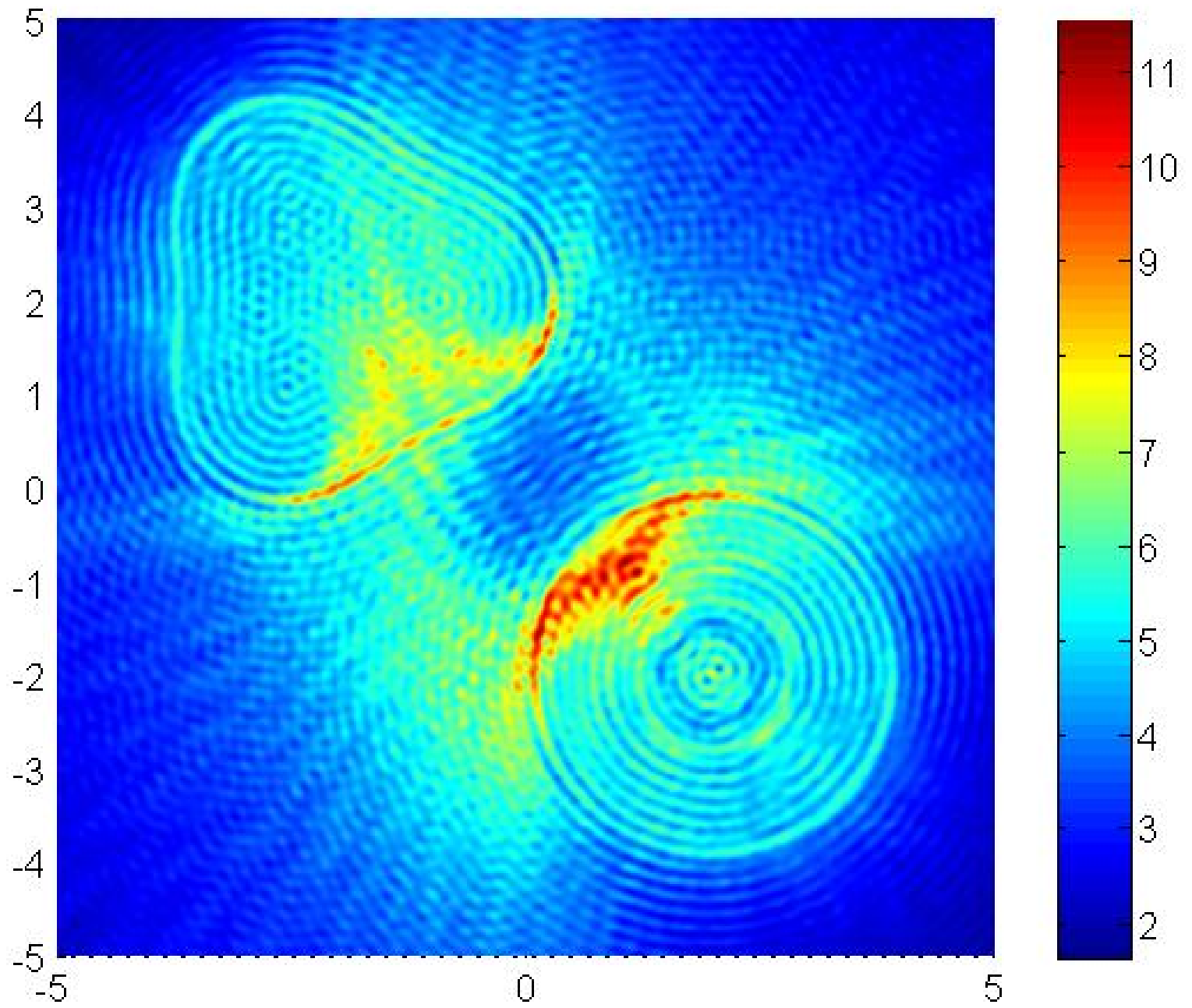}}
  \subfigure[\textbf{10\% noise, k=20}]{
    \includegraphics[width=1.5in]{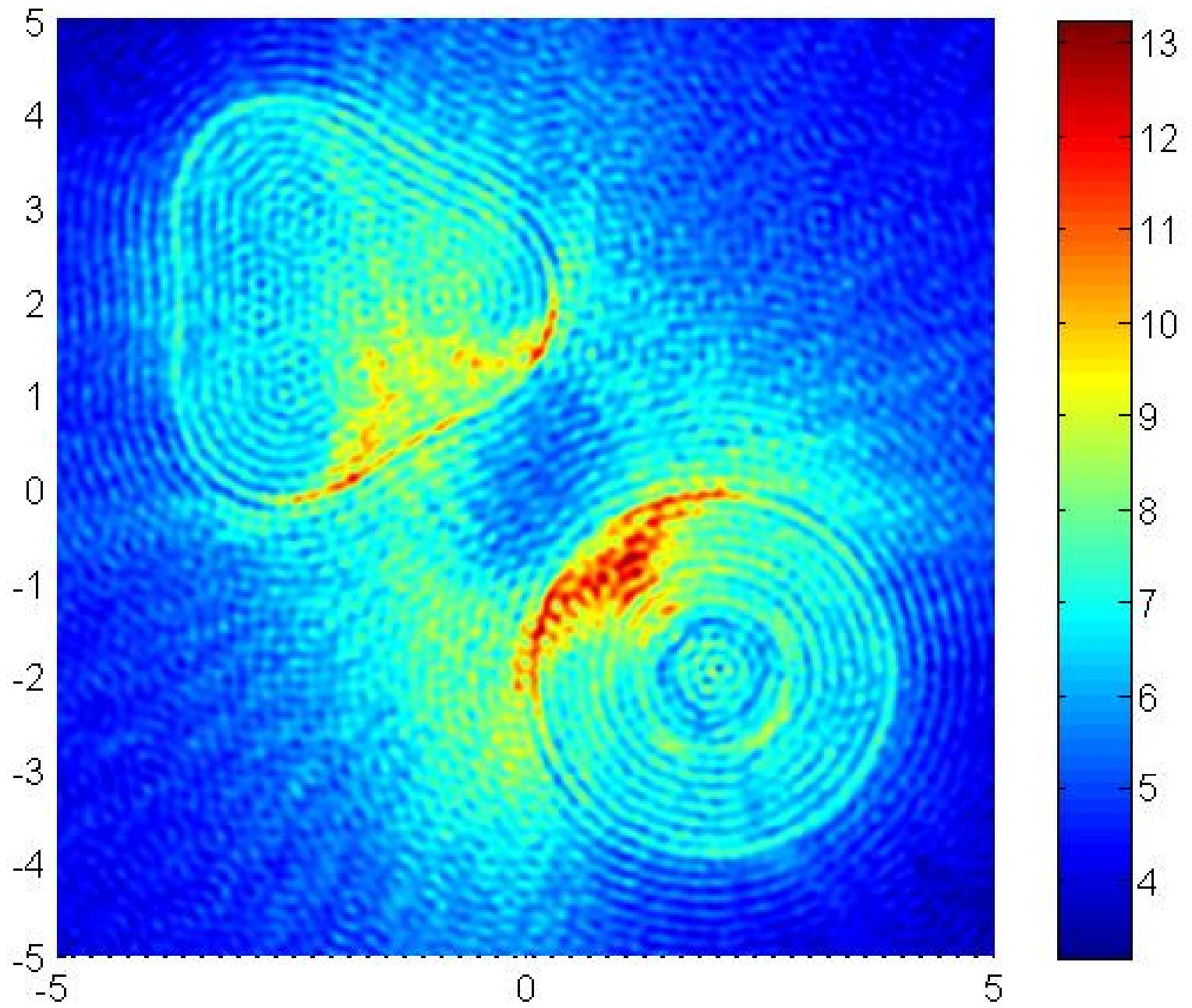}}
\caption{Imaging results of a rounded triangle-shaped, sound-soft obstacle and a circle-shaped,
penetrable obstacle of radius $r=2$ with the refractive index $n(x)=0.25$ and the transmission
constant $\la=0.5$ obtained by Algorithm \ref{al1} with phaseless data (top row) and
by the imaging algorithm with $I^A_F(z)$ in \cite{P10} with full data (bottom row), respectively.
}\label{fig12}
\end{figure}

\begin{figure}[htbp]
  \centering
  \subfigure[\textbf{Exact curves}]{\label{fig8-1}
    \includegraphics[width=1.5in]{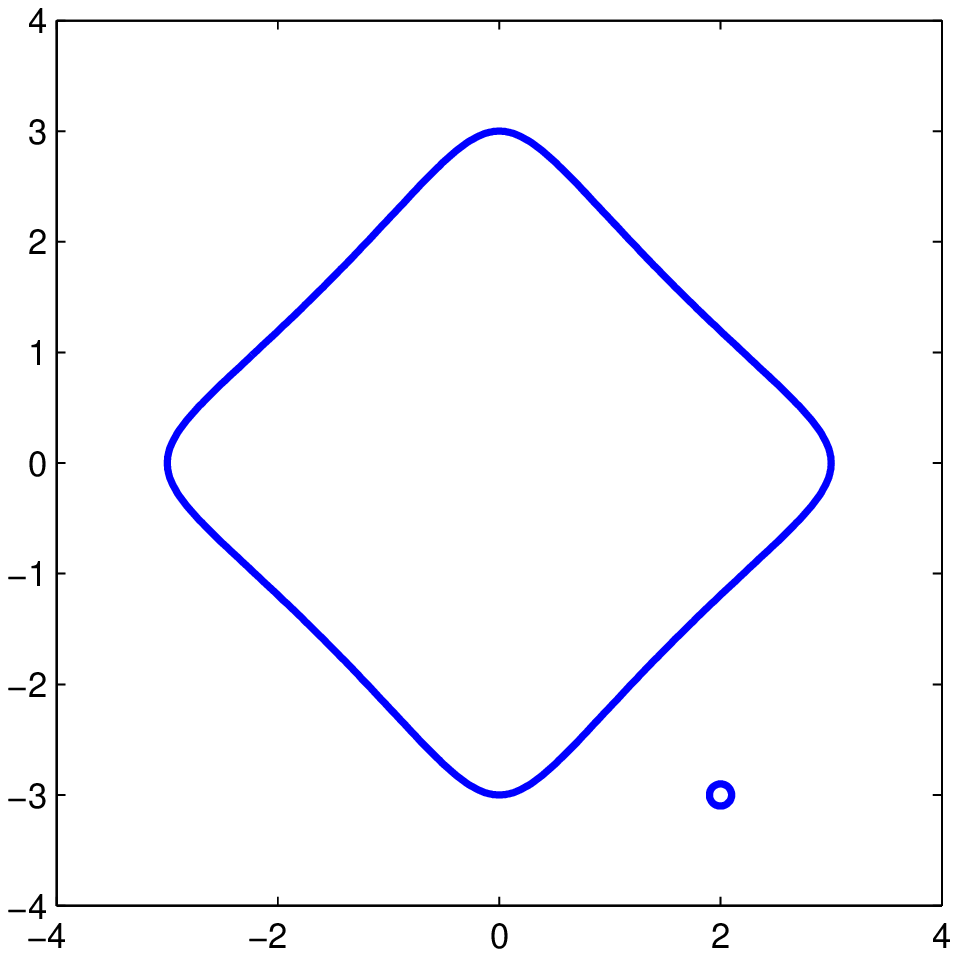}}
  \subfigure[\textbf{No noise, k=5}]{
    \includegraphics[width=1.5in]{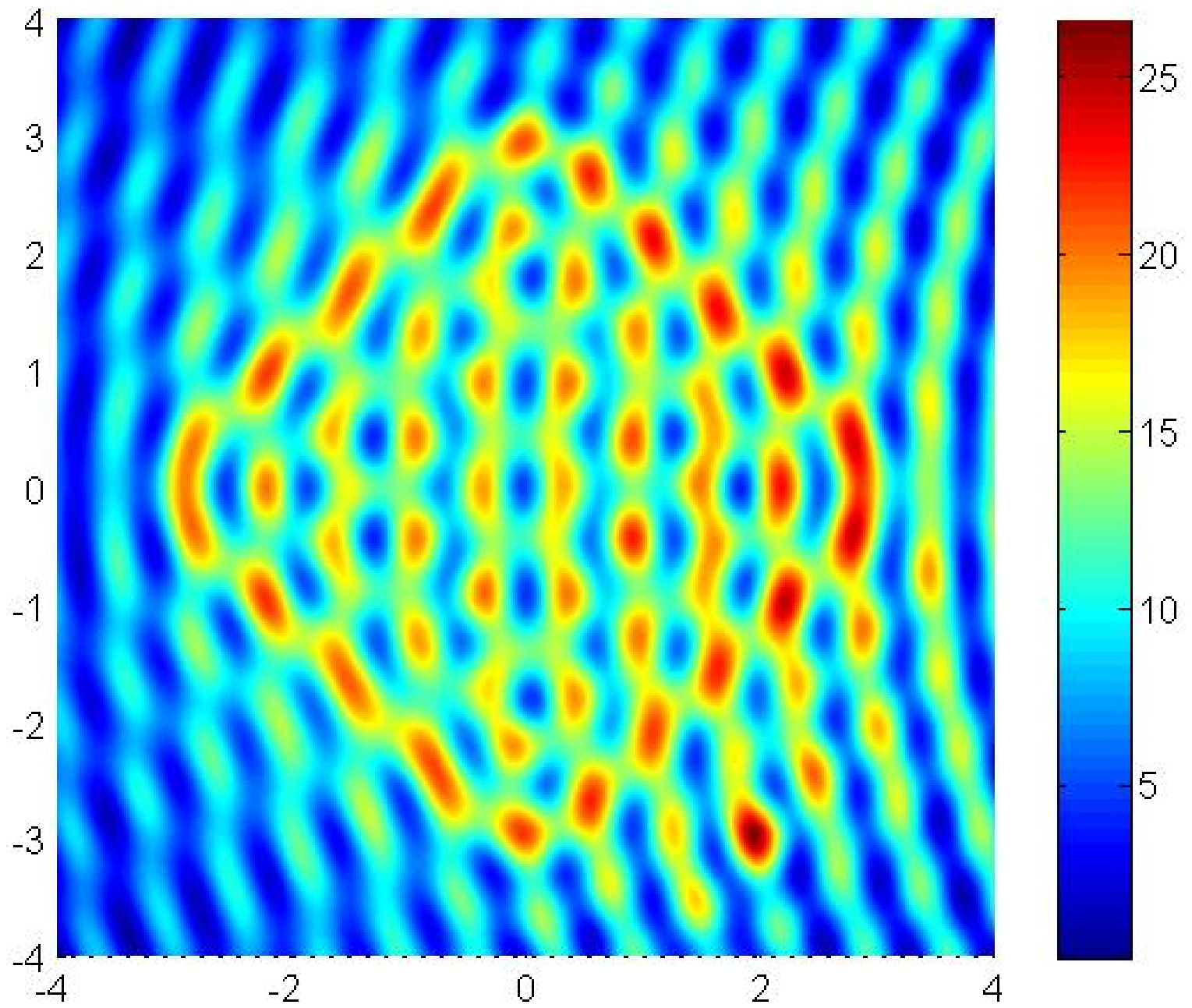}}
  \subfigure[\textbf{5\% noise, k=5}]{
    \includegraphics[width=1.5in]{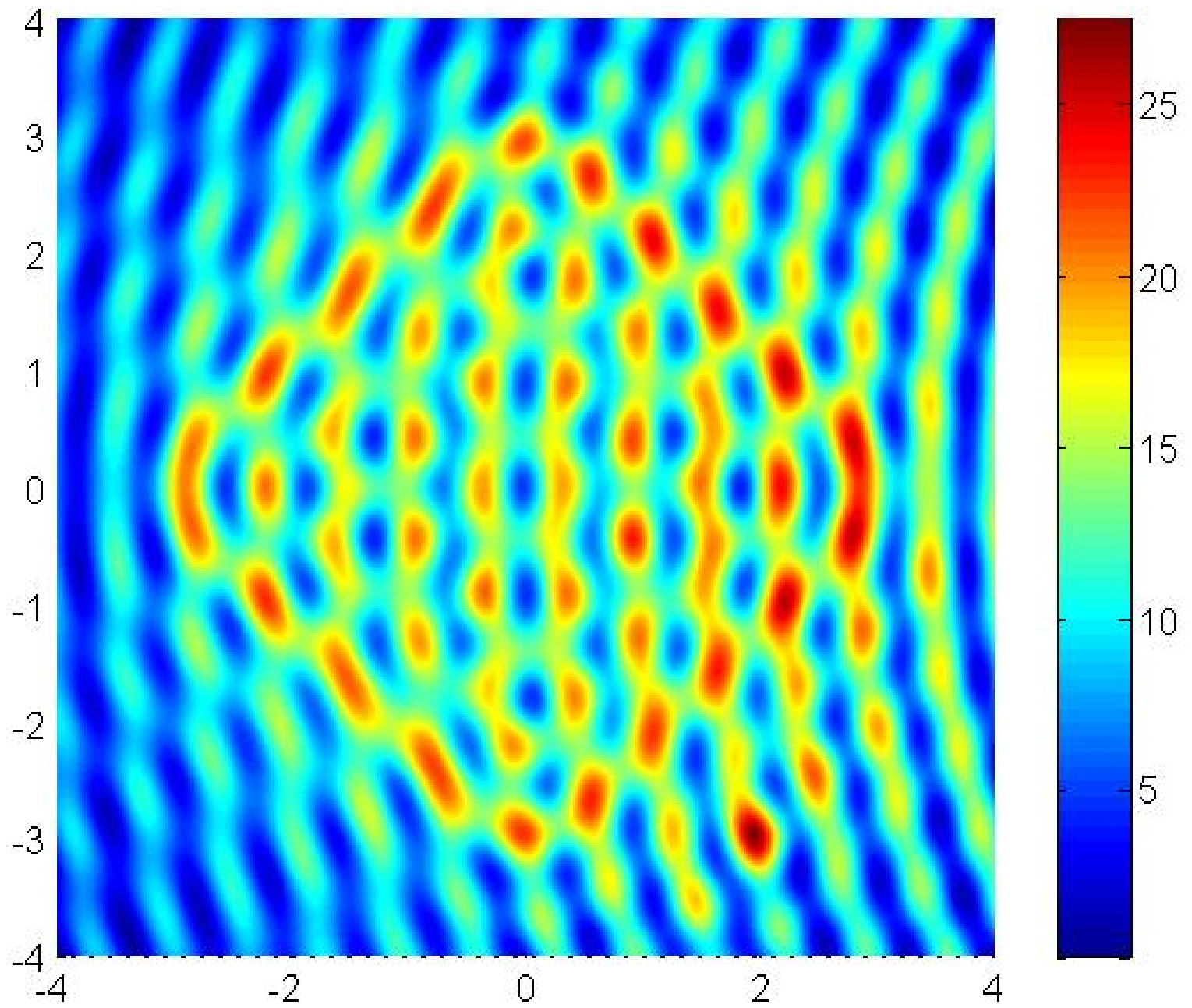}}
  \subfigure[\textbf{10\% noise, k=5}]{
    \includegraphics[width=1.5in]{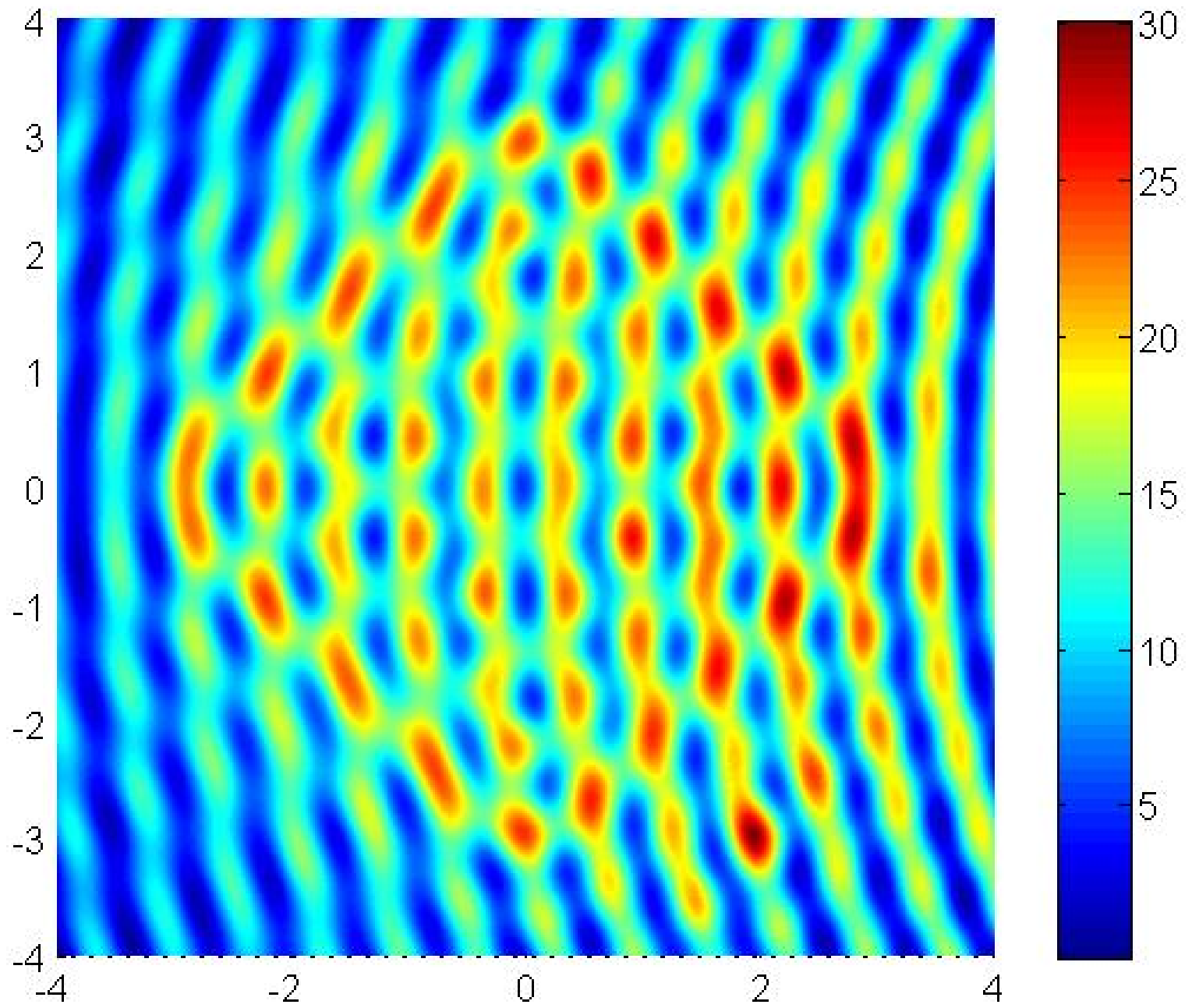}}
  \subfigure[\textbf{No noise, k=5}]{
    \includegraphics[width=1.5in]{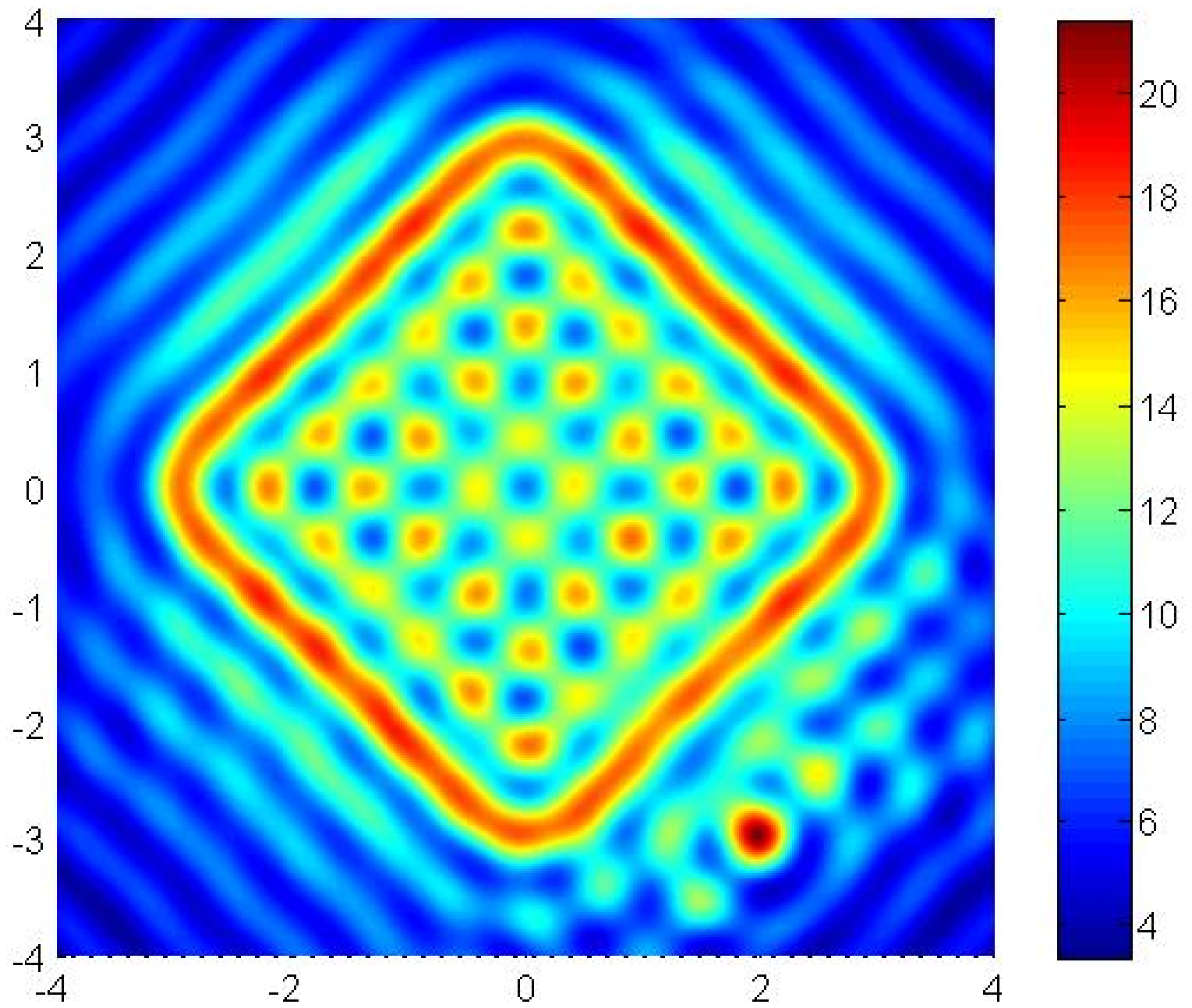}}
  \subfigure[\textbf{5\% noise, k=5}]{
    \includegraphics[width=1.5in]{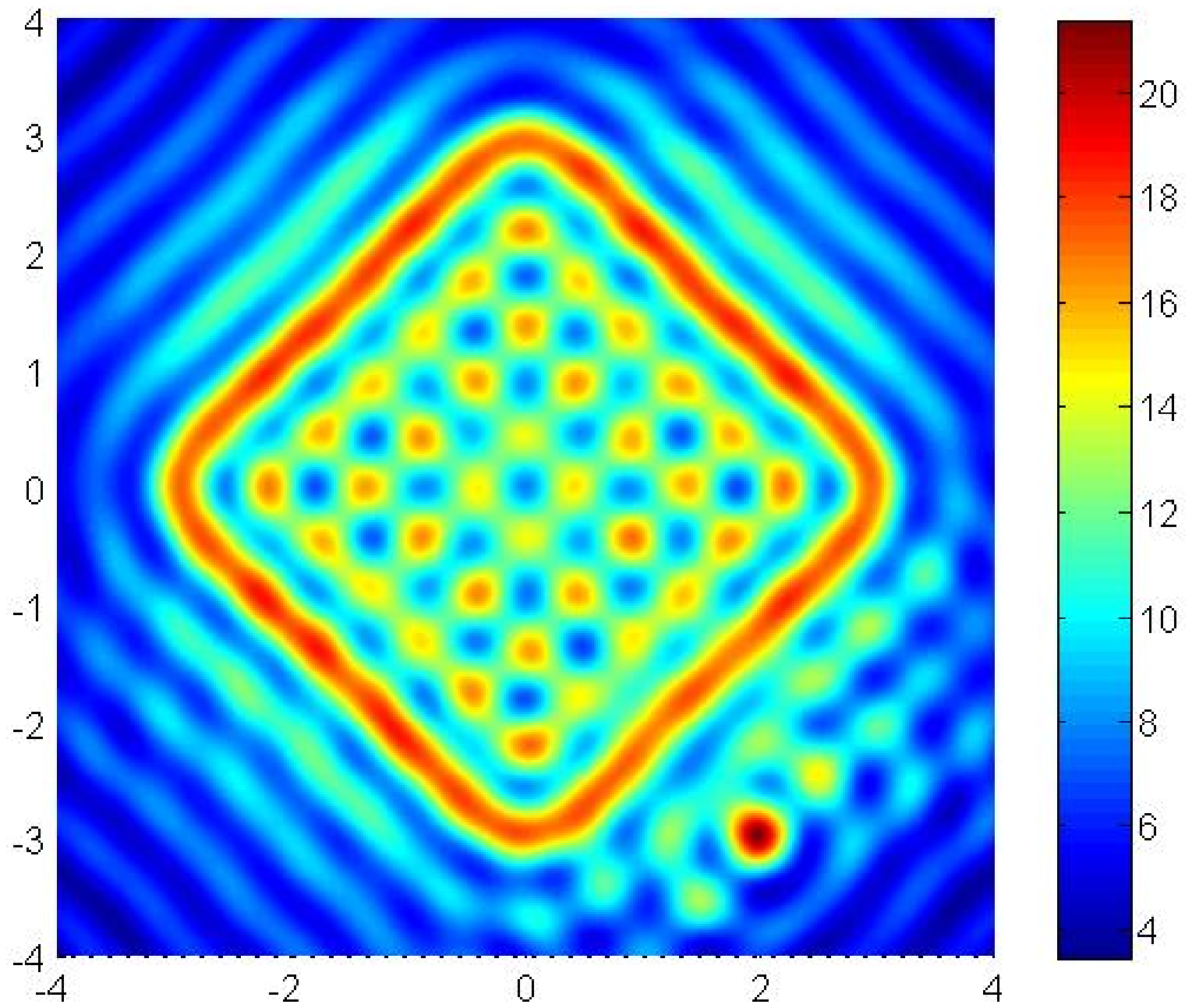}}
  \subfigure[\textbf{10\% noise, k=5}]{
    \includegraphics[width=1.5in]{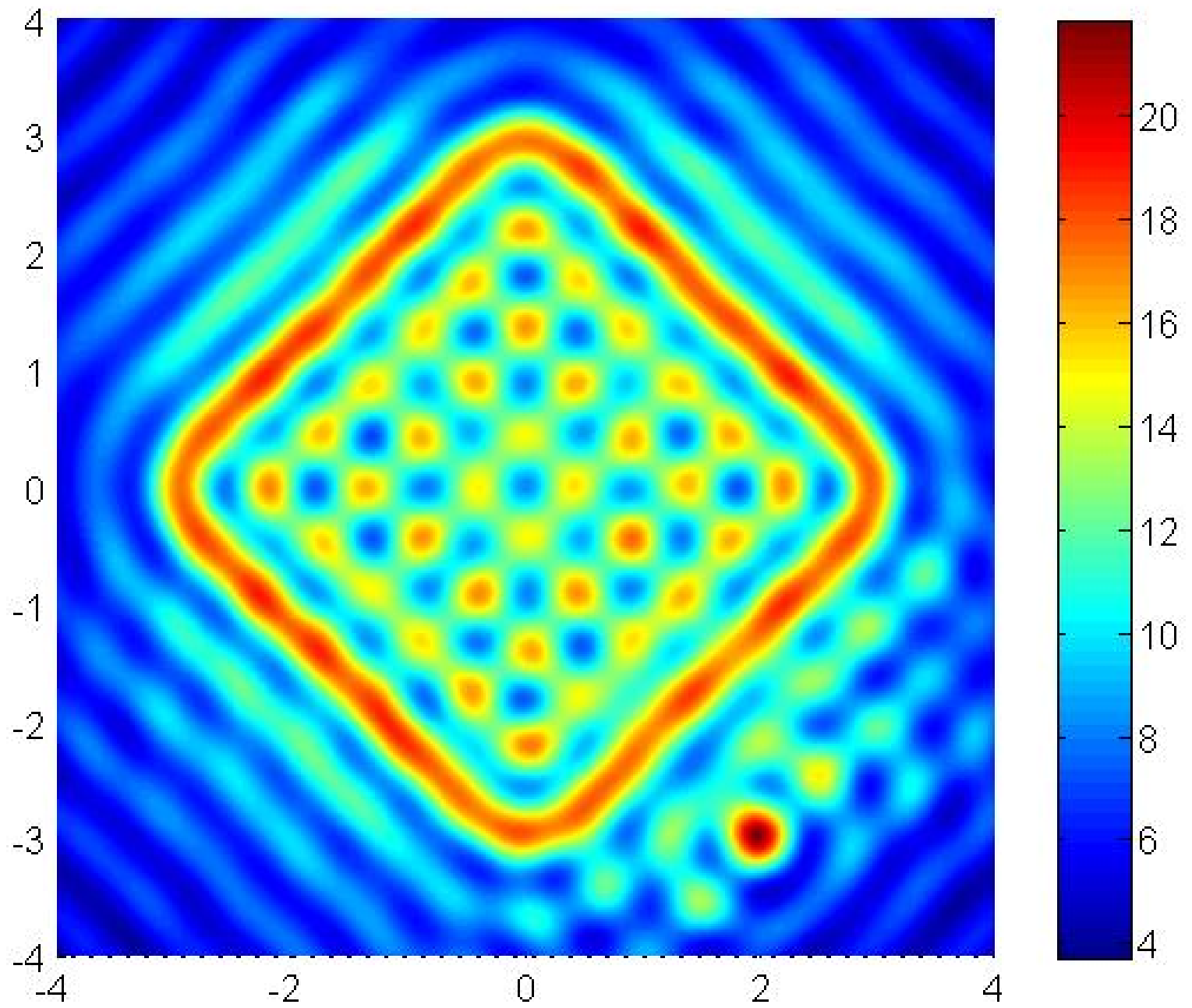}}
\caption{Imaging results of a circle-shaped, sound-soft obstacle with the radius $r=0.1$
and a rounded square-shaped, impedance obstacle with the impedance function $\rho(x)=5$,
obtained by Algorithm \ref{al1} with phaseless data (top row) and by the imaging algorithm
with $I^A_F(z)$ in \cite{P10} with full data (bottom row), respectively.
}\label{fig8}
\end{figure}

\begin{figure}[htbp]
  \centering
  \subfigure[\textbf{Exact curves}]{
    \includegraphics[width=1.5in]{pic/example4/case2/square_circle.eps}}
  \subfigure[\textbf{No noise, k=10}]{
    \includegraphics[width=1.5in]{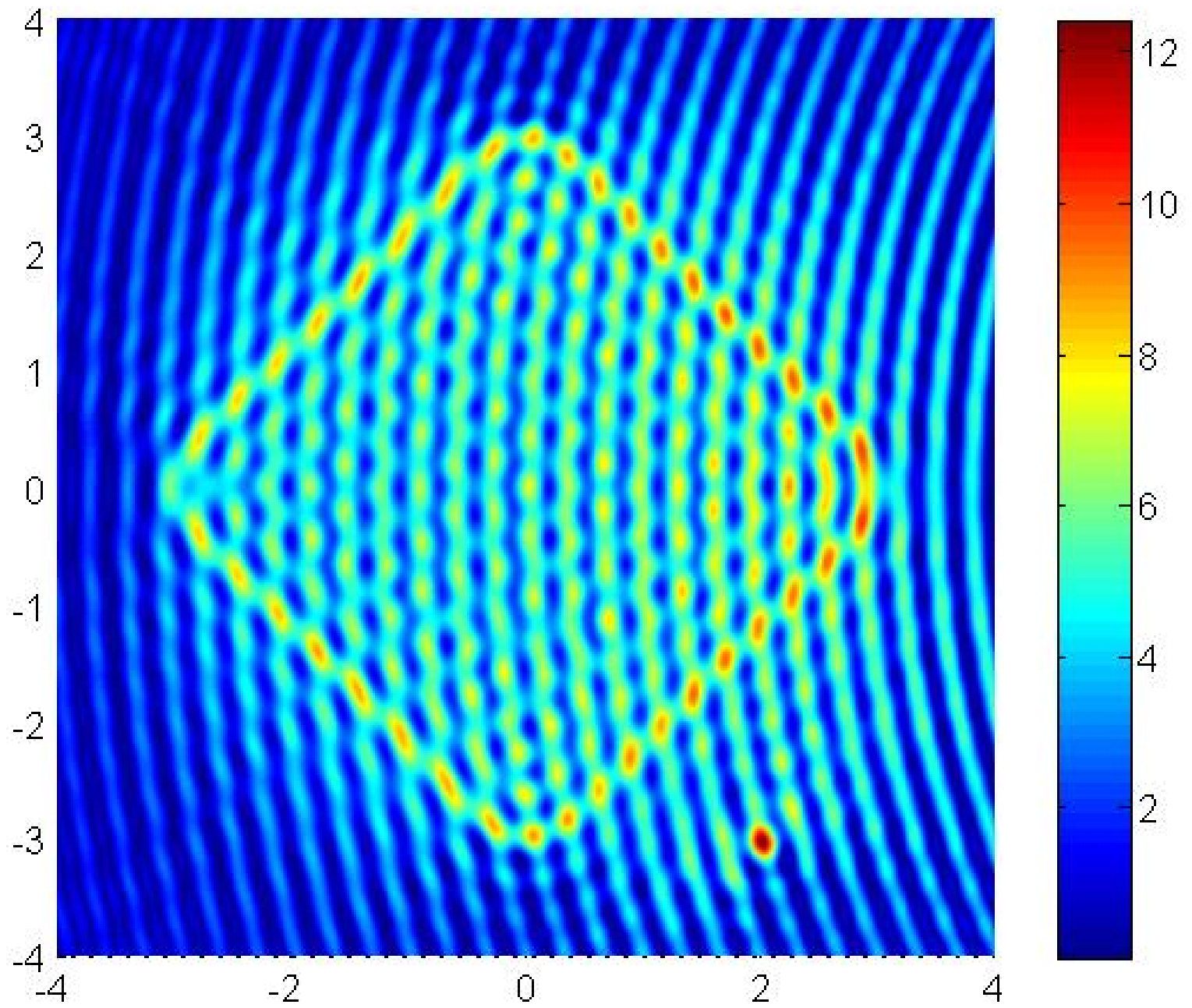}}
  \subfigure[\textbf{5\% noise, k=10}]{
    \includegraphics[width=1.5in]{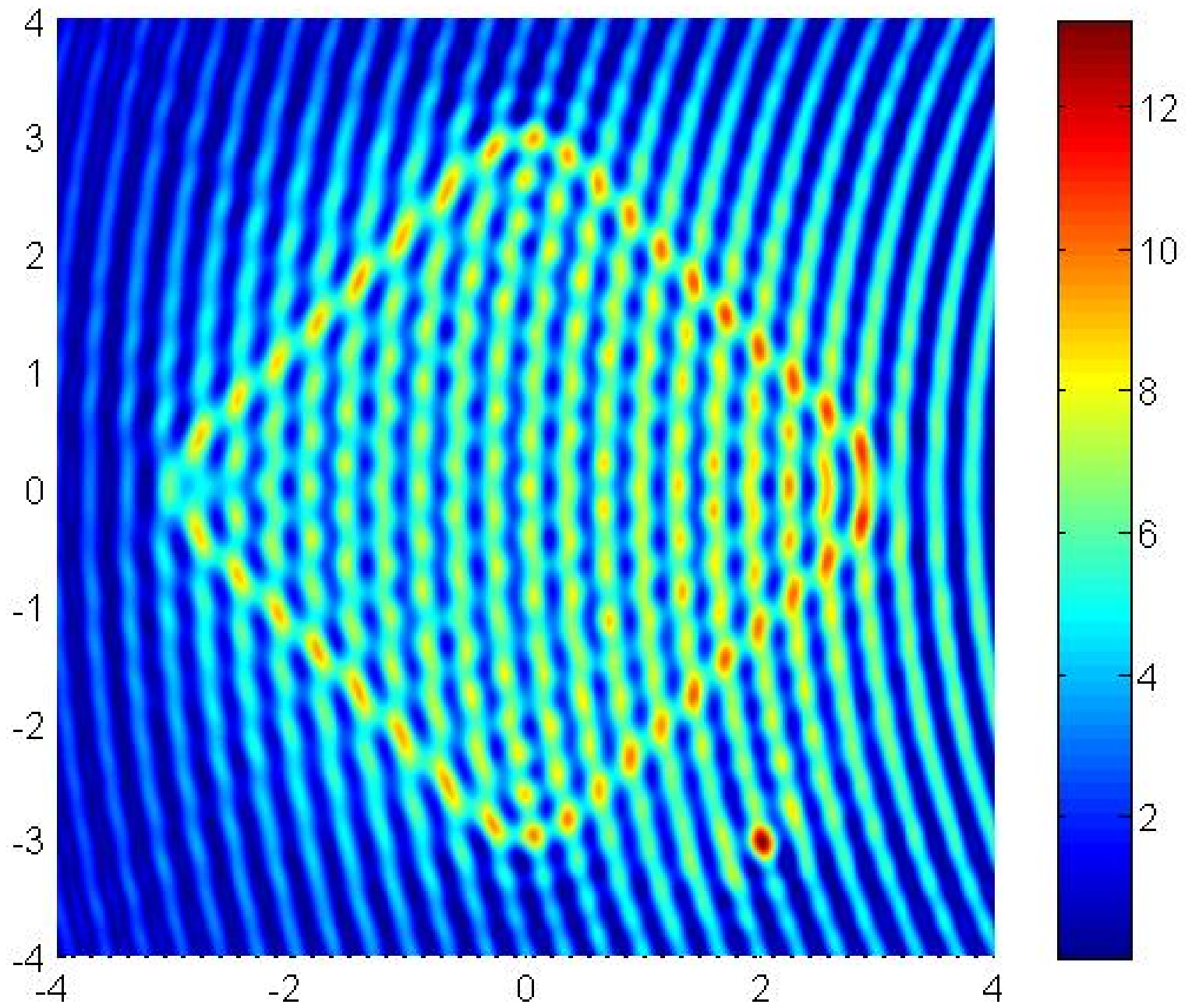}}
  \subfigure[\textbf{10\% noise, k=10}]{
    \includegraphics[width=1.5in]{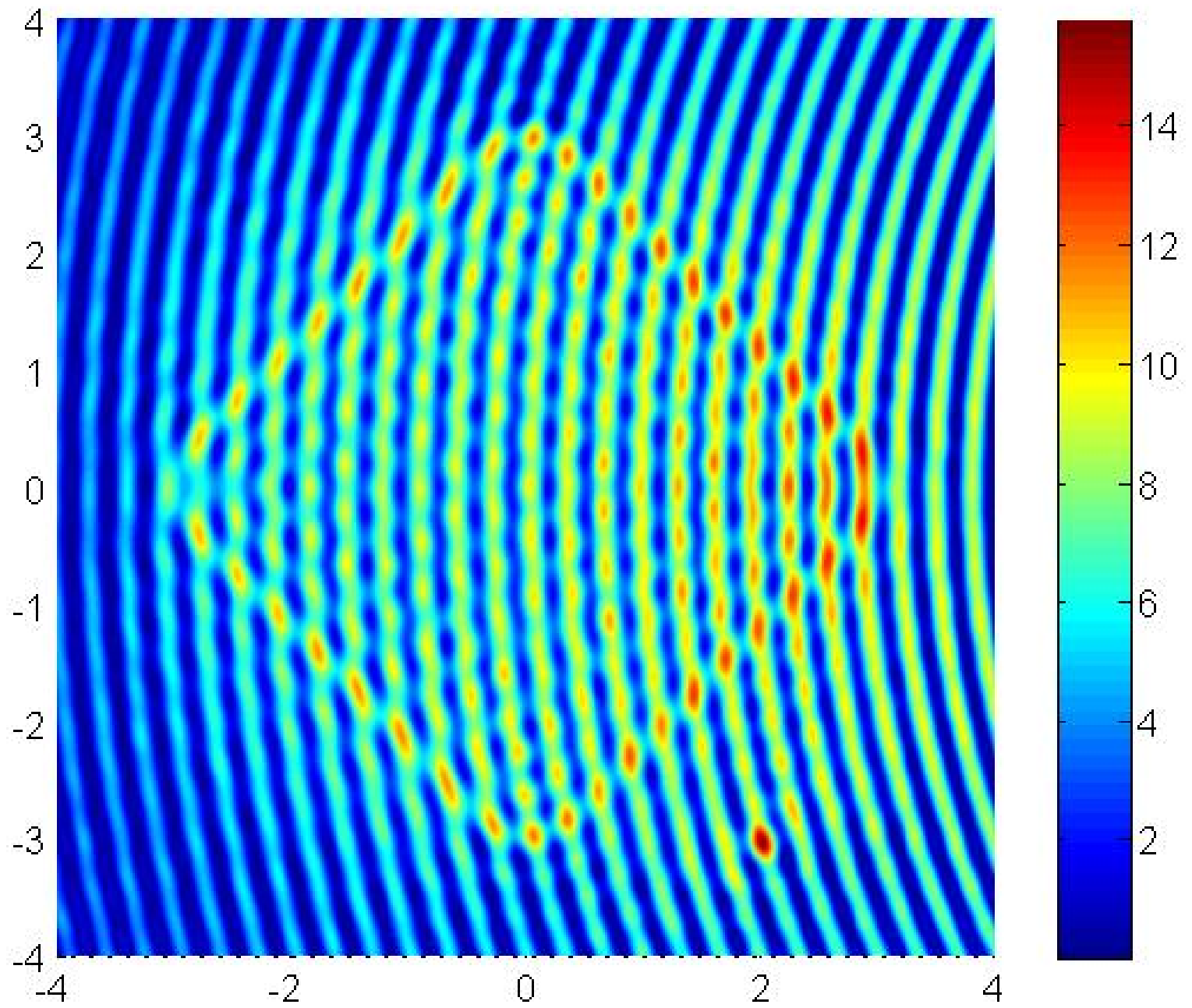}}
  \subfigure[\textbf{No noise, k=10}]{
    \includegraphics[width=1.5in]{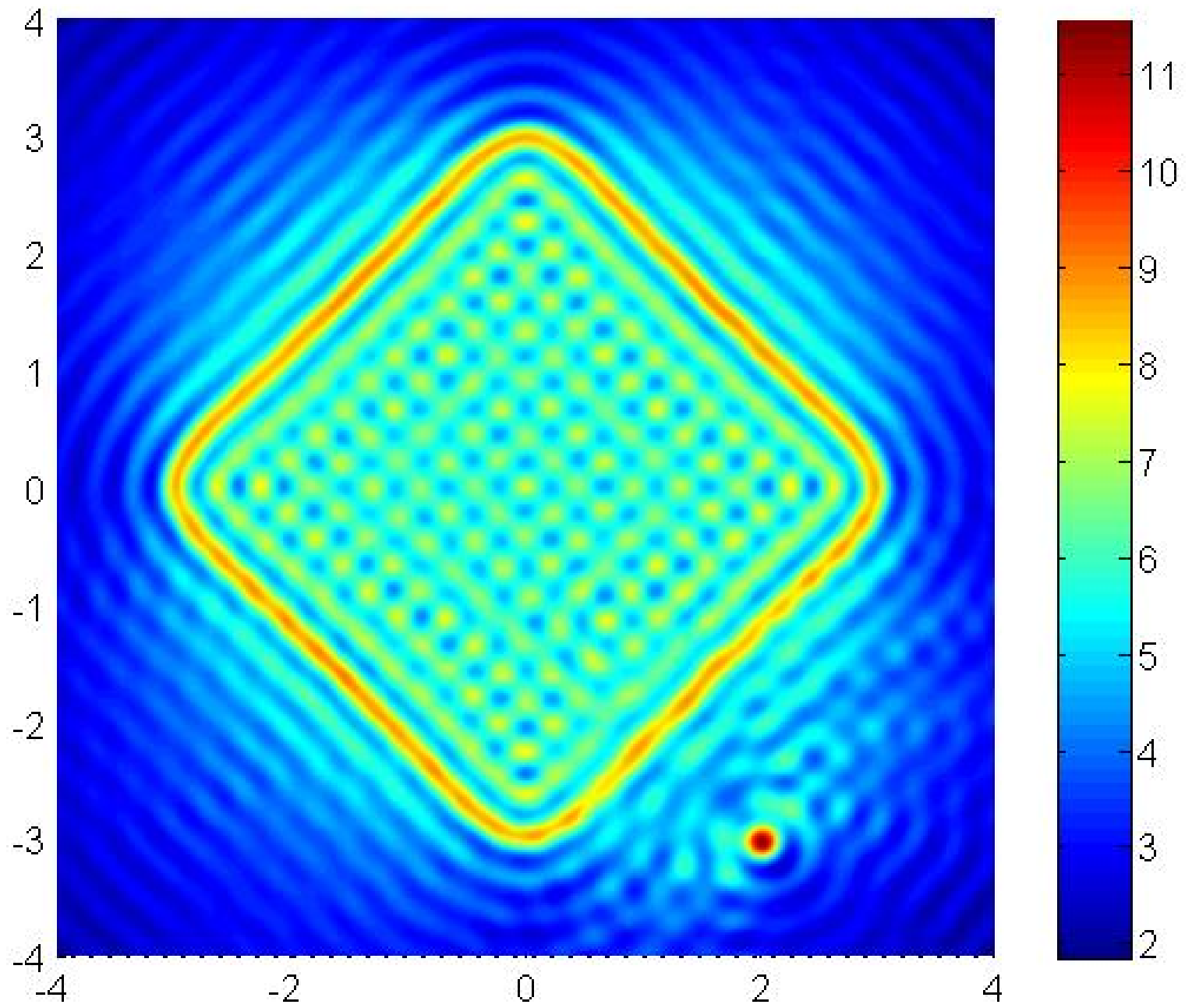}}
  \subfigure[\textbf{5\% noise, k=10}]{
    \includegraphics[width=1.5in]{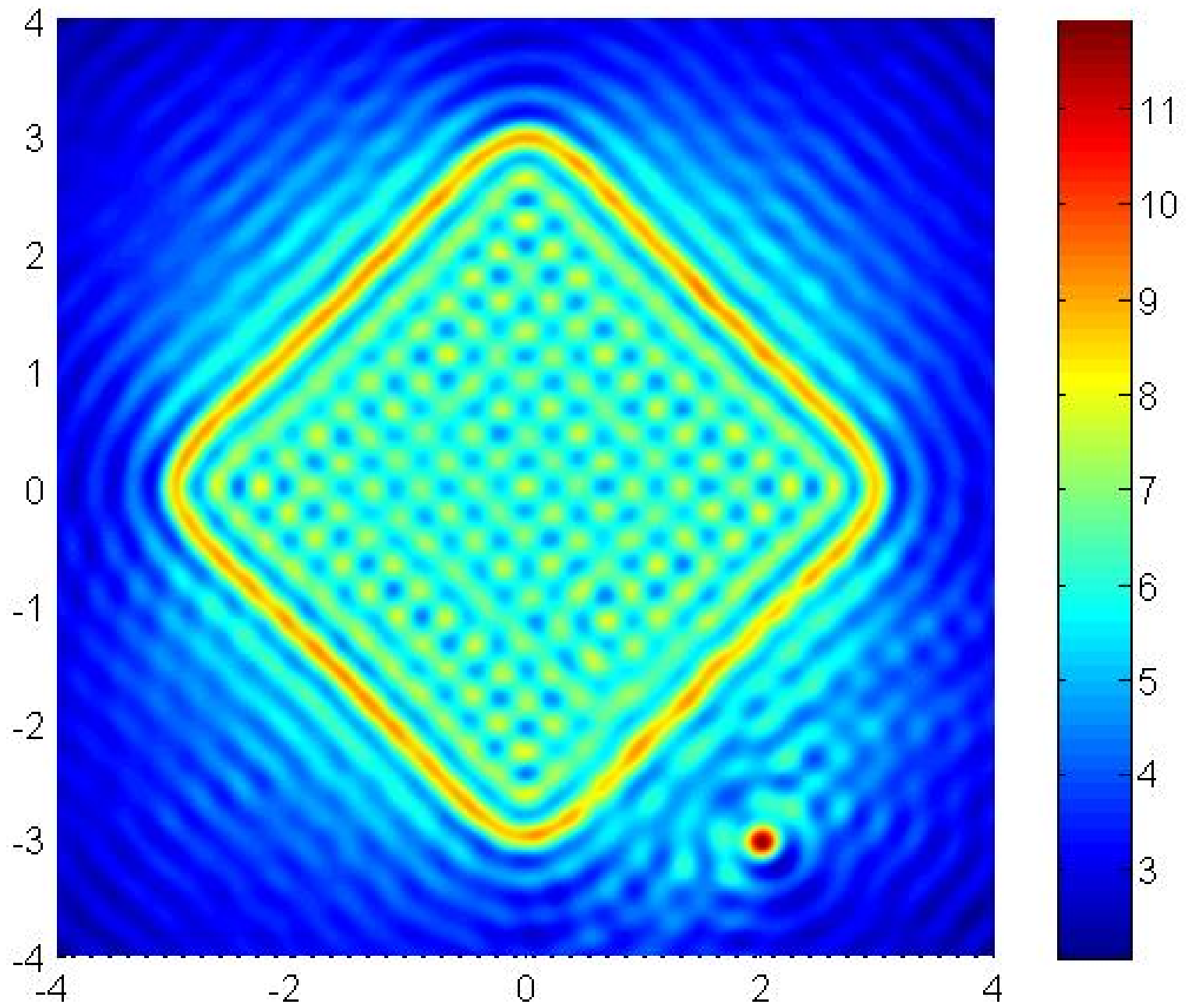}}
  \subfigure[\textbf{10\% noise, k=10}]{
    \includegraphics[width=1.5in]{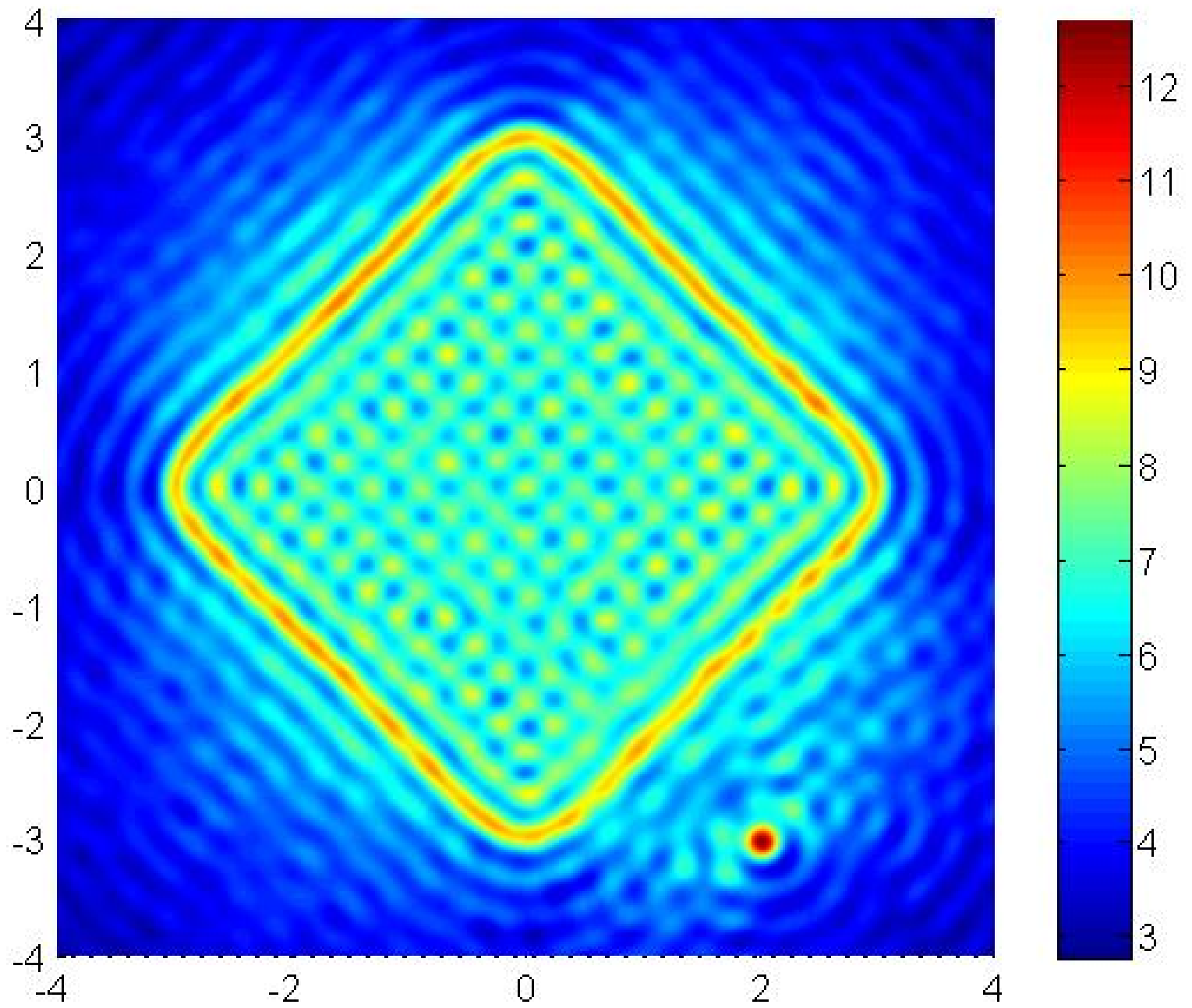}}
\caption{Imaging results of a circle-shaped, sound-soft obstacle with the radius $r=0.1$
and a rounded square-shaped, impedance obstacle with the impedance function $\rho(x)=5$,
obtained by Algorithm \ref{al1} with phaseless data (top row) and by the imaging algorithm
with $I^A_F(z)$ in \cite{P10} with full data (bottom row), respectively.
}\label{fig13}
\end{figure}

\begin{figure}[htbp]
  \centering
  \subfigure[\textbf{Exact curves}]{
    \includegraphics[width=1.5in]{pic/example4/case2/square_circle.eps}}
  \subfigure[\textbf{No noise, k=20}]{
    \includegraphics[width=1.5in]{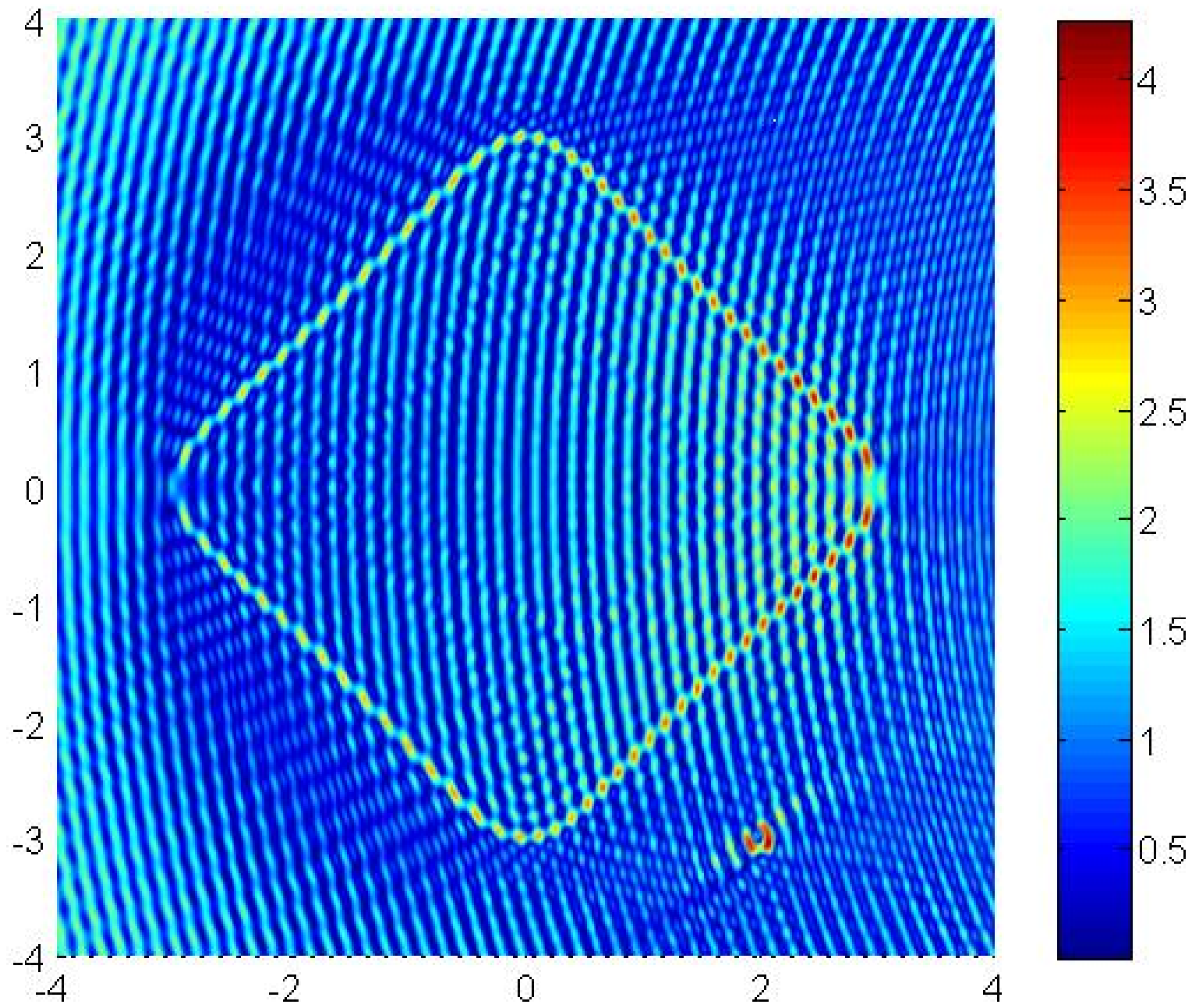}}
  \subfigure[\textbf{5\% noise, k=20}]{
    \includegraphics[width=1.5in]{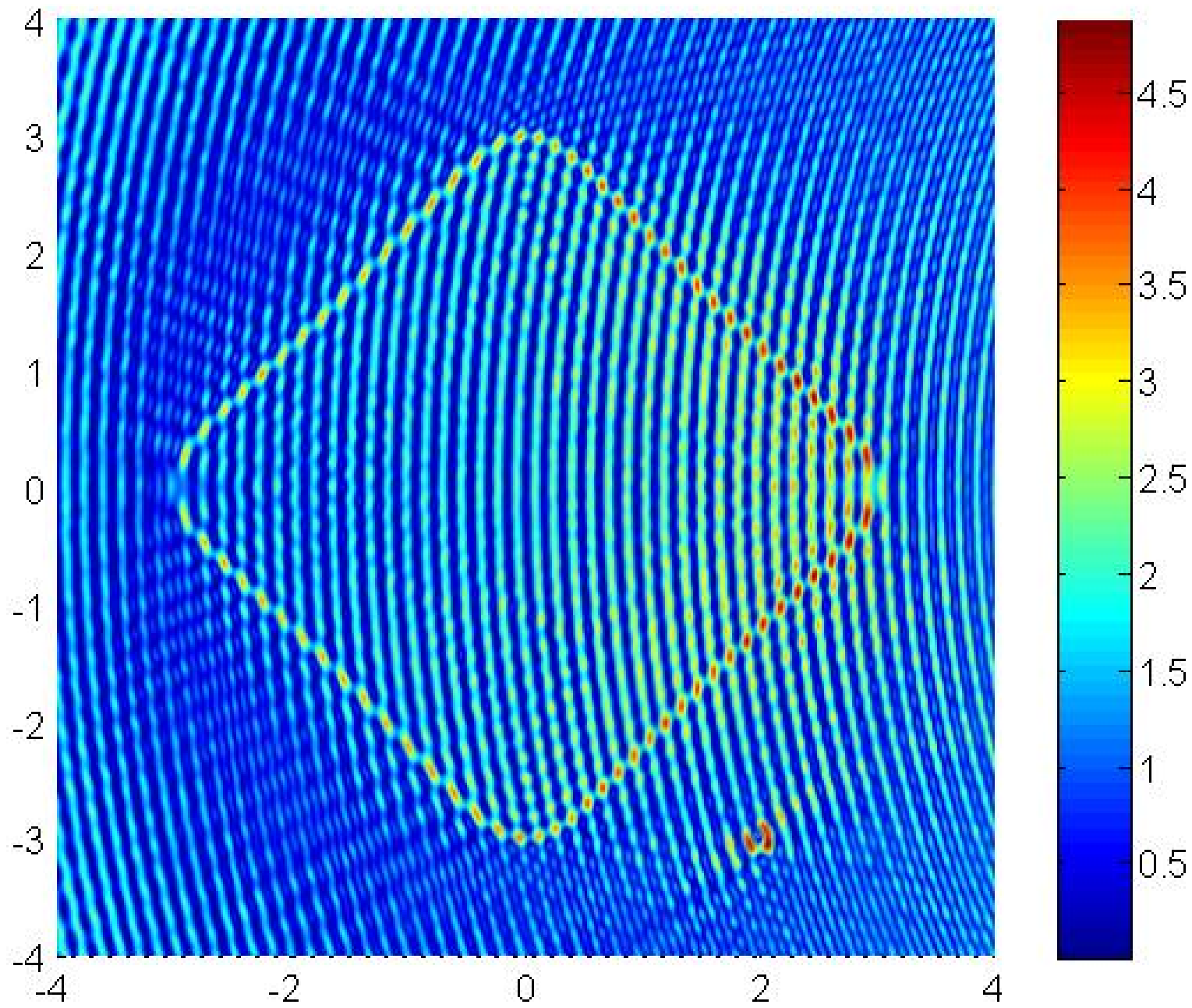}}
  \subfigure[\textbf{10\% noise, k=20}]{
    \includegraphics[width=1.5in]{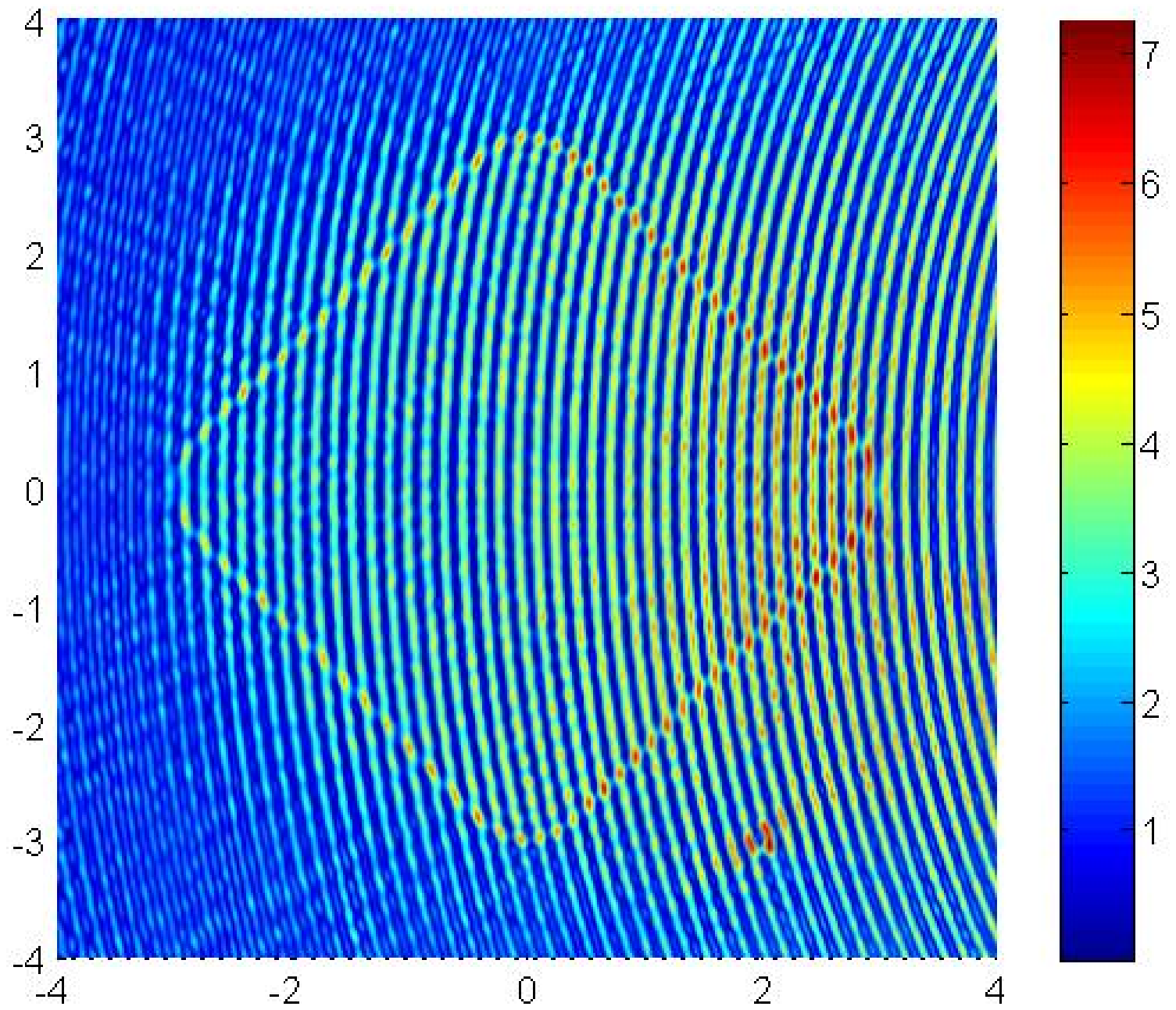}}
  \subfigure[\textbf{No noise, k=20}]{
    \includegraphics[width=1.5in]{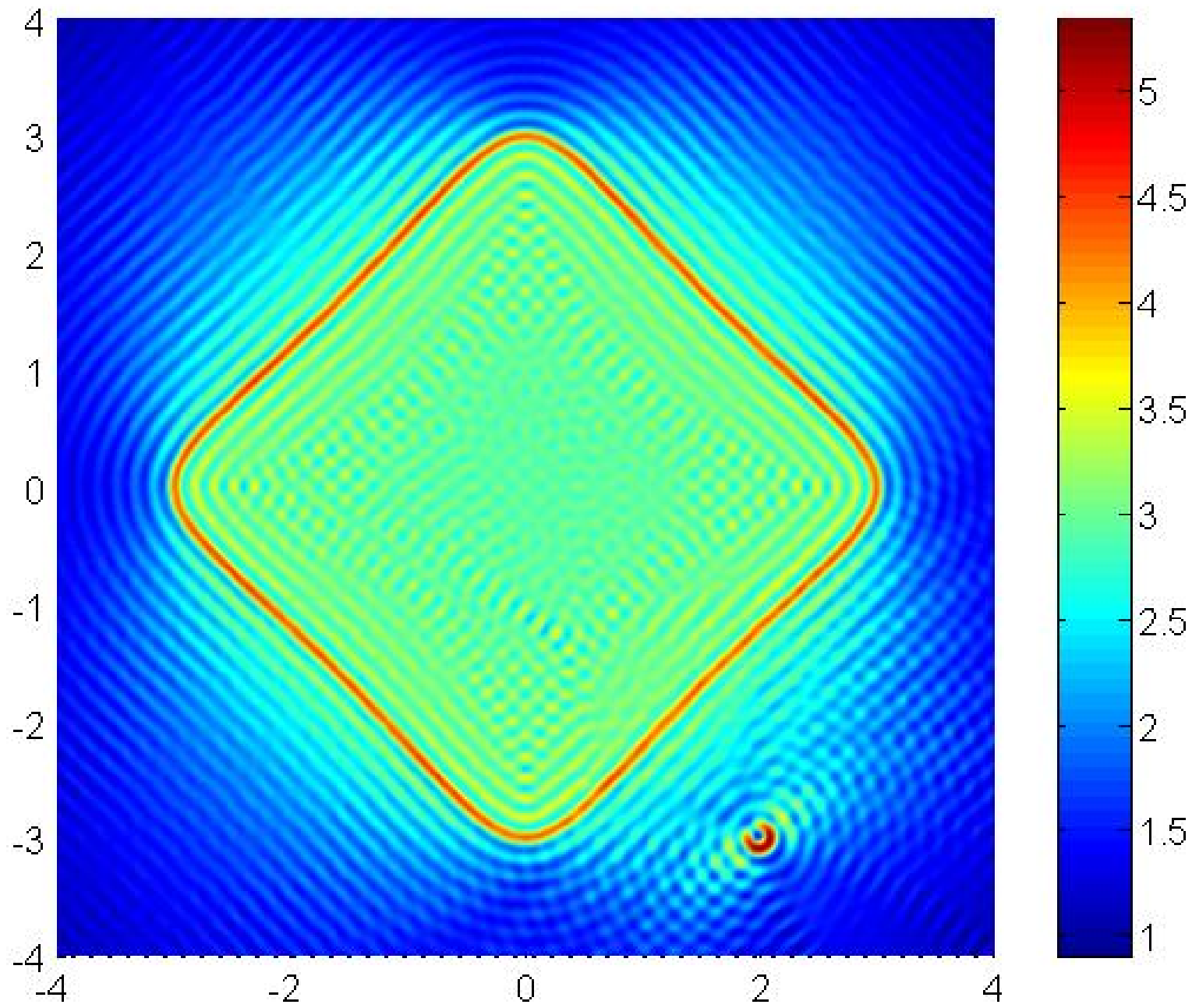}}
  \subfigure[\textbf{5\% noise, k=20}]{
    \includegraphics[width=1.5in]{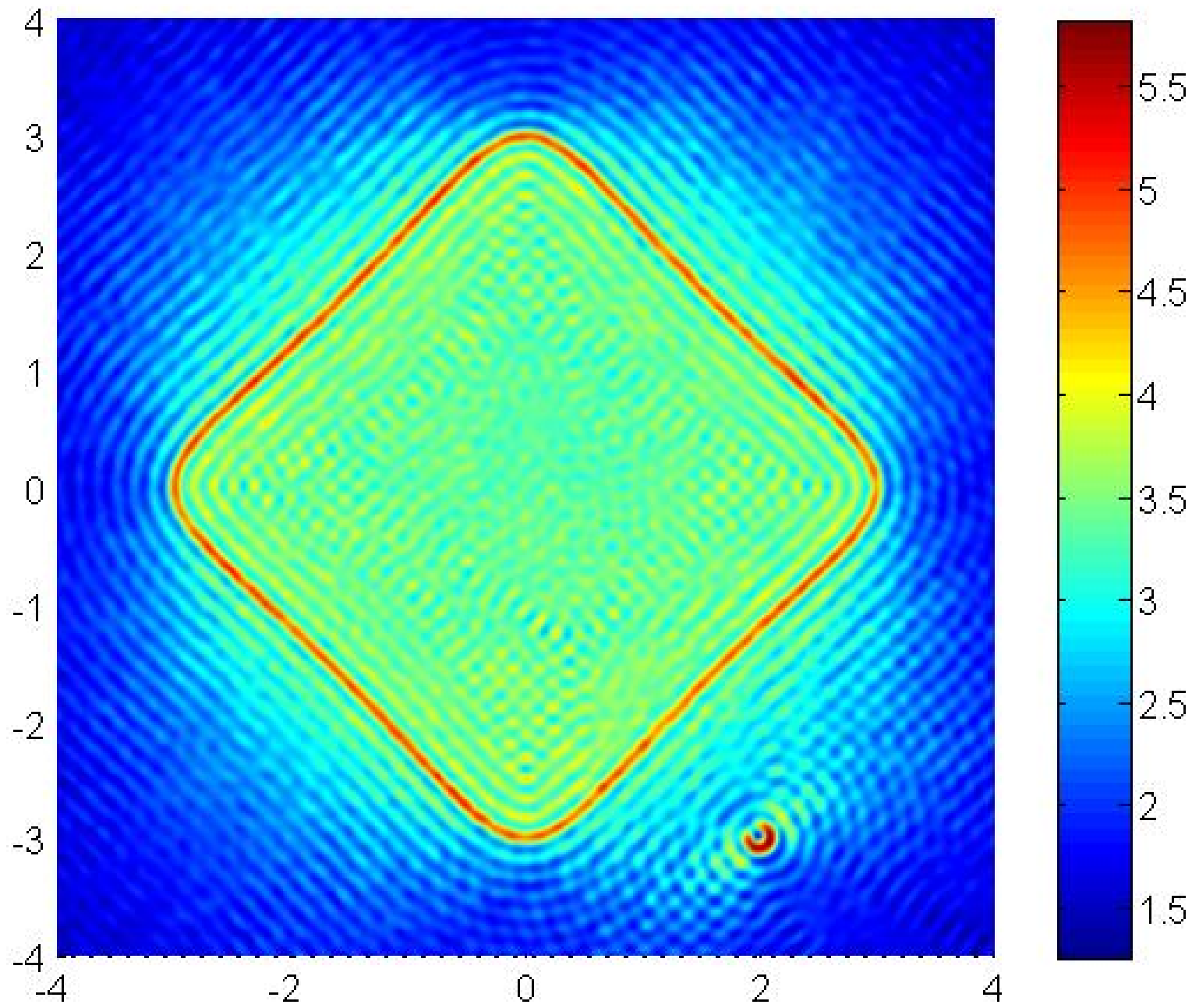}}
  \subfigure[\textbf{10\% noise, k=20}]{
    \includegraphics[width=1.5in]{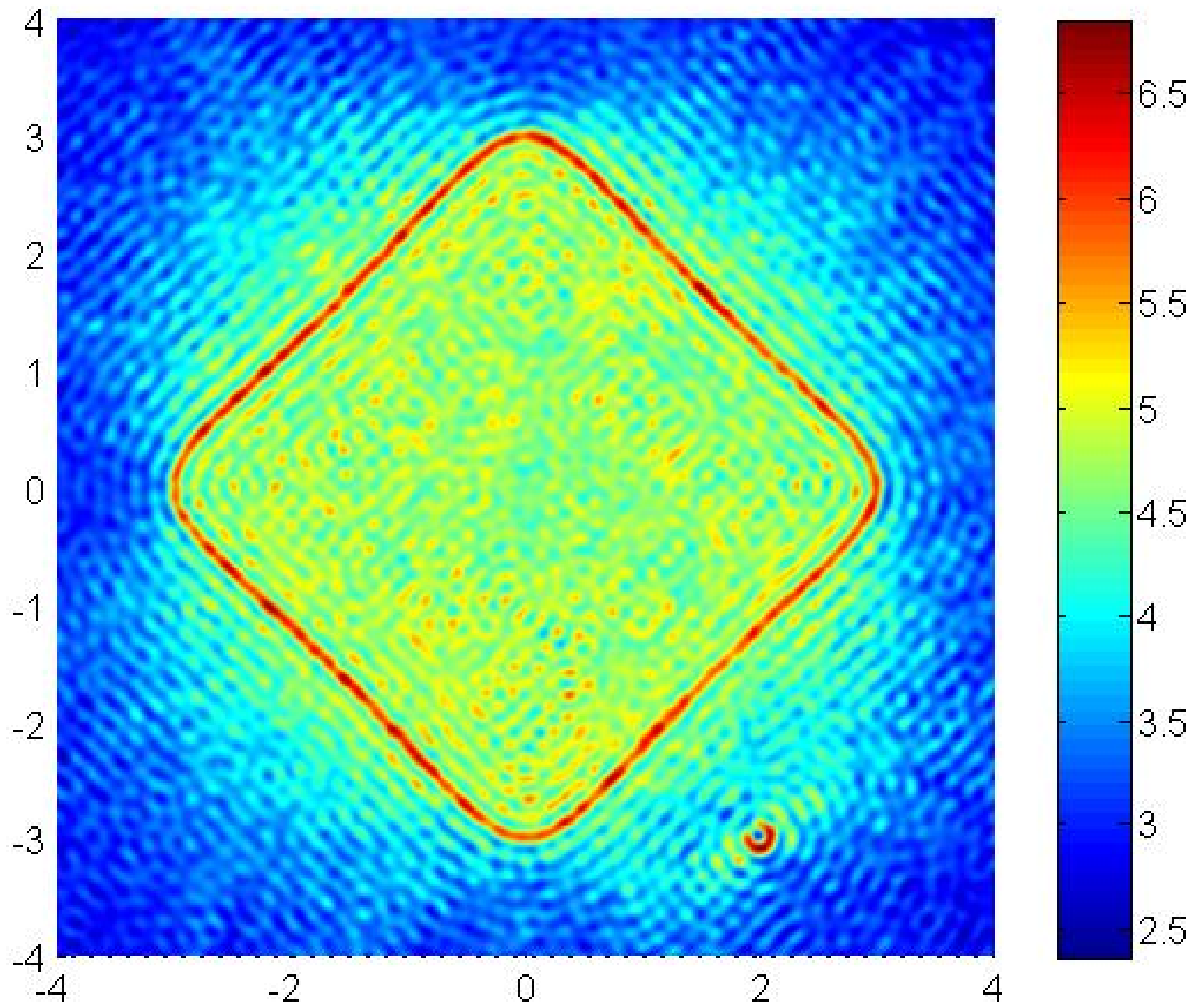}}
\caption{Imaging results of a circle-shaped, sound-soft obstacle with the radius $r=0.1$
and a rounded square-shaped, impedance obstacle with the impedance function $\rho(x)=5$,
obtained by Algorithm \ref{al1} with phaseless data (top row) and by the imaging algorithm
with $I^A_F(z)$ in \cite{P10} with full data (bottom row), respectively.
}\label{fig14}
\end{figure}

From the above examples and the other cases carried out but not presented here,
it can be seen that the proposed imaging method provides good and stable reconstructions of
impenetrable and penetrable obstacles. Further, the reconstruction results are robust to noise in data.

\section{Conclusion}\label{sec4}

In this paper, we proposed a direct imaging method to reconstruct both the location and shape
of a scattering obstacle from phaseless far-field data at a fixed frequency.
Our imaging method is motivated by our previous work \cite{ZZ17a}, where it was proved that
the translation invariance property of the phaseless far-field data can be broken by using infinitely
many sets of superpositions of two plane waves as the incident fields at a fixed frequency.
This suggests that both the location and shape of a scattering obstacle can be recovered
from such phaseless far-field data.
Recently it was proved in \cite{XZZ18} that a scattering obstacle can be uniquely determined by the
phaseless far-field patterns generated by infinitely many sets of superpositions of two plane waves
with different directions at a fixed frequency if the property of the obstacle is a priori known.
This paper gives a numerical realization of the above ideas and theoretical results.
Our imaging method only needs the calculation of the products of the measurement data with
two exponential functions at each sampling point and is thus fast and easy to implement.
Moreover, the proposed imaging method is very robust to noise in the measurement data and independent
of the physical properties of the obstacle.

\section*{Acknowledgements}

This work is partly supported by the NNSF of China grants 91630309, 11501558 and 11571355.

\end{document}